\documentclass{article}

\makeatletter
\renewcommand\paragraph{\@startsection{paragraph}{4}{\z@}%
                                    {1ex \@plus1ex \@minus.2ex}%
                                    {-1em}%
                                    {\normalfont\normalsize\bfseries}}
\makeatother
{\begin{list}               %    with flush left bullets.
    {$\bullet$ \hfill}{
        \setlength{\leftmargin}{\parindent}
        \setlength{\parsep}{0.04\baselineskip}
        \setlength{\itemsep}{0.5\parsep}
        \setlength{\labelwidth}{\leftmargin}
        \setlength{\labelsep}{0em}}
    }
{\end{list}}

\usepackage{amsmath}
\usepackage{amssymb}
\usepackage{amsthm}
\usepackage[margin=1in]{geometry}
\usepackage[sort&compress,numbers]{natbib} % for citet, citep etc
\usepackage{color}
\usepackage{enumitem}
\usepackage{hyperref}
 \hypersetup{
     colorlinks=true,
     linkcolor=blue,
     filecolor=blue,
     citecolor = blue,      
     urlcolor=black,
     }
\usepackage{bm}
\usepackage{mathrsfs}
\usepackage[scr=boondox]{mathalpha}

\usepackage{extarrows}
\usepackage{graphicx}
\usepackage{subfig}
\usepackage{mathrsfs}
\usepackage{cmll}
\usepackage{mathtools}

\usepackage[cmyk]{xcolor} 
\usepackage[title]{appendix}

\newtheorem{theorem}{Theorem}
\newtheorem{proposition}{Proposition}
\newtheorem{lemma}{Lemma}
\newtheorem{corollary}{Corollary}

\theoremstyle{definition}

\newtheorem{definition}{Definition}
\newtheorem{assumption}{Assumption}

\theoremstyle{remark}
\newtheorem{remark}{Remark}

\providecommand{\R}{\ensuremath{\mathbb{R}}}
\providecommand{\C}{\ensuremath{\mathbb{C}}}

\newcommand{\ortho}{\mathbb{O}} %% use for orthogonal group

\renewcommand{\vec}[1]{\ensuremath{\boldsymbol{#1}}}
\providecommand{\mat}[1]{\ensuremath{\boldsymbol{#1}}}

\providecommand{\mW}{\mat{W}}

\providecommand{\vr}{\vec{r}}

\providecommand{\vu}{\vec{u}}

\DeclareMathOperator{\diag}{diag}

\DeclareMathOperator{\op}{op}

\newcommand{\explain}[2]{\overset{\text{\tiny{#1}}}{#2}} %to provide explanations
\newcommand{\unif}[1]{\mathsf{Unif}(#1)} %% use for Haar measure/ Uniform measure
\newcommand{\E}{\mathbb{E}} %% for expectations
\renewcommand{\P}{\mathbb{P}} %% for probabilities
\newcommand{\Var}{\mathrm{Var}}
\newcommand{\Cov}{\mathrm{Cov}}
\newcommand{\gauss}[2]{\mathcal{N}\left( #1,#2 \right)} %% for Gaussian distribution
\newcommand{\pc}{\overset{\P}{\longrightarrow}} %conv. in P

\renewcommand{\tilde}{\widetilde}
\renewcommand{\hat}{\widehat}
\DeclareFontFamily{U}{mathx}{\hyphenchar\font45}
\DeclareFontShape{U}{mathx}{m}{n}{
      <5> <6> <7> <8> <9> <10>
      <10.95> <12> <14.4> <17.28> <20.74> <24.88>
      mathx10
      }{}
\DeclareSymbolFont{mathx}{U}{mathx}{m}{n}
\DeclareFontSubstitution{U}{mathx}{m}{n}
\DeclareMathAccent{\widecheck}{0}{mathx}{"71}
\DeclareMathAccent{\wideparen}{0}{mathx}{"75}

\renewcommand{\dim}{N} %size of the matrices
\DeclareMathOperator*{\plim}{\mathrm{plim}}

\newcommand{\BE}{\begin{equation}}
\newcommand{\EE}{\end{equation}}
\newcommand{\BS}{\begin{subequations}}
\newcommand{\ES}{\end{subequations}}
\newcommand{\UT}{\mathsf{T}}   %% upright T

\usepackage[english]{babel}
\usepackage[utf8]{inputenc}
\usepackage{algorithm}
\usepackage{algpseudocode}

\newcommand{\sa}[1]{\left\langle {#1} \right\rangle} %sample average
\newcommand{\sadiv}[1]{ \sa{\partial_{#1-1} \vu_{{#1}} }}

\newcommand{\para}[1]{\mathsf{#1}}
\newcommand{\paraB}[1]{\bm{\mathsf{#1}}}

\graphicspath{{figures/}}

\title{Unifying AMP Algorithms for Rotationally-Invariant Models}

\author{Songbin Liu $\quad \quad$ Junjie Ma\thanks{The authors are with the Institute of Computational Mathematics and Scientific/Engineering Computing, Academy of Mathematics and Systems Science, Chinese Academy of Sciences, China. \texttt{Emails:\{liusongbin,majunjie\}@lsec.cc.ac.cn}}}

\date{}

\begin{document}

\maketitle

 \begin{abstract}
This paper presents a unified framework for constructing Approximate Message Passing (AMP) algorithms for rotationally-invariant models. By employing a general iterative algorithm template and reducing it to long-memory Orthogonal AMP (OAMP), we systematically derive the correct Onsager terms of AMP algorithms. This approach allows us to rederive an AMP algorithm introduced by Fan and Opper et al., while shedding new light on the role of free cumulants of the spectral law. The free cumulants arise naturally from a recursive centering operation, potentially of independent interest beyond the scope of AMP. To illustrate the flexibility of our framework, we introduce two novel AMP variants and apply them to estimation in spiked models.
 \end{abstract}

\tableofcontents

\section{Introduction}

\subsection{AMP Algorithm}

Approximate message passing (AMP) \citep{donoho2009message,kabashima2003cdma,bayati2011dynamics,bolthausen2014iterative,feng2022unifying} has been applied to a broad range of high-dimensional estimation problems. Assume that $\bm{W}\in\mathbb{R}^{N\times N}$ is a symmetric random matrix sampled from the Gaussian orthogonal ensemble (GOE), namely, $\bm{W}=\frac{1}{\sqrt{2N}}\left(\bm{G}+\bm{G}^\UT\right)$, where $\bm{G}$ consists of independent identically distributed (i.i.d.) Gaussian entries. Starting from an initial guess $\bm{u}_0$, the AMP algorithm proceeds as follows:
\BS\label{Eqn:AMP}
\begin{align}
\vr_t &= \mW\vu_t - \sadiv{t}\cdot\vu_{t-1}, \\
\vu_{t+1} &= \eta_{t+1} (\vr_t),
\end{align}
\ES
where $\eta_t:\mathbb{R}\mapsto\mathbb{R}$ is Lipschitz continuous and applies to elements of the input vector $\bm{r}_t$ separately and $\sadiv{t}:=\sum_{i-1}^N\eta'_{t}\left(r_{t-1}[i]\right)/N$ is the divergence of the nonlinearity $\eta_t$. A distinguishing characteristic of the AMP algorithm is that the empirical law of the iterates $\bm{r}_t$ converges to a Gaussian distribution in the high-dimensional limit:
\[
\bm{r}_t\to \mathcal{N}(0,\tau_t^2),\quad\forall t\ge1.
\]
Moreover, the variance of the limiting Gaussian distribution can be tracked by a simple recursion known as \textit{state evolution}:
\[
\tau_t^2 = \mathbb{E}_{Z\sim\mathcal{N}(0,1)}\left[\eta_{t-1}^2\left(\tau_{t-1}Z\right)\right].
\]
This property enables precise characterization of the performance of AMP, and has shed important insight to various high-dimensional estimation and optimization problems \cite{Barbier2019,sur2019modern,feng2022unifying,reeves2019replica,yedla2014simple,donoho2013information,montanari2024equivalence,bayati2011lasso,donoho2011noise,el2021optimization,montanari2021estimation,bu2020algorithmic,wang2020bridge}.

A major limitation of the Gaussian AMP algorithm (referred to as \textit{Gaussian AMP} hereafter) is that its theoretical guarantee relies on the fact that $\bm{W}$ is GOE. Extending AMP beyond this setup is the subject of extensive studies in recent years \citep{ma2014turbo,opper2016theory,ma2017orthogonal,takeuchi2019rigorous,rangan2019vector,liu2022memory,takeuchi2021bayes,fan2022approximate,zhong2024approximate,dudeja2024optimality,cakmak2014s,takeuchi2023orthogonal,li2023spectrum,dudeja2022universality,dudeja2022spectral,dudeja2023universality,wang2024universality,rossetti2024linear,cademartori2024non,takeuchi2019unified,venkataramanan2022estimation,maillard2019high,fletcher2018inference,rossetti2024linear}. In particular, building on the earlier work of \citet{opper2016theory}, \citet{fan2022approximate} proposed the following AMP algorithm for \textit{rotationally-invariant} models:
\BS\label{eq:ZFAMP}
\begin{align}
\vr_t &= \mW\vu_t - \left(\para{b}_{t,1}\bm{u}_1+\para{b}_{t,2}\bm{u}_2+\cdots+\para{b}_{t,t}\bm{u}_t\right), \\
\vu_{t+1} &= \eta_{t+1} (\vr_t).
\end{align}
\ES
where the coefficients $(\para{b}_{t,i})_{1\le t,1\le i\le t}$ depend on the \textit{free cumulants} of the spectral law of $\bm{W}$. This AMP algorithm, which we refer to as rotationally-invariant AMP (RI-AMP), was first proposed in \citep{opper2016theory} using the non-rigorous dynamical functional theory in the context of Ising models; see \cite{opper2016theory,mimura2014generating} and the references therein for more information. Similar to the original AMP, the iterate $\bm{r}_t$ in RI-AMP is asymptotically Gaussian distributed. (Note that unlike the Gaussian AMP in \eqref{Eqn:AMP}, the Onsager term in the above RI-AMP algorithm involves all iterates $(\bm{u}_i)_{i\le t}$, even when the nonlinearity $\eta_t$ only depends on $\bm{r}_t$.) Moreover, the variance of the limiting Gaussian distribution can be tracked by a state evolution, which was rigorously proved in \citep{fan2022approximate} using a conditioning technique pioneered in \cite{bolthausen2014iterative,bayati2011dynamics} for Gaussian models and generalized in \cite{rangan2019vector,takeuchi2019rigorous} for rotationally-invariant models. 

The above RI-AMP algorithm has been generalized to various other setups, including the rectangular matrix setting \citep{fan2022approximate}, generalized linear models (GLM) \citep{venkataramanan2022estimation}, and multi-layer GLM \citep{xu2023approximate}.

\subsection{Orthogonal AMP Algorithm}
The orthogonal AMP (OAMP) \cite{ma2017orthogonal} or vector AMP (VAMP) \cite{rangan2019vector} algorithms, which can be derived based on expectation propagation \cite{minka2013expectation,opper2005expectation}, represent another line of work that generalizes AMP to rotationally-invariant models. This algorithm has been applied in various high-dimensional estimation problems \cite{ma2021spectral,ma2023towards,cheng2020orthogonal,pandit2020inference}. OAMP relies on the use of \textit{trace-free} matrix denoisers and \textit{divergence-free} iterate denoisers:
\BS\label{Eqn:LM_OAMP}
\begin{align}
\bm{x}_t &= \left(f_t(\bm{W})-\frac{\text{tr}\left(f_t(\bm{W})\right)}{N}\cdot\bm{I}_N\right)\bar{\bm{x}}_{t},\label{Eqn:LM_OAMP_a}\\
\bar{\bm{x}}_{t+1} &= g_{t+1} (\bm{x}_1,\ldots,\bm{x}_t)-\sum_{i=1}^t \langle\partial_i  g_{t+1} (\bm{x}_1,\ldots,\bm{x}_t)\rangle\cdot \bm{x}_i, \label{Eqn:LM_OAMP_b}
\end{align}
\ES
where $f_t:\mathbb{R}\mapsto\mathbb{R}$ is applied to the matrix $\bm{W}$ in the following sense: let $\bm{W}=\bm{O}\mathrm{diag}\left(\lambda_1,\ldots,\lambda_N\right)\bm{O}^\UT$ be the eigenvalue decomposition of $\bm{W}$, then $f_t(\bm{W}):=\bm{O}\mathrm{diag}\left(f_t(\lambda_1),f_t(\lambda_2),\ldots,f_t(\lambda_N)\right)\bm{O}^\UT$; 
and $\partial_s g_t$ denotes the partial derivative with respect to the $s$th argument of $g_t$. A precise definition of the OAMP algorithm can be found in Section \ref{Sec:OAMP_SE}. The original OAMP algorithm \cite{ma2017orthogonal} assumes $g_t$ to be a univariate function of $\bm{r}_t$, and the extension to the multivariate case is due to \citet{takeuchi2019unified}. Note that the long-memory OAMP algorithm originally proposed in \citep{takeuchi2019unified} is more general than \eqref{Eqn:LM_OAMP}. The above presentation of the OAMP algorithm follows \citet{dudeja2022spectral} and is general enough for the purpose of the present paper. Similar to RI-AMP, the iterate $\bm{x}_t$ in OAMP converges to Gaussian $\bm{x}_t\to \mathcal{N}(0,\tau_t^2)$, $\forall t\ge1$, thanks to the use of trace-free matrix denoiser and divergence-free iterate denoisers. 

The OAMP iteration in \eqref{Eqn:LM_OAMP} has been used to construct various AMP algorithms, such as convolutional AMP (CAMP) \citep{takeuchi2021bayes} and Memory AMP (MAMP) \citep{liu2022memory}. In a different direction, \citet{takeuchi2019unified} showed that the original Gaussian AMP algorithm can be mapped to some OAMP algorithm, in the context of compressed sensing. The idea of reducing a general iterative algorithm to certain OAMP algorithms has been explored in \citet{dudeja2022spectral} for proving universality (with respect to the sensing matrix) of the performance of convex regularized least squares estimators, and in \citet{dudeja2024optimality} for analyzing the statistically optimal performance achievable within a broad class of iterative algorithms for spiked models.

\subsection{Contributions}

Compared with the Gaussian AMP counterpart, the derivations of the rotationally-invariant AMP in \cite{opper2016theory,fan2022approximate} are more complicated. In this paper, we aim to provide a unified, and arguably more elementary, approach to constructing AMP algorithms (more specifically, their Onsager terms) for rotationally-invariant models. Our approach can be used to re-derive existing AMP algorithms and devise new variants quite easily. The main contributions of this paper include:

\begin{itemize}
\item We introduce a unified framework for constructing AMP algorithms for rotationally-invariant models. Our approach is through reduction to long-memory OAMP algorithms, based on \textit{orthogonal decomposition} of the iterates and recursive unfolding the algorithm. The orthogonal decomposition idea was first introduced in \cite{dudeja2022spectral} but for a different purpose. We use this technique to re-derive the RI-AMP algorithm. Our results provide an alternative and more interpretable state evolution of RI-AMP.

\item In our derivation of RI-AMP, the de-biasing coefficients are naturally represented as the normalized traces of certain polynomials of $\bm{W}$. We prove that these coefficients are related to the free cumulants of the spectral law of $\bm{W}$. Our proof relies on a recursive characterization of the free cumulants, which appears to be novel and may be of independent interest.

\item We demonstrate the versatility of our approach by devising two variants of RI-AMP. The first variant, which we refer to as RI-AMP-DF, employs a different form of the Onsager term which only cancels out the essential non-Gaussian terms. We show that RI-AMP-DF is equivalent to RI-AMP with a change of variables. The second variant (called RI-AMP-MP), which applies a nonlinear matrix processing on $\bm{W}$, is inspired by \citet{barbier2023fundamental}. We further apply RI-AMP-DF to a signal estimation problem in spiked models. Our approach provides a generalization of the BAMP algorithm proposed in \citep{barbier2023fundamental} to handle general non-polynomial matrix processing functions.
\end{itemize}

\subsection{Organization and Notations}

\paragraph{Organization.} This paper is organized as follows. We start with some preliminary results on free cumulants in Section \ref{Sec:preliminaries}. In Section \ref{Sec:OAMP_SE}, we review some existing results on the orthogonal AMP algorithm. Section \ref{sec:AMP_single} contains the main results of this paper. Section \ref{Sec:generalizations} provides some generalizations and applications of our framework. Section \ref{Sec:numerical} includes some numerical experiments. The appendices contain omitted proofs and detailed calculations.
\paragraph{Notation.} The sets $\R, \C$ represent the set real numbers and complex numbers. $[N]$ represents the set $\{1, 2, \ldots, N\}$ and $\ortho(\dim)$ denotes the set of $\dim \times \dim$ orthogonal matrices. We use the bold-face font for vectors and matrices. $\|\bm{u}\|$ denotes the $\ell_2$ norm of the vector $\bm{u}$. For a vector $\bm{u}\in\bm{R}^N$, $\langle \bm{u}\rangle:=\frac{1}{N}\sum_{i=1}^N u_i$. $\bm{A}\otimes\bm{B}$ and $\bm{A}\circ\bm{B}$ denote Kronecker product and Hadamard product of $\bm{A}$ and $\bm{B}$, respectively. $\diag(\bm{u})$ represents diagonal matrix with diagonal entries given by the entries of $\bm{u}$. $\diag(\bm{A}_1,\ldots,\bm{A}_t)$ denotes a block diagonal matrix with the matrices $\bm{A}_1,\ldots,\bm{A}_t$ placed on the diagonal blocks. $\text{tr}(\bm{M})$ and $\|\bm{M}\|_{\op}$ represent the trace and operator norms of the matrix $\bm{M}$ respectively. $\bm{I}_{k}$ denotes the $k \times k$ identity matrix. We use $\E[\cdot], \Var[\cdot], \Cov[\cdot]$ to denote expectations, variances, and covariances of random variables. $\gauss{\bm{\mu}}{\bm{\Sigma}}$ denotes
the Gaussian distribution with mean vector $\bm{\mu}$ and covariance matrix $\bm{\Sigma}$.  We use $\unif{\ortho(\dim)}$ to denote the Haar measure on the orthogonal group $\ortho(\dim)$. The probability measure $\delta_x$ on $\R$ denotes the point mass at $x\in\mathbb{R}$. We use $\pc$ to denote convergence in probability. For a sequence of real-valued random variables $(Y_{\dim})_{\dim \ge1}$, we say that $\underset{N\to\infty}{\plim}\  Y_N = y$ if $Y_{\dim} \pc y$.

\section{Preliminaries}\label{Sec:preliminaries}

The Onsager term of the rotationally-invariant AMP (RI-AMP) algorithm in \eqref{eq:ZFAMP} involves the \textit{free cumulants} \citep{nica2006lectures} of the spectral measure $\mu$. In this section, we will first review the definition of free cumulants and then introduce a recursive characterization of free cumulants. This recursive characterization, which appears to be novel, will be used in our derivation of RI-AMP.

\subsection{Free Cumulant}

\paragraph{Non-crossing partition.} A \emph{partition} of the set $ \{1,2,\ldots,k\}$ is a collection of nonempty disjoint sets $B_1, B_2,\ldots, B_k$, called \textit{blocks}, whose
union is $[k]$. A partition is \textit{non-crossing} if there are no four distinct elements $1\le a < b < c < d\le k$ such that $a, c$ are in the same block while $b, d$ are in another block. The collection of all non-crossing partition of $\{1,\ldots,k\}$ is denoted as $\text{NC}(k)$. {An example of a non-crossing partition of $\{1,2,3,4,5\}$ is shown in Fig.~\ref{Fig:NC}.}\vspace{5pt}

\begin{figure}[htbp]
\centering
\includegraphics[width=0.2\textwidth]{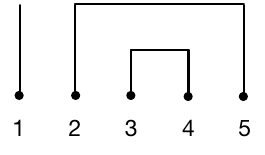}
\caption{A non-crossing partition of $\{1,2,3,4,5\}$: $(\{1\}, \{2,5\},\{3,4\})$.}\label{Fig:NC}
\end{figure}

\vspace{5pt}

\paragraph{Free cumulant.} Let $m_k:=\mathbb{E}[X^k]$ be the $k$-th moment of a random variable $X$. The free cumulants $(\kappa_k)_{k\ge1}$ of $X$ are defined implicitly in terms of the moments $(m_k)_{m\ge1}$ through the moment-cumulant formula \citep{nica2006lectures}:
\begin{equation}\label{Eqn:free_cumulant_def}
m_k=\sum_{\pi\in\text{NC}(k)}\kappa_{\pi},
\end{equation}
where $\kappa_\pi:=\prod_{B\in\pi}\kappa_{|B|}$ is the product of free cumulants corresponding to the cardinality of every block $B\in\pi$. For example, $\text{NC}(3)$ includes the following five non-crossing partitions: $(\{1,2,3\})$, $(\{1,2\},\{3\})$, $(\{1,3\},\{2\})$, $(\{1\},\{2,3\})$, $(\{1\},\{2\},\{3\})$. (In fact, all partitions of $\{1,2,3\}$ are non-crossing.) Then, according to \eqref{Eqn:free_cumulant_def},
\begin{align*}
m_3 &=\kappa_3 + \kappa_2\cdot\kappa_1+\kappa_2\cdot\kappa_1+\kappa_1\cdot\kappa_2+\kappa_1^3.
\end{align*}
Note that the recursive formula \eqref{Eqn:free_cumulant_def} uniquely determines the sequence of free cumulants $(\kappa_k)_{k\ge1}$ from the sequence of moments $(m_k)_{k\ge1}$ \citep[Lecture 10]{nica2006lectures}.

\subsection{A Recursive Characterization of Free Cumulants}\label{Sec:free_cumulants_newIte}

Proposition \ref{Lem:cumulants} below introduces a recursive characterization of free cumulants. This recursion naturally appears in our derivation of the RI-AMP algorithm to be detailed in Section \ref{sec:AMP_single}.

\begin{proposition}\label{Lem:cumulants}
Assume that the moments $(m_n)_{n\ge1}$ of a random variable $\mathsf{\Lambda}$ exist for all orders. Let $(\kappa_n)_{n\ge1}$ be the free cumulants of $\mathsf{\Lambda}$. Define a sequence of random variables $(Q_n)_{n\ge0}$ recursively as follows:
\BE\label{Eqn:Q_def}
Q_n= \mathsf{\Lambda}Q_{n-1}-\sum_{i=1}^n \mathbb{E}[\mathsf{\Lambda}Q_{i-1}]\cdot Q_{n-i},\quad\forall n\ge1,
%Q_{t+1}= \mathsf{\Lambda}Q_{t}-\left(\mathbb{E}[\mathsf{\Lambda}Q_{t}]\cdot Q_0+\mathbb{E}[\mathsf{\Lambda}Q_{t-1}]\cdot Q_1+\cdots + \mathbb{E}\left[\mathsf{\Lambda}Q_0\right]\cdot Q_t \right),\quad\forall t\ge0,
\EE
where $Q_0:=1$. Then, we have 
\BE\label{Eqn:kappa_Q}
\mathbb{E}[\mathsf{\Lambda}Q_{n-1}]=\kappa_{n}, \quad\forall n\ge1.
\EE
\end{proposition}

\begin{proof}
See Appendix \ref{App:preliminary_results}. 
\end{proof}

\begin{remark}[Connection with the partial moments in \citep{fan2022approximate}]
The random variables $(Q_n)_{n\ge0}$ defined in Proposition \ref{Lem:cumulants} are closely related to the ``partial moments'' introduced by \citet{fan2022approximate}, which are doubly-indexed sequence of coefficients that interpolate moments and free cumulants. Let $(m_n)_{n\ge1}$ and $(\kappa_n)_{n\ge0}$ (with the convention $\kappa_0 = 1$) be the moments and free cumulants of $\mathsf{\Lambda}\sim\mu$. The partial moments $(c_{k,j})_{k,j\geq 0}$ are defined via \cite[Appendix A.1]{fan2022approximate}:
\BE\label{Eqn:partial_moments_def}
c_{k,j} = \sum_{m=0}^{j+1} c_{k-1,m} \cdot \kappa_{j+1-m},\quad\forall k\ge 1,j\ge0,
\EE
with initialization $c_{0,0} = 1$ and $c_{0,j}= 0, \forall  j\geq 1$. It can be shown that
\BE\label{Eqn:partial_moment_identity}
c_{k,j} = \mathbb{E}\left [\mathsf{\Lambda}^k Q_{j}\right] ,\quad\forall k,j\ge0.
\EE
The proof of \eqref{Eqn:partial_moment_identity} can be found in Appendix \ref{App:partial_moments}.

\end{remark}

%\subsection{Calculations of Free Cumulants From Moments}

Proposition \ref{Lem:cumulants} suggests a way of computing the sequence of free cumulants $(\kappa_n)_{n\ge1}$ from the sequence of moments $(m_n)_{n\ge1}$. Note that $Q_n$ is a degree-$n$ polynomial of $\mathsf{\Lambda}$ (see \eqref{Eqn:Q_def}). The polynomial coefficients can be computed recursively, as demonstrated in the following corollary.

\begin{corollary}\label{Cor:poly_coefficients}
Let $(\alpha_{n,i})_{1\le n,0\le i\le n}$ be the coefficients for the polynomial representations of $(Q_n)_{n\ge1}$:
\BE\label{Eqn:Q_poly_def}
Q_n =\sum_{i=0}^n \alpha_{n,i} \mathsf{\Lambda}^i,\quad\forall n\ge1.
\EE
Then, $(\alpha_{n,i})_{1\le n,0\le i\le n}$ satisfy the following recursion
\BE\label{Eqn:alpha_recursion}
\alpha_{n,i}=
\alpha_{n-1,i-1}-\sum_{j=1}^{n-i}\kappa_j\cdot \alpha_{n-j,i}, \quad \forall n\ge 1,\ 0\le i\le n-1,
\EE
with $\alpha_{i,i}=1$, and $\alpha_{i,-1}=0$, $\forall i\ge0$.
\end{corollary} 

The proof of Corollary \ref{Cor:poly_coefficients} can be found in Appendix \ref{App:poly_coefficients_proof}. The free cumulant $\kappa_n$ reads
\BE\label{Eqn:kappa_polynomial}
\kappa_{n}=\mathbb{E}[\mathsf{\Lambda}Q_{n-1}]\explain{\eqref{Eqn:Q_poly_def}}{=}\sum_{i=0}^{n-1} \alpha_{n-1,i} \cdot\mathbb{E}\left[\mathsf{\Lambda}^{i+1}\right]=\sum_{i=0}^{n-1} \alpha_{n-1,i} \cdot m_{i+1},\quad\forall n\ge1.
\EE
Overall, the sequence of free cumulants $(\kappa_n)_{n\ge1}$ can be calculated by alternatively updating \eqref{Eqn:alpha_recursion} and \eqref{Eqn:kappa_polynomial} in a recursive fashion. For the reader's convenience, we describe the whole procedure in Algorithm \ref{Alg:exact}.

\begin{center}  % center algorithm
\begin{minipage}{.7\linewidth}  % mini-page
\begin{algorithm}[H]
\caption{Calculating free cumulants from moments}
\begin{algorithmic}[1]
\Statex \textbf{Input:} $(m_n)_{n\ge1}$
\Statex \textbf{Initilization:} $\alpha_{n,n}=1$, $\alpha_{n,-1}=0$, $\forall n\ge0$. $\kappa_1=m_1$.
    \For{$n=1,2,\ldots,$}  
        		\For{$i=0,1,\ldots,n-1$}
		\State $\alpha_{n,i}=\alpha_{n-1,i-1}-\sum_{j=1}^{n-i}\kappa_j\cdot \alpha_{n-j,i}$
		\EndFor
		\State $\kappa_{n+1}=\sum_{j=0}^{n} \alpha_{n,j} \cdot m_{j+1}$
    \EndFor 
\Statex \textbf{Output:} $({\kappa}_n)_{n\ge1}$
\end{algorithmic}\label{Alg:exact}
\end{algorithm}
\end{minipage}
\end{center}

\vspace{5pt}

\begin{remark}[Alternative calculations of free cumulants]
A more common approach to compute the free cumulant $\kappa_n$ of a measure from its moments $m_1,\ldots,m_n$ is the following recursive formula:
\BE\label{Eqn:moment_cumulant_polynomial}
\kappa_n =m_n - \sum_{j=1}^{n-1} \prod_{\substack{k_1,\ldots,k_j\ge0\\ k_1+\cdots+k_j=n-j} } \kappa_j m_{k_1}\cdots m_{k_j}.
\EE
For a derivation of \eqref{Eqn:moment_cumulant_polynomial} from the moment-cumulant formula \eqref{Eqn:free_cumulant_def}, refer to \cite[Proposition 17]{mingo2017free}. The proposed method is computationally favorable than a naive implementation of \eqref{Eqn:moment_cumulant_polynomial}.

\end{remark}

%==============================================
\subsection{Monte Carlo Estimator of Free Cumulants}

The Onager term of the RI-AMP algorithm involves the free cumulants of the limiting spectral law $\mathsf{\Lambda}\sim\mu$. In applications, it is often the case that $\mu$ is not known and only a sample of $\bm{W}$ is available. In this case, we may simply replace the moments $(m_n)_{n\ge1}$ in Algorithm \ref{Alg:exact} or \eqref{Eqn:moment_cumulant_polynomial} by their estimates $(\hat{m}_n)_{n\ge1}$. Assuming that the empirical eigenvalue distribution of $\bm{W}$ converges weakly to a compactly supported distribution, we have $\frac{1}{N}\bm{g}^\UT \bm{W}^n\bm{g}\overset{a.s.}{\longrightarrow }\mathbb{E}[\mathsf{\Lambda}^n]=m_n$ where $\bm{g}\sim\mathcal{N}(\bm{0},\bm{I}_N)$. Motivated by this fact, we consider the following estimator of $m_n$:
\BE\label{Eqn:mk_estimate}
\hat{m}_n:=\frac{1}{N}\bm{g}^\UT \bm{W}^n\bm{g},\quad \bm{g}\sim\mathcal{N}(\bm{0},\bm{I}).
\EE
This estimator does not require eigenvalue decomposition of $\bm{W}$ and is thus computationally favorable for large-scale problems. Similar idea was proposed by \citet{liu2022memory} in the context of linear models.

\begin{center}  % center algorithm
\begin{minipage}{.7\linewidth}  % mini-page
\begin{algorithm}[H]
\caption{Monte Carlo estimator of free cumulants}
\begin{algorithmic}[1]

%\Procedure{Roy}{$a,b$}       \Comment{This is a test}
\Statex \textbf{Input: $\bm{W}$, $\bm{g}\sim\mathcal{N}(\bm{0},\bm{I}_N)$, $\bm{z}_0:=\bm{g}$, $\hat{\kappa}_0=1$, $\bm{h}:=\bm{Wg}$} 
    \For{$n=1,2,\ldots,$}  
    \State $\hat{\kappa}_n=\frac{1}{N}\bm{h}^\UT \bm{z}_{n-1}$
        \State $\bm{z}_n=\bm{Wz}_{n-1}-\sum_{i=1}^n \hat{\kappa}_i\cdot  \bm{z}_{n-i}$
    \EndFor
%\EndProcedure
\Statex \textbf{Output:} $(\hat{\kappa}_n)_{n\ge1}$
\end{algorithmic}\label{Alg:proposed}
\end{algorithm}
\end{minipage}
\end{center}

Here we propose an alternative Monte Carlo estimator of the free cumulants based on Proposition \ref{Lem:cumulants}. Let $(\bm{Q}_n)_{n\ge1}$ be a sequence of polynomials of $\bm{W}$ defined via (cf.~\eqref{Eqn:Q_def})
\BE\label{Eqn:Q_matrix_recursion}
\bm{Q}_n= \bm{WQ}_{n-1}-\sum_{i=1}^n \kappa_i\cdot \bm{Q}_{n-i},\quad\forall n\ge1
\EE
with $\bm{Q}_0=\bm{I}_N$. Similar to \eqref{Eqn:mk_estimate} we can estimate $\kappa_k$ as follows 
\[
\hat{\kappa}_n = \frac{1}{N}\bm{g}^\UT\bm{W}\bm{Q}_{n-1}\bm{g},\quad\bm{g}\sim\mathcal{N}(\bm{0}_N,\bm{I}_N).
\]
Note that we do not need to compute the matrices $(\bm{Q}_n)_{n\ge1}$ explicitly. An implementation of this idea is detailed in Algorithm \ref{Alg:proposed}.

\section{Orthogonal AMP Algorithm}\label{Sec:OAMP_SE}

Our approach to derive the RI-AMP algorithm is by reducing it (in a recursive fashion) to certain long-memory orthogonal AMP (OAMP) algorithm \cite{takeuchi2019unified} which has very simple state evolution characterization. In this section, we collect some existing results about the OAMP algorithm for later use.

\begin{definition}[Orthogonal AMP algorithm]\label{Def:LM_OAMP}
Starting from an initialization $\bar{\bm{x}}_1\in\mathbb{R}^N$, an orthogonal AMP algorithm  proceeds as follows:
\BS\label{Eqn:LM_OAMP_recall}
\begin{align}
\bm{x}_t &= \left(f_t(\bm{W})-\frac{\text{tr}\left(f_t(\bm{W})\right)}{N}\cdot\bm{I}_N\right)\bar{\bm{x}}_{t},\\
\bar{\bm{x}}_{t+1} &= g_{t+1} (\bm{x}_1,\ldots,\bm{x}_t;\bm{a})-\sum_{i=1}^t \langle\partial_i  g_{t+1} (\bm{x}_1,\ldots,\bm{x}_t;\bm{a})\rangle\cdot \bm{x}_i,\label{Eqn:LM_OAMP_recall_b}
\end{align}
\ES
where 
\begin{itemize}
\item $f_{t}:\mathbb{R}\mapsto\mathbb{R}$ is applied to $\bm{W}$ in the following sense: let $\bm{W}=\bm{O}\mathrm{diag}\left(\lambda_1,\ldots,\lambda_N\right)\bm{O}^\UT$ be the eigenvalue decomposition of $\bm{W}$, then $f_t(\bm{W}):=\bm{O}\mathrm{diag}\left(f_t(\lambda_1),f_t(\lambda_2),\ldots,f_t(\lambda_N)\right)\bm{O}^\UT$.
\item $g_{t+1}:\mathbb{R}^t\times\mathbb{R}^k\mapsto\mathbb{R}$ acts separately on the $N$ rows of $(\bm{x}_1,\ldots,\bm{x}_t;\bm{a})\in\mathbb{R}^{N\times t}\times\mathbb{R}^{N\times k}$, and $\langle\partial_i  g_{t+1} (\bm{x}_1,\ldots,\bm{x}_t;\bm{a})\rangle$ denotes the empirical average of the partial derivative of $g_{t+1}$ w.r.t. the $i$-th argument, i.e.,
\[
\langle\partial_i  g_{t+1} (\bm{x}_1,\ldots,\bm{x}_t;\bm{a})\rangle:=\frac{1}{N}\sum_{n=1}^N \frac{\partial g_{t+1}(\bm{x}_1[n],\ldots,\bm{x}_t[n];\bm{a}[n])}{\partial \bm{x}_i[n]}.
\]
Here, $\bm{a}\in\mathbb{R}^{N\times k}$ represents available side information, and $\bm{x}_i[n]\in\mathbb{R}$ and $\bm{a}[n]\in\mathbb{R}^{k}$ denote the $n$-th component of $\bm{x}_i$ and $n$-th row of $\bm{a}$ respectively. 
\end{itemize}
\end{definition}

\vspace{5pt}

Throughout this paper, we will use the following convergence of high-dimensional vectors. The reader is referred to \cite{bayati2011dynamics,fan2022approximate,feng2022unifying} for more information about this notion of convergence.

\begin{definition}[Convergence of high-dimensional vectors]\label{Eqn:W2_convergence}
Let $(\bm{z}_1,\ldots,\bm{z}_\ell)\in\mathbb{R}^{N\times \ell}$ be a collection of random vectors. We say its empirical distribution converges to random variables $(\mathsf{Z}_1,\ldots,\mathsf{Z}_\ell)$ as $N\to\infty$, which we denote as 
\[
(\bm{z}_1,\ldots,\bm{z}_\ell)\overset{W_2}{\longrightarrow}(\mathsf{Z}_1,\ldots,\mathsf{Z}_\ell),
\]
if for any test function $\psi:\mathbb{R}^{\ell}\mapsto\mathbb{R}$ satisfying
\BE
|\psi(\bm{a})-\psi(\bm{b})|\le L\|\bm{a}-\bm{b}\|_2(1+\|\bm{a}\|_2+\|\bm{b}\|_2),\quad\forall \bm{a},\bm{b}\in\mathbb{R}^{\ell},
\EE
the following holds as $N\to\infty$:
\BE
\frac{1}{N}\sum_{n=1}^N \psi\left(\bm{z}_1[n],\ldots,\bm{z}_\ell[n]\right)\overset{\mathbb{P}}{\longrightarrow}\mathbb{E}\left[\psi\left(\mathsf{Z}_1,\ldots,\mathsf{Z}_\ell\right)\right].
\EE
\end{definition}

The following assumptions are needed for the high-dimensional asymptotic characterization of OAMP.

\begin{assumption}[Assumptions for OAMP]\label{Ass:LM_OAMP}
$\ $
\begin{enumerate}
\item[(1)] Let $\bm{W}=\bm{O}\text{diag}(\bm{\lambda})\bm{O}^\UT$ be the eigenvalue decomposition of $\bm{W}$. We assume $\bm{O}\sim\mathsf{Unif}(\mathbb{O}(N))$ and $\bm{\lambda}\in\mathbb{R}^N$ is deterministic. Moreover, the empirical distribution of $\bm{\lambda}$ converges weakly to a compactly support probability measure $\mu$. Additionally, the operator norm of $\bm{W}$ is upper bounded by an $N$-independent constant $C$.
\item[(2)] The side information $\bm{a}\in\mathbb{R}^{N\times k}$ and the initialization $\bar{\bm{x}}_1$ are independent of $\bm{O}$. Moreover, $(\bar{\bm{x}}_1,\bm{a})\overset{W_2}{\longrightarrow}(\bar{\mathsf{X}}_1,\mathsf{A})$, where $\mathbb{E}[\bar{\mathsf{X}}_1^2]<\infty$ and $\mathbb{E}[\|\mathsf{A}\|^2]<\infty$.
\item[(3)] For all $t\ge1$, the matrix denoiser $f_t:\mathbb{R}\mapsto\mathbb{R}$ is continuous, and does not depend on $N$.
\item[(4)] For all $t\ge1$, the function $g_{t+1}:\mathbb{R}^t\times\mathbb{R}^k\mapsto\mathbb{R}$ is continuously-differentiable and Lipschitz, and does not depend on $N$.
\end{enumerate}
\end{assumption}

The high-dimensional asymptotic performance of OAMP admits a state evolution description. This was first conjectured in \cite{ma2017orthogonal} for the OAMP algorithm with univariate denoisers. A rigorous proof was provided by \cite{rangan2019vector,takeuchi2019rigorous} based on a generalization of the conditioning technique in \cite{bolthausen2014iterative,bayati2011dynamics} to Haar random orthogonal matrices. (The proof in \cite{rangan2019vector} was for the vector AMP (VAMP) algorithm. In this paper, the acronyms OAMP and VAMP are used interchangeably.) Generalizations of OAMP to the case of multivariate denoisers are introduced in \cite{takeuchi2019unified,dudeja2022spectral,liu2022memory}. The following theorem provides a state evolution characterization for the version of OAMP algorithm as defined in Definition \ref{Def:LM_OAMP}. Its proof, which is based on a simple modification of \cite{dudeja2022spectral}, is deferred to Appendix \ref{App:proof_OAMP}.

\begin{theorem}[State evolution of OAMP \cite{dudeja2022spectral}]\label{The:OAMP_SE}
Let $(\bm{x}_t)_{t\ge1}$ be generated via the OAMP algorithm. Suppose Assumption \ref{Ass:LM_OAMP} holds. Then, the following holds as $N\to\infty$:
\BS
\BE
(\bm{x}_1,\ldots,\bm{x}_t;\bm{a})\overset{W_2}{\longrightarrow}(\mathsf{X}_1,\ldots,\mathsf{X}_t;\mathsf{A}),\quad\forall t\ge1,
\EE
where $(\mathsf{X}_1,\ldots,\mathsf{X}_t)\sim\mathcal{N}(\bm{0}_t,\paraB{\Omega}_t)$ is independent of $(\bar{\mathsf{X}}_1,\mathsf{A})$, and
\BE\label{Eqn:SE_OAMP_first}
\paraB{\Omega}_t =
\begin{bmatrix}
\mathrm{Cov}_\mu\left[f_1,f_1\right]   &  \cdots & \mathrm{Cov}_\mu\left[f_1,f_t\right]   \\
\mathrm{Cov}_\mu\left[f_2,f_1\right]   &  \cdots & \mathrm{Cov}_\mu\left[f_2,f_t\right]  \\
\vdots& \ddots  & \vdots \\
\mathrm{Cov}_\mu\left[f_t,f_1\right]   &  \cdots & \mathrm{Cov}_\mu\left[f_t,f_t\right]   \\
\end{bmatrix}
\circ
\begin{bmatrix}
\mathbb{E}\left[\bar{\mathsf{X}}_1\bar{\mathsf{X}}_1\right]   &  \cdots & \mathbb{E}\left[\bar{\mathsf{X}}_1\bar{\mathsf{X}}_t\right]  \\
\mathbb{E}\left[\bar{\mathsf{X}}_2\bar{\mathsf{X}}_1\right]   &  \cdots &\mathbb{E}\left[\bar{\mathsf{X}}_2\bar{\mathsf{X}}_t\right]  \\
\vdots &\ddots  & \vdots \\
\mathbb{E}\left[\bar{\mathsf{X}}_t\bar{\mathsf{X}}_1\right]  & \cdots  & \mathbb{E}\left[\bar{\mathsf{X}}_t\bar{\mathsf{X}}_t\right] 
\end{bmatrix},
\EE
where $\mathrm{Cov}_\mu\left[f_i,f_j\right]$ denotes the covariance between the random variables $f_i(\mathsf{\Lambda})$ and $f_j(\mathsf{\Lambda})$ (with $\mathsf{\Lambda}\sim\mu$), and $\circ$ denotes Hadamard product and
\BE\label{Eqn:SE_OAMP_b}
\bar{\mathsf{X}}_{t}:=
g_{t} (\mathsf{X}_1,\ldots,\mathsf{X}_{t-1};\mathsf{A})-\sum_{i=1}^{t-1}  \mathbb{E}\left[\partial_i  g_{t} (\mathsf{X}_1,\ldots,\mathsf{X}_{t-1};\mathsf{A})\right]\cdot \mathsf{X}_{i},\quad\forall t\ge2.
\EE
In the above equation, $\mu$ denotes the limiting eigenvalue distribution of $\bm{W}$.
\ES
\end{theorem}

\vspace{5pt}

\begin{remark}[Interpretation of OAMP's state evolution]
The state evolution equation \eqref{Eqn:SE_OAMP_first} makes precise the following approximation:
\BS\label{Eqn:decoupling_intro}
\begin{align}
\frac{1}{N}\bm{x}_s^\UT\bm{x}_t &\explain{\eqref{Eqn:LM_OAMP_recall}}{=} \frac{1}{N}\bar{\bm{x}}_s^\UT \left(f_s(\bm{W})-\frac{\text{tr}\left(f_s(\bm{W})\right)}{N}\cdot\bm{I}_N\right) \left(f_t(\bm{W})-\frac{\text{tr}\left(f_t(\bm{W})\right)}{N}\cdot\bm{I}_N\right)\bar{\bm{x}}_t,\quad \forall s,t\ge1\\
&\simeq \frac{1}{N}\mathsf{tr} \left(f_s(\bm{W})-\frac{\text{tr}\left(f_s(\bm{W})\right)}{N}\cdot\bm{I}_N\right) \left(f_t(\bm{W})-\frac{\text{tr}\left(f_t(\bm{W})\right)}{N}\cdot\bm{I}_N\right)\cdot \frac{1}{N}\bar{\bm{x}}_s^\UT\bar{\bm{x}}_t.
\end{align}
\ES
Standard random matrix results show that the above approximation is exact (under mild conditions) in the large-$N$ limit, if $\bar{\bm{x}}_t$ and $\bar{\bm{x}}_s$ were independent of the noise matrix $\bm{W}$. Therefore, for a heuristic interpretation of the state evolution, we may treat the iterates as if they are independent of $\bm{W}$. This property significantly simplifies the dynamics of OAMP algorithms and renders the state evolution easily interpretable.
\end{remark}

A distinguishing feature of OAMP algorithm is the pairwise asymptotic orthogonality between the iterates $\bm{x}_s$ and $\bar{\bm{x}}_t$, $\forall s,t\ge1$. This orthogonality property justifies the name ``orthogonal AMP''.

\begin{proposition}[Pairwise orthogonality]\label{Pro:orthogonality}
The following holds as $N\to\infty$,
\BE\label{Eqn:OAMP_orthogonality}
\underset{N\to\infty}{\mathrm{plim}}\ \frac{1}{N}\bm{x}_s^\UT\bar{\bm{x}}_t=0,\quad \forall s,t\ge1.
\EE
\end{proposition}

\begin{proof}
The case $t=1$ is straightforward. We consider the case $t>1$. From Theorem \ref{The:OAMP_SE}, the joint empirical distributions of $\bm{x}_s$ and $(\bm{x}_1,\ldots,\bm{x}_{t-1};\bm{a})$ converges in the sense of Definition \ref{Eqn:W2_convergence}. Consequently,
\begin{align*}
\underset{N\to\infty}{\mathrm{plim}}\ \frac{1}{N}\bm{x}_s^\UT\bar{\bm{x}}_t&\explain{(a)}{=} \underset{N\to\infty}{\mathrm{plim}}\ \frac{1}{N}\bm{x}_s^\UT\left(g_{t} (\bm{x}_1,\ldots,\bm{x}_{t-1};\bm{a})-\sum_{i=1}^{t-1} \langle\partial_i  g_{t+1} (\bm{x}_1,\ldots,\bm{x}_{t-1};\bm{a})\rangle\cdot \bm{x}_i\right)\\
&\explain{(b)}{=}\mathbb{E}\left[\mathsf{X}_s g_t(\mathsf{X}_1,\ldots,\mathsf{X}_{t-1};\mathsf{A})\right]-\sum_{i=1}^{t-1}\mathbb{E}\left[\partial_i g_t(\mathsf{X}_1,\ldots,\mathsf{X}_{t-1};\mathsf{A})\right]\cdot \mathbb{E}[\mathsf{X}_s \mathsf{X}_i ]\\
&\explain{(c)}{=}0,
\end{align*}
where step (a) is from the definition of $\bar{\bm{x}}_t$ in \eqref{Eqn:LM_OAMP_recall_b}; step (b) is due to the convergence result in Theorem \ref{The:OAMP_SE} and the fact that $\partial_i g_t$ is continuous and bounded (from the assumption that $g_t$ is Lipschitz and continuously-differentiable); step (c) is from the multivariate version of Stein's lemma.
\end{proof}

\begin{remark}
When $s\ge t$, $\bm{x}_s$ is not an immediate input of $\bar{\bm{x}}_t:=g_{t} (\bm{x}_1,\ldots,\bm{x}_{t-1};\bm{a})-\sum_{i=1}^{t-1} \langle\partial_i  g_{t+1} (\bm{x}_1,\ldots,\bm{x}_{t-1};\bm{a})\rangle\cdot \bm{x}_i$. Note that the orthogonality property \eqref{Eqn:OAMP_orthogonality} still holds in such cases.
\end{remark}

\vspace{5pt}

\begin{remark}[Orthogonal decomposition]
We can decompose ${g}_{t+1}(\bm{x}_1,\ldots,\bm{x}_t;\bm{a})$ as
\BE\label{Eqn:orthogonal_decomposition}
{g}_{t+1}(\bm{x}_1,\ldots,\bm{x}_t;\bm{a}) = \sum_{i=1}^t \langle\partial_i  g_{t+1} (\bm{x}_1,\ldots,\bm{x}_t;\bm{a})\rangle\cdot \bm{x}_i + \bar{\bm{x}}_{t+1},
\EE
where $\bar{\bm{x}}_{t+1}=g_{t+1} (\bm{x}_1,\ldots,\bm{x}_t;\bm{a})-\sum_{i=1}^t \langle\partial_i  g_{t+1} (\bm{x}_1,\ldots,\bm{x}_t;\bm{a})\rangle\cdot \bm{x}_i$. Proposition \ref{Pro:orthogonality} guarantees that $\bar{\bm{x}}_{t+1}$ is asymptotically orthogonal to $\bm{x}_i$, $\forall i\in[t]$. Namely, \eqref{Eqn:orthogonal_decomposition} is a decomposition of ${g}_{t+1}(\bm{x}_1,\ldots,\bm{x}_t;\bm{a})$ into a component that lies in $\text{span}\{\bm{x}_1,\ldots,\bm{x}_t\}$ and a component perpendicular to it asymptotically.
\end{remark}

\vspace{5pt}

OAMP algorithms have simple and interpretable state evolution, thanks to the use of trace-free and divergence-free denoisers. Moreover, with the introduction of general long-memory processing, the OAMP framework is very flexible and allows for convenient derivation of AMP algorithms (as we show in subsequent sections). These attributes make OAMP a unified and simple approach for deriving AMP algorithms tailored to rotationally-invariant models.

%---------------------------- % --------------------------
%---------------------------- % --------------------------
%---------------------------- % --------------------------
%---------------------------- % --------------------------
\section{Main Results}\label{sec:AMP_single}

This section presents the main results of this paper. We provide a unified way of deriving the AMP algorithms for rotationally-invariant models.

\subsection{Rotationally-Invariant AMP Algorithm}\label{Sec:RI-AMP}

The starting point of our derivation of RI-AMP algorithm is the first order method (FOM) defined below. We remark that FOM has the same form as RI-AMP, but the de-biasing coefficients in FOM are free parameters and not necessarily set as in RI-AMP.

\begin{definition}[First order method (FOM)]\label{Def:FOM}
Starting from an initialization $\bm{u}_1\in\mathbb{R}^N$, a first order method (FOM) proceeds as
\BS\label{eq:ZFAMP_general_def}
\begin{align}
\vr_t &= \mW\vu_t - \left(\para{b}_{t,1}\bm{u}_1+\para{b}_{t,2}\bm{u}_2+\cdots+\para{b}_{t,t}\bm{u}_t\right), \quad\forall t\in\mathbb{N},\label{eq:ZFAMP_a}\\
\vu_{t+1} &= \eta_{t+1} (\vr_1,\ldots,\vr_t),
\end{align}
\ES
where $\eta_{t+1}:\mathbb{R}^t\mapsto\mathbb{R}$ acts separately on the $N$ rows of $(\bm{r}_1,\ldots,\bm{r}_t)\in\mathbb{R}^{N\times t}$. 
\end{definition}

For convenience, let $\paraB{B}_t\in\mathbb{R}^{t\times t}$ be the matrix collecting the de-biasing coefficients $(\para{b}_{t,i})_{1\le t,1\le i\le t}$ as
\BE\label{Eqn:B_t_def}
\paraB{B}_t:=
\begin{bmatrix}
\para{b}_{1,1}& & & \\
\para{b}_{2,1}& \para{b}_{2,2} & & \\
\vdots & \vdots & \ddots & \\
\para{b}_{t,1}& \para{b}_{t,2}& \cdots & \para{b}_{t,t}
\end{bmatrix}.
\EE 
The rotationally-invariant AMP algorithm is a specific way of setting the de-biasing matrix $\paraB{B}_t$ in order to make the iterates $(\bm{r}_t)_{t\ge1}$ asymptotically Gaussian. We can now define the RI-AMP algorithm formally.

\begin{definition}[RI-AMP algorithm \citep{fan2022approximate}]\label{Def:RI-AMP}
The rotationally-invariant AMP (RI-AMP) algorithm is a first order method (FOM) with the de-biasing coefficients matrix $\paraB{B}_t$ set to
\BE\label{Eqn:RI-AMP_B_Phi}
\paraB{B}_t = \sum_{i=1}^{t} \kappa_{i} (\hat{\paraB{\Phi}}_t)^{i-1},
\EE
where $(\kappa_i)_{i\ge1}$ are the free cumulants of the limiting spectral law $\mu$, and the matrix $\hat{\paraB{\Phi}}_t\in\mathbb{R}^{t\times t}$ collects the empirical partial derivatives of $(\eta_i)_{i\in[t]}$ (i.e., ``divergences"):
\BE\label{Eqn:Phi_def}
\hat{\paraB{\Phi}}_t :=
\begin{bmatrix}
0&  &  &  &  \\
\langle\partial_1 \bm{u}_2\rangle& 0 &  &  & \\
\langle\partial_1 \bm{u}_3\rangle& \langle\partial_2 \bm{u}_3 \rangle &0 & &\\
\vdots & \vdots & \ddots &  &  \\
\langle\partial_1 \bm{u}_t\rangle & \langle\partial_2 \bm{u}_t\rangle & \cdots & \langle\partial_{t-1} \bm{u}_t\rangle  & 0
\end{bmatrix}.
\EE
\end{definition}

\begin{remark}[R-transform]
Recall that the free cumulants $(\kappa_i)_{i\ge1}$ are the coefficients in the power series expansion of the R-transform \cite{nica2006lectures}. Consequently, the de-biasing matrix formula \eqref{Eqn:RI-AMP_B_Phi} can be expressed in terms of the R-transform, as employed in the work of \citet{opper2016theory}. The R-transform emerges in \citet{opper2016theory} as a result of evaluating certain high-dimensional integrals asymptotically. In the present work, free cumulants arise naturally within a recursive centering operation.
\end{remark}

\begin{remark}[Reduction to standard Gaussian AMP]
For the special iid Gaussian model, namely when $\bm{W}$ is drawn from the Gaussian orthogonal Ensemble (GOE), the RI-AMP algorithm simplifies significantly. In this case, the free cumulants vanish except for $\kappa_2$: 
\[
(\kappa_1,\kappa_2,\kappa_3,\kappa_4,\ldots)=(0,1,0,0,\ldots).
\]
Hence, the de-biasing matrix formula \eqref{Eqn:RI-AMP_B_Phi} reduces to $\paraB{B}_t = \hat{\paraB{\Phi}}_t$, and RI-AMP reduces to the standard (long-memory) Gaussian AMP.
\end{remark}

\vspace{5pt}

Following \citep{fan2022approximate}, we make the following assumptions in order to analyze the asymptotic performance of RI-AMP.

\begin{assumption}[Assumptions for RI-AMP]\label{Ass:RI-AMP}
$\ $
\begin{enumerate}
\item[(1)] Let $\bm{W}=\bm{O}\text{diag}(\bm{\lambda})\bm{O}^\UT$ be the eigenvalue decomposition of $\bm{W}$. We assume $\bm{O}\sim\mathsf{Unif}(\mathbb{O}(N))$ and $\bm{\lambda}\in\mathbb{R}^N$ is deterministic. Moreover, the empirical distribution of $\bm{\lambda}$, which is assumed to be deterministic, converges weakly to a compactly support probability measure $\mu$. Additionally, the operator norm of $\bm{W}$ is upper bounded by an $N$-independent constant $C$.
\item[(2)] The initialization $\bm{u}_1\in\mathbb{R}^N$ is independent of $\bm{O}$. Moreover, $\bm{u}_1\overset{W_2}{\longrightarrow} U_1$, where the random variable $U_1$ has moments with all orders.
\item[(3)] For all $t\ge1$, the function $\eta_t:\mathbb{R}^t\mapsto\mathbb{R}$ is Lipschitz continuous.
\end{enumerate}

\end{assumption}

\subsection{Overview of Our Approach}\label{Sec:overview}

Before presenting the technical details, we first give a brief overview of our approach to AMP algorithms for rotationally-invariant models. Recall that the RI-AMP algorithm has the following form:
\BS\label{eq:ZFAMP_recall}
\begin{align}
\vr_t &= \mW\vu_t - \left(\para{b}_{t,1}\bm{u}_1+\para{b}_{t,2}\bm{u}_2+\cdots+\para{b}_{t,t}\bm{u}_t\right),\\
\vu_{t+1} &= \eta_{t+1}(\vr_t).
\end{align}
\ES
For ease of discussion, here we assumed that $\eta_t(\cdot)$ is a single-iterate function that only depends on $\bm{r}_t$. The question we aim to address is:
\begin{itemize}
\item \textit{How to derive the correct de-biasing coefficients $(\para{b}_{t,i})_{1\le t,1\le i\le t}$ such that $(\bm{r}_t)_{t\ge1}$ are asymptotically Gaussian distributed?}
\end{itemize}
Our approach to this question is based on an {orthogonal decomposition} idea first introduced by \citet{dudeja2022spectral}. Specifically, we decompose the iterate $\bm{u}_{t+1}:=\eta_{t+1}(\bm{r}_t)$, for each $t\ge1$, into a component that is parallel to the input vector $\bm{r}_t$ and a component orthogonal to it (denoted as $\bar{\bm{u}}_{t+1}$). For the more general scenario where $\eta_{t+1}(\cdot)$ depends on all past iterates $(\bm{r}_1,\ldots,\bm{r}_t)$, this orthogonal decomposition idea can be generalized naturally: we decompose $\eta_{t+1}(\bm{r}_1,\ldots,\bm{r}_t)$ into a component that lies in $\text{span}(\bm{r}_1,\ldots,\bm{r}_t)$ and a component orthogonal to it. (Notice that if the input vectors $(\bm{r}_1,\ldots,\bm{r}_t)$ are jointly Gaussian, the residual orthogonal component $\bar{\bm{u}}_{t+1}$ is nothing but the divergence-free estimate defined in Section \ref{Sec:OAMP_SE}.) Recursively unfolding the algorithm using this orthogonal decomposition, it is possible to represent the iterates $(\bm{r}_t)_{t\ge1}$ as linear combinations of the orthogonal components $(\bar{\bm{u}}_{t})_{t\ge1}$. This representation is already close to an OAMP algorithm defined in Section \ref{Sec:OAMP_SE}. It turns out that, by properly choosing the de-biasing coefficients, it is always possible to make the matrices appeared in this linear representation of $(\bar{\bm{u}}_{t})_{t\ge1}$ asymptotically trace-free, thereby casting the algorithm into an OAMP algorithm. The asymptotic Gaussian distribution then follows from the general property of OAMP.

We illustrate the idea by considering the first few iterations. In our discussions, we will encounter a centering operation that makes a matrix asymptotically trace-free, and we shall denote this centering operation by $\mathcal{T}$. Specifically, letting $f:\mathbb{R}\mapsto\mathbb{R}$ be a matrix denoising function which acts on input matrices as in \eqref{Eqn:LM_OAMP_recall}, we denote
\BE
\mathcal{T}\left(\bm{A}\right):=\bm{A}-\mathbb{E}_{\mathsf{\Lambda}\sim\mu}[f(\mathsf{\Lambda})]\cdot\bm{I}_N,\quad \text{where }\bm{A}:=f(\bm{W}).
\EE
We now consider the first iteration of the algorithm. From \eqref{eq:ZFAMP_recall}, we have $\bm{r}_1=\left(\bm{W} - \para{b}_{1,1}\bm{I}_N\right)\bm{u}_1$. By setting $ \para{b}_{1,1}=m_1:=\mathbb{E}[\mathsf{\Lambda}]$, we have
\BE\label{Eqn:AMP_r1}
\bm{r}_1=\mathcal{T}(\bm{W})\bm{u}_1:=Q_1(\bm{W})\bm{u}_1.
\EE
This guarantees $\bm{r}_1$ to be asymptotically Gaussian by Theorem \ref{The:OAMP_SE} in Section \ref{Sec:OAMP_SE}. Next, we show how to choose $\para{b}_{2,1}$ and $\para{b}_{2,2}$ to make $\bm{r}_2$ asymptotically Gaussian. Our approach is based on an {orthogonal decomposition} of $\bm{u}_1$:
\BE\label{Eqn:AMP_u2}
\bm{u}_2:=\para{d}_1\cdot\bm{r}_1 + \bar{\bm{u}}_{2},
\EE
where the first term is parallel to $\bm{r}_1$ and the second component orthogonal to $\bm{r}_1$ (in certain asymptotic sense). When $\bm{r}_1$ is asymptotically Gaussian, by appealing to Stein's lemma, the choice $\para{d}_1:=\left\langle \partial_1 \bm{u}_{2}\right\rangle$ satisfies the asymptotic orthogonality condition. The orthogonal residual $\bar{\bm{u}}_{2}:=\eta_1(\bm{r}_1)-\para{d}_1\cdot\bm{r}_1$ is a divergence-free function of $\bm{r}_1$. {For notational simplicity, we assume $\left\langle \partial_1 \bm{u}_{2}\right\rangle=1$ in the following discussions.} Now, substituting \eqref{Eqn:AMP_u2} into \eqref{eq:ZFAMP_recall} and denoting $\bar{\bm{u}}_1:=\bm{u}_1$ yields
\BS
\begin{align*}
\bm{r}_2&=\bm{W u}_2-(\para{b}_{2,1}\bm{u}_1+\para{b}_{2,2}\bm{u}_2)\\
&\explain{\eqref{Eqn:AMP_u2}}{=}\bm{W}\left(\bm{r}_1 + \bar{\bm{u}}_{2}\right)-\para{b}_{2,1}\bar{\bm{u}}_{1}-\para{b}_{2,2}\left(\bm{r}_1+\bar{\bm{u}}_2\right)\\
&\explain{\eqref{Eqn:AMP_r1}}{=}{\left(\bm{W}Q_1(\bm{W})-\para{b}_{2,2}Q_1(\bm{W})-\para{b}_{2,1}\bm{I}\right)}\bar{\bm{u}}_1+{\left(\bm{W}-\para{b}_{2,2}\bm{I}\right)}\bar{\bm{u}}_2\\
&=\underbrace{\mathcal{T}\big(\bm{W}Q_1(\bm{W})-\para{b}_{2,2}Q_1(\bm{W})-\para{b}_{2,1}\bm{I}\big)}_{Q_2(\bm{W})}\bar{\bm{u}}_1+\underbrace{\mathcal{T}\left(\bm{W}-\para{b}_{2,2}\bm{I}\right)}_{Q_1(\bm{W})}\bar{\bm{u}}_2,
\end{align*}
\ES
where the last step holds when $\para{b}_{2,1}$ and $\para{b}_{2,2}$ are chosen to center the respective matrices:
\BE\label{Eqn:r2}
\begin{split}
\para{b}_{2,2} &=\lim_{N\to\infty} \frac{1}{N}\text{tr}\left(\bm{W}\right)=m_1,\\
\para{b}_{2,1} &=\lim_{N\to\infty} \frac{1}{N}\text{tr}\left(\bm{W}Q_1(\bm{W})-\para{b}_{2,2}Q_1(\bm{W})\right)=\lim_{N\to\infty} \frac{1}{N}\text{tr}\left(\bm{W}Q_1(\bm{W})\right)\explain{\eqref{Eqn:AMP_r1}}{=}m_2-m_1^2.
\end{split}
\EE
Again, by Theorem \ref{The:OAMP_SE} in Section \ref{Sec:OAMP_SE}, vectors in the form of $P(\bm{W})\bar{\bm{u}}$ are asymptotically Gaussian with $P(\cdot)$ a polynomial and $P(\bm{W})$ \emph{trace-free}, and $\bar{\bm{u}}$ \emph{divergence-free}. Hence, both $Q_2(\bm{W})\bar{\bm{u}}_1$ and $Q_1(\bm{W})\bar{\bm{u}}_2$ are asymptotically Gaussian, and so is $\bm{r}_2=Q_2(\bm{W})\bar{\bm{u}}_1+Q_1(\bm{W})\bar{\bm{u}}_2$. 

We can continue this process. In each step, we decompose the estimate $\eta_t(\bm{r}_t)$ as a linear term plus a divergence-free term. {To illustrate the main idea and simplify discussions, we assume that $\langle \partial_t(\bm{u}_{t+1})\rangle=1$ $\forall t\ge1$ in this section.} It is not difficult to show that, under proper choices of the de-biasing coefficients, we could represent $(\bm{r}_t)_{t\ge1}$ as:
\BE\label{Eqn:r1_rt}
\begin{split}
\bm{r}_1 &=Q_1(\bm{W})\bar{\bm{u}}_1,\\
\bm{r}_2 &=Q_2(\bm{W})\bar{\bm{u}}_1+Q_1(\bm{W})\bar{\bm{u}}_2,\\
\bm{r}_3 &=Q_3(\bm{W})\bar{\bm{u}}_1+Q_2(\bm{W})\bar{\bm{u}}_2 +Q_1(\bm{W})\bar{\bm{u}}_3,\\
&\qquad\qquad \vdots\\
\bm{r}_t &=Q_t(\bm{W})\bar{\bm{u}}_1+Q_{t-1}(\bm{W})\bar{\bm{u}}_2 +Q_{t-2}(\bm{W})\bar{\bm{u}}_3+\cdots +Q_1(\bm{W})\bar{\bm{u}}_t,
\end{split}
\EE
where $\bar{\bm{u}}_1:=\bm{u}_1$. In the above display, $({Q}_t)_{t\ge1}$ is a sequence of polynomials that have zero-mean w.r.t. $\mu$ (i.e., $Q_t(\bm{W})$ is trace-free) and they satisfy certain recursive relationship. 

In our approach, the way to set the de-biasing coefficients is conceptually very simple: \textit{we choose $\paraB{B}_t$ to center the matrices in the representations of $(\bm{r}_t)_{t\ge1}$ as linear combinations of $(\bar{\bm{u}}_t)_{t\ge1}$}. We shall prove in the next subsections that such choice of $\paraB{B}_t$ is unique, and is precisely the one used in RI-AMP.

%--------------------- % ----------------------------%
%--------------------- % ----------------------------%
%--------------------- % ----------------------------%

\subsection{Derivation of Rotationally-Invariant AMP}\label{Sec:main}

In this section, we present our approach to derive the RI-AMP algorithm. The basic idea is to start with a FOM and figure out the correct de-biasing matrix such that the FOM can be recursively reduced to certain OAMP algorithm.

Similar to the strategy used in Section \ref{Sec:overview}, we introduce a set of intermediate variables $(\bar{\bm{u}}_t)_{t\ge1}$. We reformulate the iterates $(\bm{r}_t)_{t\ge1}$ in a FOM as linear combinations of $(\bar{\bm{u}}_t)_{t\ge1}$, which we record in Lemma \ref{Lem:RI_AMP_reformulation} below. We emphasize that the results in this lemma are purely \textit{algebraic}. In particular, the quantities $(\para{d}_{t,i})$ are understood as free parameters in this lemma and need not be the divergences of the de-noisers.

\begin{lemma}[Reformulation of FOM]\label{Lem:RI_AMP_reformulation}
Let $(\bm{u}_t)_{t\ge1}$ and $(\bm{r}_t)_{t\ge1}$ be generated as in \eqref{eq:ZFAMP_general_def}. Let $(\para{d}_{t,i})_{1\le t,1\le i\le t}$ be an arbitrary sequence. Define
\BE\label{Eqn:ut_bar_def}
\bar{\bm{u}}_t:= {\bm{u}}_t-\left(\para{d}_{t-1,1}\cdot\bm{r}_1+\cdots+\para{d}_{t-1,t-1}\cdot\bm{r}_{t-1}\right),\quad\forall t>1,
\EE
and $\bar{\bm{u}}_1:=\bm{u}_1$. Then, $(\bm{r}_t)_{t\ge1}$ can be represented as
\BE\label{Eqn:lemma_reparametrization}
\begin{bmatrix}
\bm{r}_1\\
\bm{r}_2\\
\vdots\\
\bm{r}_t
\end{bmatrix}
=
\begin{bmatrix}
P_{1,1}(\bm{W})&  &   &   \\
P_{2,1}(\bm{W})& P_{2,2}(\bm{W})&   &   \\
\vdots & \vdots & \ddots & \\
P_{t,1}(\bm{W})& P_{t,1}(\bm{W})& \cdots  &  P_{t,t}(\bm{W})
\end{bmatrix}
\begin{bmatrix}
\bar{\bm{u}}_1\\
\bar{\bm{u}}_2\\
\vdots\\
\bar{\bm{u}}_t
\end{bmatrix},
\EE
where $(P_{t,i})_{1\le t,1\le i\le t}$ is a sequence of polynomials. Let $\paraB{P}_t:\mathbb{R}\mapsto\mathbb{R}^{t\times t}$ be a collection of these polynomials:
\BE\label{Eqn:Pt_def_first}
\paraB{P}_t(\lambda):=
\begin{bmatrix}
P_{1,1}(\lambda)&  &   &  \\
P_{2,1}(\lambda)& P_{2,2}(\lambda)&   &   \\
\vdots & \vdots & \ddots & \\
P_{t,1}(\lambda)& P_{t,1}(\lambda)& \cdots  &  P_{t,t}(\lambda)
\end{bmatrix},
\quad\forall \lambda\in\mathbb{R}.
%=\sum_{i=0}^{t-1}(\lambda \para{D}_t-\para{B}_t\para{D}_t)^i\left(\lambda\bm{I}_t-\para{B}_t\right) 
\EE
Then, $\paraB{P}_t(\lambda)$ admits the following explicit expression:
\BE
\paraB{P}_t(\lambda)=(\paraB{I}_t-\lambda\paraB{D}_t+\paraB{B}_t\paraB{D}_t)^{-1}(\lambda\paraB{I}_t-\paraB{B}_t),\quad\forall \lambda\in\mathbb{R}.
\EE
In the above equation, $\paraB{B}_t\in\mathbb{R}^{t\times t}$ denotes the de-biasing matrix \eqref{Eqn:B_t_def} and $\paraB{D}_t\in\mathbb{R}^{t\times t}$ is defined as
\BE\label{Eqn:D_def}
\paraB{D}_t :=
\begin{bmatrix}
0&  &  &  &  \\
\para{d}_{1,1}& 0 &  &  & \\
\para{d}_{2,1}& \para{d}_{2,2} &0 & & \\
\vdots & \vdots & \ddots &  &  \\
\para{d}_{t-1,1} & \para{d}_{t-1,2} & \cdots &\para{d}_{t-1,t-1}  & 0
\end{bmatrix}.
\EE
\end{lemma}

\begin{proof}
See Appendix \ref{App:reformulation}.
\end{proof}

In Lemma \ref{Lem:RI_AMP_reformulation}, we have rewritten the iterate $\bm{r}_t$ as a linear combination of the variables $(\bar{\bm{u}}_1,\ldots,\bar{\bm{u}}_t)$, which we defined in \eqref{Eqn:ut_bar_def}. (The reason why we introduce these new variables will be clear in our later discussions.) Towards reducing the FOM to a OAMP algorithm and as the second step of our derivation, we impose a trace-free constraint on the matrices  $(P_{t,i})_{1\le t,1\le i\le t}$ in \eqref{Eqn:ut_bar_def}, in order to cast the FOM into certain long-memory OAMP algorithm. The following lemma shows that this trace-free constraint naturally yields the form of the de-biasing matrix $\paraB{B}_t$ used in RI-AMP. 

\begin{lemma}[Choice of de-biasing matrix $\paraB{B}_t$]\label{Lem:RI_AMP_zero_trace}
Let $\paraB{B}\in\mathbb{R}^{t\times t}$ and $\paraB{D}_t\in\mathbb{R}^{t\times t}$ be two deterministic matrices defined as in \eqref{Eqn:B_t_def} and \eqref{Eqn:D_def}. Define $\paraB{P}_t:\mathbb{R}\mapsto\mathbb{R}^{t\times t}$ as
\BE\label{Eqn:Pt_inverse_def}
\paraB{P}_t(\lambda):=(\paraB{I}_t-\lambda\paraB{D}_t+\paraB{B}_t\paraB{D}_t)^{-1}(\lambda\paraB{I}_t-\paraB{B}_t),\quad\forall \lambda\in\mathbb{R}.
\EE
Then, the following hold true.
\begin{enumerate}
\item[(1)] For any given $\paraB{D}_t$, the following equation has a unique solution in $\paraB{B}_t$:
\BE\label{Eqn:trace_free_eqn}
\mathbb{E}\left[\paraB{P}_t({\mathsf{\Lambda}})\right]=\bm{0}_{t\times t},\quad \mathsf{\Lambda}\sim\mu.
\EE
\item[(2)] The matrix $\paraB{B}_t$ that solves \eqref{Eqn:trace_free_eqn} can be represented as
\BE\label{Eqn:B_solution}
\paraB{B}_t = \sum_{i=1}^{t}\alpha_i\paraB{D}_t^{i-1},
\EE
where we adopted the convention $\paraB{D}_t^0:=\bm{I}$. Let $(Q_{n})_{n\ge0}$ be a sequence of polynomials recursively defined by ($Q_0(\lambda):=1,\forall\lambda\in\mathbb{R}$):
\BS\label{Eqn:poly_recursive_def}
\BE
Q_n(\lambda)=\lambda Q_{n-1}(\lambda) - \sum_{i=1}^n \mathbb{E}\left[\mathsf{\Lambda}  Q_{i-1}(\mathsf{\Lambda})\right]\cdot Q_{n-i}(\lambda),\quad\forall \lambda\in\mathbb{R},n\ge1.
\EE
The coefficients $(\alpha_n)_{n\ge1}$ in \eqref{Eqn:B_solution} are given by
\BE
\alpha_{n} :=\mathbb{E}\left[\mathsf{\Lambda}  Q_{n-1}(\mathsf{\Lambda})\right],\quad\forall n\ge1.
\EE
\ES
\item[(3)] The matrix defined in \eqref{Eqn:B_solution} can be alternatively represented as
\BE
\paraB{B}_t = \sum_{i=1}^{t}\kappa_i\paraB{D}_t^{i-1},
\EE
where $(\kappa_i)_{i\ge1}$ denotes the free cumulants of $\mu$.
\item[(4)] With $\paraB{B}_t$ given by \eqref{Eqn:B_solution}, $\paraB{P}_t(\lambda) $ can be represented as
\BE
\paraB{P}_t(\lambda) = \sum_{i=1}^{t}Q_i(\lambda)\paraB{D}_t^{i-1}.
\EE 
%where the sequence of polynomials $(P_i)_{i\ge1}$ are defined in \eqref{Eqn:poly_recursive_def}.
\end{enumerate}
\end{lemma}

\begin{proof}
See Appendix \ref{App:zero_trace}.
\end{proof}

\begin{remark}[On the appearance of free cumulants]
Lemma \ref{Lem:RI_AMP_zero_trace} shows that, to make the matrices in the linear representation \eqref{Eqn:lemma_reparametrization} asymptotically trace-free (in the sense of \eqref{Eqn:trace_free_eqn}), the de-biasing matrix $\paraB{B}_t$ is uniquely given by a polynomial of $\paraB{D}_t$, whose coefficients are free cumulants of the limiting spectral law $\mu$. We remark that, in our derivation of RI-AMP, the appearance of the recursion \eqref{Eqn:poly_recursive_def} is natural, while its connection with free cumulants (as established in Proposition \ref{Lem:cumulants}) may not be immediately clear. In Section \ref{Sec:AMP_variant}, we will introduce a variant of RI-AMP  where the Onsager term is a linear combination of all past divergence-free estimates (which may be interpreted as ``essential" non-Gaussian terms). In this variant of RI-AMP, the de-biasing coefficients satisfy a recurrence different from \eqref{Eqn:poly_recursive_def} and has no direct relationship with the free cumulants.
\end{remark}

The above lemma shows that choosing de-biasing matrix $\paraB{B}_t = \sum_{i=1}^{t}\kappa_i\paraB{D}_t^{i-1}$ ensures the matrices $(\bm{P}_{i,j}(\bm{W}))_{1\le i,1\le j\le i}$ to be asymptotically trace-free. To fully specify the FOM, it remains to determine the matrix $\paraB{D}_t$ which appears in the definition of $(\bar{\bm{u}}_t)_{t\ge1}$. In RI-AMP, this matrix is set to be $\paraB{D}_t=\hat{\paraB{\Phi}}_t$, where $\hat{\paraB{\Phi}}_t$ collects the empirical divergences of the de-noisers; see \eqref{Eqn:Phi_def}. Lemma \ref{Lem:FOM_D} below summarizes our new formulation of RI-AMP in terms of the intermediate variables $(\bar{\bm{u}}_t)_{t\ge1}$.

\begin{lemma}[Choice of $\paraB{D}_t$ and final RI-AMP algorithm]\label{Lem:FOM_D}
Let $(\bm{r}_t)_{t\ge1}$ be the iterates generated by the RI-AMP algorithm. With slight abuse of notations, define $(\bar{\bm{u}}_t)_{t\ge1}$ as in \eqref{Eqn:ut_bar_def}, but with $\paraB{D}_t=\hat{\paraB{\Phi}}_t$. Namely, $\bar{\bm{u}}_1:=\bm{u}_1$ and
\BE\label{Eqn:ut_bar_def2}
\bar{\bm{u}}_t:= 
{\bm{u}}_t-\left(\langle\partial_1 \bm{u}_t\rangle\cdot\bm{r}_1+\cdots+\langle\partial_{t-1} \bm{u}_t\rangle \cdot\bm{r}_{t-1}\right), \quad\forall t>1.
\EE
Then, $(\bm{r}_t)_{t\ge1}$ can be represented as
\BE\label{Eqn:lemma_reparametrization_recall}
\begin{bmatrix}
\bm{r}_1\\
\bm{r}_2\\
\vdots\\
\bm{r}_t
\end{bmatrix}
=
\begin{bmatrix}
\hat{P}_{1,1}(\bm{W})& &    &    \\
\hat{P}_{2,1}(\bm{W})& \hat{P}_{2,2}(\bm{W})&    &    \\
\vdots & \vdots & \ddots & \\
\hat{P}_{t,1}(\bm{W})& \hat{P}_{t,1}(\bm{W})& \cdots  &  \hat{P}_{t,t}(\bm{W})
\end{bmatrix}
\begin{bmatrix}
\bar{\bm{u}}_1\\
\bar{\bm{u}}_2\\
\vdots\\
\bar{\bm{u}}_t
\end{bmatrix},
\EE
where $(\hat{P}_{t,i})_{1\le t,1\le i\le t}$ are a sequence of polynomials. Let $\hat{\paraB{P}}_t(\lambda)\in\mathbb{R}^{t\times t}$ be the collection of these polynomials (cf.~\eqref{Eqn:Pt_def_first}). We have
\BE\label{Eqn:Pt_Phi_t_lemma1}
\hat{\paraB{P}}_t(\lambda)=\sum_{i=1}^t Q_i(\lambda)\hat{\paraB{\Phi}}_t^{i-1},\quad\forall \lambda\in\mathbb{R},
\EE
where $(Q_i(\lambda))_{i\ge1}$ are defined recursively
\BE\label{Eqn:Q_poly_recursion_lemma}
Q_n(\lambda)=\lambda Q_{n-1}(\lambda)-\sum_{i=1}^{n}\mathbb{E}[\mathsf{\Lambda} Q_{i-1}(\mathsf{\Lambda})] \cdot Q_{n-i}(\lambda),\quad\forall\lambda\in\mathbb{R},n\ge1,
\EE
under the initialization $ Q_0(\lambda)=1,\forall\lambda\in\mathbb{R}$.
\end{lemma}

\begin{proof}
Note that RI-AMP is an instance of FOM with $\paraB{B}_t = \sum_{i=1}^{t}\kappa_i\hat{\paraB{\Phi}}_t^{i-1}$. The claimed result is a consequence of Lemma \ref{Lem:RI_AMP_reformulation} and Lemma \ref{Lem:RI_AMP_zero_trace}. However, we emphasize one subtle point here. Note that we have assumed $\paraB{B}_t$ and $\paraB{D}_t$ to be deterministic in Lemma \ref{Lem:RI_AMP_zero_trace}, based on which $\mathbb{E}_{\mathsf{\Lambda}\sim\mu}\left[\paraB{P}_t(\mathsf{\Lambda})\right]$ was calculated. On the other hand, when $\paraB{D}_t=\hat{\paraB{\Phi}}_t$, which contains the empirical divergences of the de-noisers, the matrix $\paraB{D}_t$ is no longer deterministic and correlated with $\bm{W}$. Nevertheless, Lemma \ref{Lem:FOM_D} still holds. To see this, note that by Lemma \ref{Lem:RI_AMP_reformulation}, \eqref{Eqn:lemma_reparametrization_recall} holds with 
\BS\label{Eqn:Pt_temp}
\begin{align}
\hat{\paraB{P}}_t(\lambda)&=(\paraB{I}_t-\lambda\hat{\paraB{\Phi}}_t+\paraB{B}_t\hat{\paraB{\Phi}}_t)^{-1}(\lambda\paraB{I}_t-\paraB{B}_t)\\
&=\sum_{i=1}^t\left(\lambda\hat{\paraB{\Phi}}_t-\paraB{B}_t\hat{\paraB{\Phi}}_t\right)^{i-1}(\lambda\paraB{I}_t-\paraB{B}_t),
\end{align}
\ES
where the second step is due to the following identity for a strictly lower triangular matrix $\paraB{A}\in\mathbb{R}^{t\times t}$: $(\paraB{I}_t-\paraB{A})^{-1}=\paraB{I}+\paraB{A}+\paraB{A}^2+\cdots+\paraB{A}^{t-1}$. Now, $\paraB{B}_t=\sum_{i=1}^t \kappa_i\hat{\paraB{\Phi}}^{i-1}_t $, we have
\BE\label{Eqn:Bt_temp}
\paraB{B}_t=\sum_{i=1}^t \kappa_i\hat{\paraB{\Phi}}^{i-1}_t=\sum_{i=1}^t \mathbb{E}[\mathsf{\Lambda} Q_{i-1}(\mathsf{\Lambda})]\cdot \hat{\paraB{\Phi}}^{i-1}_t,
\EE
where $(Q_i(\lambda))_{i\ge1}$ are defined recursively in \eqref{Eqn:Q_poly_recursion_lemma}. The second equality in the above equation is due to the recursive characterization of free cumulants in Lemma \ref{Lem:cumulants}. From \eqref{Eqn:Pt_temp} and \eqref{Eqn:Bt_temp}, $\hat{\paraB{P}}_t(\lambda)$ is a polynomial of $\hat{\paraB{\Phi}}_t$. Applying similar algebraic calculations as the proof of Lemma \ref{Lem:RI_AMP_zero_trace}-(2) leads to the expression of $\hat{\paraB{\Phi}}_t$ \eqref{Eqn:Pt_Phi_t_lemma1} that we aim to derive.
\end{proof}

\begin{remark}[RI-AMP with univariate denoisers]
When the de-noiser $\eta_{t+1}(\cdot)$ only depends on $\bm{r}_t$, the above representation of RI-AMP can be further simplified. In particular, if $\langle \partial_{t}\bm{u}_{t+1}\rangle=1,\forall t\ge1$, the matrix $\hat{\paraB{P}}_t(\lambda)$ becomes (see \eqref{Eqn:Pt_Phi_t_lemma1} and \eqref{Eqn:Phi_def})
\BS
\begin{align*}
\hat{\paraB{P}}_t(\lambda)&=\sum_{i=1}^t Q_i(\lambda)
\begin{bmatrix}
0&  &  &  &  \\
1& 0 &  &  & \\
0& 1 &0 & & \\
\vdots & \vdots & \ddots &  &  \\
0& 0& \cdots &1& 0
\end{bmatrix}^{i-1}
=
\begin{bmatrix}
Q_1(\lambda)&  &  &  &  \\
Q_2(\lambda)& Q_1(\lambda) &  &  & \\
Q_3(\lambda)& Q_2(\lambda) &Q_1(\lambda) & & \\
\vdots & \vdots & \ddots &  &  \\
Q_t(\lambda)& Q_{t-1}(\lambda)& \cdots &Q_2(\lambda)& Q_1(\lambda)
\end{bmatrix}.
\end{align*}
\ES
Hence, the representation of $(\bm{r}_t)_{t\ge1}$ in \eqref{Eqn:lemma_reparametrization_recall} recovers \eqref{Eqn:r1_rt} which we introduced in Section \ref{Sec:overview}.
\end{remark}

Notice that the representation of the iterates $(\bm{r}_t)_{t\ge1}$ of RI-AMP in \eqref{Eqn:lemma_reparametrization_recall} is \textit{deterministic}. At this point, we have rewritten the RI-AMP algorithm as an OAMP algorithm. The asymptotic Gaussian distribution of $(\bm{r}_t)_{t\ge1}$ then follows from the state evolution results of OAMP. Theorem \ref{Th:AMP_SE} below summarizes the state evolution of RI-AMP, which reproduces the state evolution result of RI-AMP \citep{fan2022approximate}, but presented differently.

\begin{theorem}[State evolution of RI-AMP]\label{Th:AMP_SE}
Let $(\bm{r}_t)_{t\ge1}$ be generated by the RI-AMP algorithm. Suppose that Assumption \ref{Ass:RI-AMP} holds. Then, the following holds as $N\to\infty$:
\BS
\BE\label{Eqn:AMP_SE_convergence}
(\bm{r}_1,\ldots,\bm{r}_t)\overset{W_2}{\longrightarrow}(\mathsf{R}_1,\ldots,\mathsf{R}_t)\sim\mathcal{N}(\bm{0}_t,\paraB{\Sigma}_t),\quad \forall t\ge1,
\EE
where 
\BE\label{Eqn:Cov_SE_new}
\paraB{\Sigma}_t = \mathbb{E}_{\mathsf{\Lambda}\sim\mu}\left[{\paraB{P}}_t(\mathsf{\Lambda}) \, \bar{\paraB{\Delta}}_t\,{\paraB{P}}_t(\mathsf{\Lambda})^\UT\right],
\EE
with
\BE\label{Eqn:P_population}
{\paraB{P}}_t(\mathsf{\Lambda}):=\sum_{i=1}^t Q_i(\mathsf{\Lambda}){\paraB{\Phi}}_t^{i-1}.
\EE
In the above equations,
\begin{itemize}
\item The sequence of polynomials $(Q_i)_{i\ge1}$ are defined in \eqref{Eqn:poly_recursive_def}.
\item $\paraB{\Phi}_t$ is defined as in \eqref{Eqn:Phi_def} but with the empirical divergence $\langle\partial_i \bm{u}_j\rangle:=\langle\partial_i \eta_j(\bm{r}_1,\ldots,\bm{r}_{j-1})\rangle$ replaced by $\mathbb{E}\left[\partial_i\eta_j(\mathsf{R}_1,\ldots,\mathsf{R}_{j-1})\right]$, for all $j\in[t]$ and $i\in[j-1]$.
\item $\bar{\paraB{\Delta}}_t$ is defined as
\BE\label{Eqn:Delta_bar_def}
\bar{\paraB{\Delta}}_t :=
\begin{bmatrix}
\mathbb{E}\left[\bar{\mathsf{U}}_1\bar{\mathsf{U}}_1\right]&\mathbb{E}\left[\bar{\mathsf{U}}_1\bar{\mathsf{U}}_2\right]   &  \cdots & \mathbb{E}\left[\bar{\mathsf{U}}_1\bar{\mathsf{U}}_t\right]  \\
\mathbb{E}\left[\bar{\mathsf{U}}_2\bar{\mathsf{U}}_1\right] & \mathbb{E}\left[\bar{\mathsf{U}}_2\bar{\mathsf{U}}_2\right]   &  \cdots &\mathbb{E}\left[\bar{\mathsf{U}}_2\bar{\mathsf{U}}_t\right]  \\
%\langle\partial_1 \bm{u}_3\rangle& \langle\partial_2 \bm{u}_3 \rangle &0 & &\\
\vdots& \vdots &\ddots  & \vdots \\
\mathbb{E}\left[\bar{\mathsf{U}}_t\bar{\mathsf{U}}_1\right] & \mathbb{E}\left[\bar{\mathsf{U}}_t\bar{\mathsf{U}}_2\right]   & \cdots  & \mathbb{E}\left[\bar{\mathsf{U}}_t\bar{\mathsf{U}}_t\right] 
\end{bmatrix},
\EE
where $\bar{\mathsf{U}}_j:=\eta_j(\mathsf{R}_1,\ldots,\mathsf{R}_{j-1})-\sum_{i=1}^{j-1}\mathbb{E}\left[\partial_i\eta_j(\mathsf{R}_1,\ldots,\mathsf{R}_{j-1})\right]\cdot \mathsf{R}_i$, $\forall j\ge2$.
\end{itemize}
\ES
\end{theorem}

\begin{proof}
See Appendix \ref{App:AMP_SE}.
\end{proof}

The state evolution of RI-AMP in Theorem \ref{Th:AMP_SE} is presented in a way that is most natural from the perspective of OAMP, and is in a different format as that given in \citep[Section 4]{fan2022approximate}. The following proposition shows that these two formulations are equivalent.

\begin{proposition}[Consistency with the state evolution in \citep{fan2022approximate}]\label{Pro:equivalence}
The covariance $\paraB{\Sigma}_t$ \eqref{Eqn:Cov_SE_new} can be equivalently written as
\BS\label{Eqn:cov_Fan}
\begin{align}
\paraB{\Sigma}_t = \sum_{j=0}^{\infty}\sum_{i=0}^j \kappa_{j+2}\paraB{\Phi}_t^i\paraB{\Delta}_t\left((\paraB{\Phi}_t)^{j-i}\right)^\UT,
\end{align}
where 
\BE\label{Eqn:Delta_def_Fan}
\paraB{\Delta}_t:=
\begin{bmatrix}
\mathbb{E}\left[{\mathsf{U}}_1{\mathsf{U}}_1\right]&\mathbb{E}\left[{\mathsf{U}}_1{\mathsf{U}}_2\right]   &  \cdots & \mathbb{E}\left[{\mathsf{U}}_1{\mathsf{U}}_t\right]  \\
\mathbb{E}\left[{\mathsf{U}}_2{\mathsf{U}}_1\right] & \mathbb{E}\left[{\mathsf{U}}_2{\mathsf{U}}_2\right]   &  \cdots &\mathbb{E}\left[{\mathsf{U}}_2{\mathsf{U}}_t\right]  \\
%\langle\partial_1 \bm{u}_3\rangle& \langle\partial_2 \bm{u}_3 \rangle &0 & &\\
\vdots& \vdots &\ddots  & \vdots \\
\mathbb{E}\left[{\mathsf{U}}_t{\mathsf{U}}_1\right] & \mathbb{E}\left[{\mathsf{U}}_t{\mathsf{U}}_2\right]   & \cdots  & \mathbb{E}\left[{\mathsf{U}}_t{\mathsf{U}}_t\right] 
\end{bmatrix}
\EE
\ES
with ${\mathsf{U}}_j:=\eta_j(\mathsf{R}_1,\ldots,\mathsf{R}_{j-1})$, and $(\mathsf{R}_1,\ldots,\mathsf{R}_t)$ are the state evolution random variables described in \eqref{Eqn:AMP_SE_convergence}.
\end{proposition}

\begin{proof}
See Appendix \ref{App:SE_equivalence}.
\end{proof}

\section{Generalizations and Applications}\label{Sec:generalizations}

In the preceding section, we demonstrated how our approach can be employed to derive an existing AMP algorithm. In this section, we showcase the versatility of our approach by devising two novel AMP variants. The first variant employs a different form of Onsager term. The second variant allows an additional matrix denoising step, inspired by the recent line of work \cite{barbier2023fundamental,dudeja2024optimality,barbier2024information}. It is possible to produce numerous other variants of AMP, but we do not pursue it further in this paper. Finally, we discuss the applications in spiked models. 

\subsection{Generalization of {RI-AMP}: Different Form of Onsager Terms}\label{Sec:AMP_variant}

In the RI-AMP algorithm, the Onsager term $\para{b}_{t,1}\bm{u}_1+\cdots+\para{b}_{t,t}\bm{u}_t$ is employed to cancel out the non-Gaussian component within $\mW\vu_t $. Our derivation of RI-AMP in Section \ref{sec:AMP_single} indicates that it is possible to achieve the same goal by using an Onsager term that is a linear combination of $(\bar{\bm{u}}_1,\ldots,\bar{\bm{u}}_t)$. Motivated by this insight, we introduce a new variant of RI-AMP, coined {RI-AMP-DF}.

\begin{definition}[A variant of RI-AMP: {RI-AMP-DF}]
Starting from an initialization $\bm{u}_1=\bar{\bm{u}}_1\in\mathbb{R}^N$, {RI-AMP-DF} proceeds as
\BS\label{Eqn:RI-AMP-DF-def}
\begin{align}
\vr_t &= \mW\vu_t - \left(\para{c}_{t,1}\bar{\bm{u}}_1+\para{c}_{t,2}\bar{\bm{u}}_2+\cdots+\para{c}_{t,t}\bar{\bm{u}}_t\right), \quad\forall t\ge1,\\
\vu_{t+1} &= \eta_{t+1} (\vr_1,\ldots,\vr_t),\\
\bar{\bm{u}}_{t+1} &=\eta_{t+1} (\vr_1,\ldots,\vr_t)-\left(\langle\partial_1 \bm{u}_{t+1}\rangle\cdot\bm{r}_1+\cdots+\langle\partial_{t} \bm{u}_{t+1}\rangle \cdot\bm{r}_{t}\right),
\end{align}
\ES
where $\eta_{t+1}:\mathbb{R}^t\mapsto\mathbb{R}$ acts separately on the $N$ rows of $(\bm{r}_1,\ldots,\bm{r}_t)\in\mathbb{R}^{N\times t}$. The de-biasing coefficients are set to
\BE\label{Eqn:C_t_def}
\paraB{C}_t:=
\begin{bmatrix}
\para{c}_{1,1}& & & \\
\para{c}_{2,1}& \para{c}_{2,2} & & \\
\vdots & \vdots & \ddots & \\
\para{c}_{t,1}& \para{c}_{t,2}& \cdots & \para{c}_{t,t}
\end{bmatrix}
=\sum_{i=1}^t \gamma_i (\hat{\paraB{\Phi}}_t)^{i-1},
\EE 
where $\hat{\paraB{\Phi}}_t$ is defined in \eqref{Eqn:Phi_def} and the sequence $(\gamma_n)_{n\in\mathbb{N}}$ are defined as
\BE
\gamma_{n}=\mathbb{E}\left[\mathsf{\Lambda}H_{n-1}(\mathsf{\Lambda})\right],\quad \forall n\ge1, \quad \mathsf{\Lambda}\sim\mu,
\EE
with $(H_i(\lambda))_{i\ge1}$ a sequence of polynomials satisfying the recurrence ($H_0(\lambda):=1$): 
\BE\label{Eqn:variant_poly_recursive_def}
H_{n}(\lambda)=\lambda H_{n-1}(\lambda)-\mathbb{E}\left[\mathsf{\Lambda}H_{n-1}(\mathsf{\Lambda})\right],\quad \forall n\ge1,\quad \lambda\in\mathbb{R}.
\EE
\end{definition}

The de-biasing matrix $\paraB{C}_t$ \eqref{Eqn:C_t_def} is derived using the approach detailed in Section \ref{Sec:main}. Note that both the de-biasing matrix $\paraB{B}_t$ \eqref{Eqn:RI-AMP_B_Phi} for RI-AMP and the de-biasing matrix $\paraB{C}_t$ for {RI-AMP-DF} are polynomials of the divergence matrix $\hat{\paraB{\Phi}}_t$. Moreover, the coefficients in both polynomial representations admit recursive characterization; see \eqref{Eqn:poly_recursive_def} and \eqref{Eqn:variant_poly_recursive_def}.

As before, we can reduce {RI-AMP-DF} to a certain OAMP algorithm. The following theorem summarizes the reduction result.

\begin{theorem}[Reduction of {RI-AMP-DF} to OAMP and state evolution]\label{The:RI-AMP-DF}
Let $(\bm{r}_t)_{t\ge1}$ and $(\bar{\bm{u}}_t)_{t\ge1}$ be generated via the {RI-AMP-DF} algorithm. Then, the following statements hold. 
\begin{itemize}
\item[(1)] For all $t\ge1$,
\BS
\begin{align}
\begin{bmatrix}
\vr_1 \\
\vr_2 \\
\vdots \\
\vr_t
\end{bmatrix} 
&=
\begin{bmatrix}
\hat{G}_{1,1}(\bm{W}) &  &  & \\
\hat{G}_{2,1}(\bm{W}) &\hat{G}_{2,2}(\bm{W})  &  & \\
\vdots & \vdots &\ddots  & \\
\hat{G}_{t,1}(\bm{W}) & \hat{G}_{t,2}(\bm{W}) &\cdots &\hat{G}_{t,t}(\bm{W})
\end{bmatrix}
\begin{bmatrix}
\bar{\vu}_1 \\
\bar{\vu}_2 \\
\vdots \\
\bar{\vu}_t 
\end{bmatrix},
\end{align}
Let $\hat{\paraB{G}}_t(\lambda)$ be the matrix that collects the polynomials $(\hat{G}_{i,j})_{1\le i\le t,1\le j\le i}$ (cf.~\eqref{Eqn:Pt_def_first}).
\ES
Then, $\hat{\paraB{G}}_t(\lambda)$ admits the following representation:
\BE
\hat{\paraB{G}}_t(\lambda) =\sum_{i=1}^t H_i(\lambda)\hat{\paraB{\Phi}}_t^{i-1},\quad\forall \lambda\in\mathbb{R},
\EE
where $(H_n(\lambda))_{n\ge1}$ are defined in \eqref{Eqn:variant_poly_recursive_def}.
\item[(2)] Suppose that Assumption \ref{Ass:RI-AMP} holds. As $N\to\infty$,
\BS
\BE\label{Eqn:AMP-DF_SE_convergence}
(\bm{r}_1,\ldots,\bm{r}_t)\overset{W_2}{\longrightarrow}(\mathsf{R}_1,\ldots,\mathsf{R}_t)\sim\mathcal{N}(\bm{0}_t,\paraB{\Delta}_t),\quad \forall t\ge1,
\EE
where 
% \BE\label{Eqn:Cov-DF_SE_new}
% \paraB{\Delta}_t = \sum_{i=1}^t\sum_{j=1}^t \mathbb{E}_{\mathsf{\Lambda}\sim\mu}\left[H_i(\mathsf{\Lambda})H_j(\mathsf{\Lambda})\right]\cdot \paraB{\Phi}_t^{i-1}\, \bar{\paraB{\Delta}}_t\, (\paraB{\Phi}_t^{j-1})^\UT .
% \EE
\BE\label{Eqn:Cov-DF_SE_new}
\paraB{\Delta}_t = \mathbb{E}_{\mathsf{\Lambda}\sim\mu}\left[{\paraB{G}}_t(\mathsf{\Lambda})\, \bar{\paraB{\Delta}}_t\, {\paraB{G}}_t(\mathsf{\Lambda})^\UT\right]
\EE
with
\BE
{\paraB{G}}_t(\mathsf{\Lambda}) :=\sum_{i=1}^t H_i(\mathsf{\Lambda}){\paraB{\Phi}}_t^{i-1}.
\EE
\ES
In the above equation, $\bar{\paraB{\Delta}}_t$ and $\paraB{\Phi}_t$ are defined similarly as in Theorem \ref{Th:AMP_SE}.
\end{itemize}
\end{theorem}

\begin{proof}
See Appendix \ref{App:RI-AMP-DF-reduction}.
\end{proof}

\begin{remark}[{RI-AMP-DF} can implement any GFOM]
The Onsager terms of {RI-AMP-DF} and the original RI-AMP algorithm are slightly different. However, these two algorithms are essentially equivalent up to a proper change of variables. In fact, both {RI-AMP-DF} and RI-AMP are equivalent to a broader class algorithms, namely, the generalized first order methods (GFOM) introduced in \cite{montanari2022statistically,celentano2020estimation}. More details about this reduction can be found in Appendix \ref{App:RI-AMP-DF-GFOM}.
\end{remark}

\subsection{Generalization of RI-AMP: Matrix Processing}\label{Sec:AMP_varying_spectrum}
Motivated by \cite{barbier2023fundamental}, we introduce another variant of RI-AMP that applies a nonlinear processing function on the matrix $\bm{W}$ in each iteration. In the following, we call this algorithm {RI-AMP-MP}. It is shown in \cite{barbier2023fundamental} that adding the matrix processing operation could lead to better signal estimation for spiked models.

\begin{definition}[A variant of RI-AMP: {RI-AMP-MP}]\label{Def:RI_AMP_MP}
Starting from an initialization $\bm{u}_1\in\mathbb{R}^N$, {RI-AMP-MP} proceeds as follows
\BS\label{Eqn:RI-AMP-MP-def}
\begin{align}
\vr_t &= f_t(\mW)\vu_t - \left(\para{e}_{t,1}\bm{u}_1+\para{e}_{t,2}\bm{u}_2+\cdots+\para{e}_{t,t}\bm{u}_t\right), \\
\vu_{t+1} &= \eta_t (\vr_1,\ldots,\bm{r}_t),
\end{align}
\ES
where $f_t:\mathbb{R}\mapsto\mathbb{R}$ is continuous and applied to the eigenvalues of $\bm{W}$ with the eigenvectors unchanged.  Let $\paraB{E}_t\in\mathbb{R}^{t\times t}$ be a matrix collecting the de-biasing coefficients:
\BE\label{Eqn:E_t_def}
\paraB{E}_t:=
\begin{bmatrix}
\para{e}_{1,1}& & & \\
\para{e}_{2,1}& \para{e}_{2,2} & & \\
\vdots & \vdots & \ddots & \\
\para{e}_{t,1}& \para{e}_{t,2}& \cdots & \para{e}_{t,t}
\end{bmatrix}.
\EE 
The de-biasing matrix $\paraB{E}_t$ will be set according to Lemma \ref{The:RI-AMP-MP} below.
\end{definition}

\vspace{3pt}

We make two remarks about the {RI-AMP-MP} algorithm.
\begin{itemize}
\item[(1)] If the matrix processing function $(f_t)_{t\ge1}$ are identical across different iterations, then {RI-AMP-MP} effectively reduces to an RI-AMP with the original rotationally-invariant matrix replaced by $\hat{\bm{W}}:=f(\bm{W})$. Employing an iteration-dependent matrix processing matrix complicates the derivations of the algorithm. 
\item[(2)] In \cite{barbier2023fundamental}, the matrix processing functions $(f_t)_{t\ge1}$ are assumed to be polynomials, and the Onsager terms and the associated state evolution are derived by mapping the corresponding algorithm to certain RI-AMP algorithm together with re-indexing. Here, we consider general continuous matrix processing functions.
\end{itemize}

The following lemma summarizes the reduction of {RI-AMP-MP} to OAMP. Its proof is similar to that of Lemma \ref{Lem:RI_AMP_reformulation} and Lemma \ref{Lem:RI_AMP_zero_trace}, and thus omitted.

\begin{lemma}[Reduction of {RI-AMP-MP} to OAMP]\label{The:RI-AMP-MP}
Let $(\bm{r}_t)_{t\ge1}$ be generated via the {RI-AMP-MP} algorithm and let $(\bar{\bm{u}}_t)_{t\ge1}$ be defined as in \eqref{Eqn:ut_bar_def2}. Then, the following holds for all $t\ge1$:
\begin{itemize}
\item [(1)] For any fixed $\paraB{E}_t$, we have
\BS\label{Eqn:RI-AMP-MP-OAMP}
\begin{align}
\begin{bmatrix}
\vr_1 \\
\vr_2 \\
\vdots \\
\vr_t
\end{bmatrix} 
&=
\begin{bmatrix}
\hat{J}_{1,1}(\bm{W}) &  &  & \\
\hat{J}_{2,1}(\bm{W}) &\hat{J}_{2,2}(\bm{W})  &  & \\
\vdots & \vdots &\ddots  & \\
\hat{J}_{t,1}(\bm{W}) & \hat{J}_{t,2}(\bm{W}) &\cdots &\hat{J}_{t,t}(\bm{W})
\end{bmatrix}
\begin{bmatrix}
\bar{\vu}_1 \\
\bar{\vu}_2 \\
\vdots \\
\bar{\vu}_t 
\end{bmatrix},
\end{align}
where $(\hat{J}_{i,j})_{1\le i,1\le j\le i}$ is a sequence of functions. Let $\hat{\paraB{J}}_t:\mathbb{R}\mapsto\mathbb{R}^{t\times t}$ be a matrix representation of these functions (similar to \eqref{Eqn:Pt_def_first}).
Then, $\hat{\paraB{J}}_t(\lambda)$ can be written into the following explicit form:
\BE\label{Eqn:J_def}
\hat{\paraB{J}}(\lambda):=
\left(\paraB{I}_t-\mathrm{diag}\{f_1(\lambda),\ldots,f_t(\lambda)\}\hat{\paraB{\Phi}}_t +\paraB{E}_t\hat{\paraB{\Phi}}_t \right)^{-1}\left(\mathrm{diag}\{f_1(\lambda),\ldots,f_t(\lambda)\}-\paraB{E}_t\right),\quad\forall \lambda\in\mathbb{R}.
\EE
\ES
\item[(2)] For any $\hat{\paraB{\Phi}}_t\in\mathbb{R}^{t\times t}$, the following equation has a unique solution in $\paraB{E}_t$:
\BE\label{Eqn:E_equation}
\mathbb{E}\left[\hat{\paraB{J}}(\mathsf{\Lambda})\right]=\bm{0}_{t\times t},
\EE 
where the expectation is taken w.r.t. $\mathsf{\Lambda}\sim\mu$, with $\mathsf{\Lambda}\sim\mu$ independent of $\hat{\paraB{\Phi}}_t$.
\end{itemize}
\end{lemma}

Some comments are in order:
\begin{itemize} 
\item[(1)] It seems that the de-biasing matrix $\paraB{E}_t$ in {RI-AMP-MP}, determined by \eqref{Eqn:E_equation}, cannot be expressed as a simple polynomial of $\hat{\paraB{\Phi}}_t$. While we can solve \eqref{Eqn:E_equation} for $\paraB{E}_t$ recursively (namely, row-by-row), as in the proof of Lemma \ref{Lem:RI_AMP_zero_trace}-(1), the non-commutativity of $\text{diag}\{f_1(\lambda),\ldots,f_t(\lambda)\}$ and $\hat{\paraB{\Phi}}_t$ precludes a simple polynomial form.
\item[(2)] Similar to {RI-AMP-DF}, we may use an Onsager term that is a linear combination of the divergence-free estimates. In this case, the de-biasing matrix can be expressed explicitly, but still cannot be be written as a simple polynomial of $\hat{\paraB{\Phi}}_t$.
\item[(3)] The state evolution equations of {RI-AMP-MP} can be readily derived from the representation \eqref{Eqn:RI-AMP-MP-OAMP}. We omit the details here.
\end{itemize}

\subsection{Application: Spiked Matrix Model}

Until now, we have focused on the derivation of general AMP iterations. In this section, we illustrate the application of our approach to one concrete signal estimation problem, namely, spiked models. AMP algorithms have been instrumental to the theoretical underpinnings of spiked models \cite{rangan2012iterative,parker2014bilinear,kabashima2016phase,montanari2021estimation,li2022non,li2023approximate}. We shall show that our approach can conveniently recover and generalize existing AMP algorithms in the context of spiked models.

Consider a \textit{spiked matrix model} where the observation matrix $\bm{Y}\in\mathbb{R}^{N\times N}$ reads:
\BE\label{Eqn:spiked_model}
\bm{Y}=\frac{\theta}{N}\bm{x}_\star\bm{x}_\star^\UT+\bm{W}.
\EE
In the above equation, $\bm{x}_\star\in\mathbb{R}^N$ is the signal vector to be estimated, $\theta>0$ is a parameter that dictates the SNR, and $\bm{W}$ is a noise matrix which we assume to be symmetric and rotationally-invariant. We first recall the definitions of the OAMP algorithm introduced in \cite{dudeja2024optimality} for spiked models. Note that compared with the OAMP algorithms in \eqref{Eqn:LM_OAMP_recall}, the only difference is that the matrix $\bm{W}$ in \eqref{Eqn:LM_OAMP_recall} is replaced by the observation matrix $\bm{Y}$. 

\begin{definition}[OAMP algorithm for spiked models \cite{dudeja2024optimality}]\label{Def:OAMP_spiked}
Starting from an initialization $\bar{\bm{x}}_1\in\mathbb{R}^N$, an orthogonal AMP algorithm for the spiked model \eqref{Eqn:spiked_model} proceeds as follows:
\BS\label{Eqn:LM_OAMP_spiked}
\begin{align}
\bm{x}_t &= \left(f_t(\bm{Y})-\frac{\text{tr}\left(f_t(\bm{Y})\right)}{N}\cdot\bm{I}_N\right)\bar{\bm{x}}_{t},\\
\bar{\bm{x}}_{t+1} &= g_{t+1} (\bm{x}_1,\ldots,\bm{x}_t;\bm{a})-\sum_{i=1}^t \langle\partial_i  g_{t+1} (\bm{x}_1,\ldots,\bm{x}_t;\bm{a})\rangle\cdot \bm{x}_i,
\end{align}
\ES
where $f_{t}:\mathbb{R}\mapsto\mathbb{R}$ and $g_{t+1}:\mathbb{R}^t\times\mathbb{R}^k\mapsto\mathbb{R}$ are defined similarly as in Definition \ref{Def:LM_OAMP}.
\end{definition}

The above definition of OAMP is slightly different from that in \cite{dudeja2024optimality} in the sense that the empirical estimates $\text{tr}\left(f_t(\bm{Y})\right)/N$ and $\langle\partial_i  g_{t+1} (\bm{z}_1,\ldots,\bm{z}_t;\bm{a})\rangle$ are involved in \eqref{Eqn:LM_OAMP_spiked}, whereas the OAMP in \cite{dudeja2024optimality} uses their corresponding large-$N$ limits. Despite this difference, the state evolution result in \cite{dudeja2024optimality} still applies and we record the results in the following lemma. (This can be shown following the same arguments as in the proof of Theorem \ref{The:OAMP_SE}, and we omit the details.) We also refer the reader to \cite[Section~4.1]{dudeja2024optimality} for a heuristic derivation of the state evolution equations.

\begin{lemma}[State evolution of OAMP for spiked models \cite{dudeja2024optimality}]\label{Lem:OAMP_spiked}
Let $(\bm{x}_t)_{t\in\mathbb{N}}$ be generated via \eqref{Eqn:LM_OAMP_spiked}. Suppose Assumption \ref{Ass:LM_OAMP} holds. Assume additionally that $(\bm{x}_\star;\bm{a})\overset{W_2}{\longrightarrow}(\mathsf{X}_\star;\mathsf{A})\sim\pi$ where $\mathbb{E}[\mathsf{X}_\star^2]=1$ and $\mathbb{E}[\|\mathsf{A}\|^2]<\infty$. Then, the following holds for all $t\ge1$: 
\BS
\BE
(\bm{x}_\star,\bm{x}_1,\ldots,\bm{x}_t;\bm{a})\overset{W_2}{\longrightarrow}(\mathsf{X}_\star,\mathsf{X}_1,\ldots,\mathsf{X}_t;\mathsf{A}),
\EE
where $(\mathsf{X}_\star,\mathsf{A})\sim\pi$ and 
\begin{align}
\left[\mathsf{X}_1,\mathsf{X}_2,\ldots,\mathsf{X}_t\right]^\UT
=
\paraB{\beta}_t
\mathsf{X}_\star+
\left[\mathsf{Z}_1,\mathsf{Z}_2,\ldots,\mathsf{Z}_t\right]^\UT,
\end{align}
where $[\mathsf{Z}_1,\ldots,\mathsf{Z}_t]\sim\mathcal{N}(\bm{0},\paraB{\Omega}_{t,1}+\paraB{\Omega}_{t,2})$ is independent of $\mathsf{X}_\star$, and $(\paraB{\beta}_t,\paraB{\Omega}_{t,1},\paraB{\Omega}_{t,2})$ are given by
\begin{align}
\paraB{\beta}_t&=
\begin{bmatrix}
\mathbb{E}\left[\bar{f}_1(\mathsf{\Lambda}_\nu)\right]\\
%\mathbb{E}\left[\bar{f}_2(\mathsf{\Lambda}_\nu)\right]\\
\vdots\\
\mathbb{E}\left[\bar{f}_t(\mathsf{\Lambda}_\nu)\right]
\end{bmatrix}\circ
\paraB{\alpha}_t,\\
\paraB{\Omega}_{t,1}&=
\begin{bmatrix}
\mathrm{Cov}_\nu\left[\bar{f}_1,\bar{f}_1\right]   &  \cdots & \mathrm{Cov}_\nu\left[\bar{f}_1,\bar{f}_t\right]    \\
%\mathrm{Cov}_\nu\left[\bar{f}_2,\bar{f}_1\right]    &  \cdots & \mathrm{Cov}_\nu\left[\bar{f}_2,\bar{f}_t\right]   \\
\vdots& \ddots  & \vdots \\
\mathrm{Cov}_\nu\left[\bar{f}_t,\bar{f}_1\right]   &  \cdots & \mathrm{Cov}_\nu\left[\bar{f}_t,\bar{f}_t\right]   
\end{bmatrix}
\circ(\paraB{\alpha}_t\paraB{\alpha}_t^\UT),\\
\paraB{\Omega}_{t,2}&=
\begin{bmatrix}
\mathrm{Cov}_\mu\left[\bar{f}_1,\bar{f}_1\right]   &  \cdots & \mathrm{Cov}_\mu\left[\bar{f}_1,\bar{f}_t\right]    \\
%\mathrm{Cov}_\mu\left[\bar{f}_2,\bar{f}_1\right]    &  \cdots & \mathrm{Cov}_\mu\left[\bar{f}_2,\bar{f}_t\right]   \\
\vdots& \ddots  & \vdots \\
\mathrm{Cov}_\mu\left[\bar{f}_t,\bar{f}_1\right]   &  \cdots & \mathrm{Cov}_\mu\left[\bar{f}_t,\bar{f}_t\right]   
\end{bmatrix}
\circ(\bar{\paraB{\Delta}}_t-\paraB{\alpha}_t\paraB{\alpha}_t^\UT),
\end{align}
where $\bar{f}_i(\lambda):=f_i(\lambda)-\mathbb{E}_{\mathsf{\Lambda}\sim\mu}[\mathsf{\Lambda}],\forall \lambda\in\mathbb{R},i\in[t]$ and $\mathrm{Cov}_\mu\left[\bar{f}_i,\bar{f}_j\right]$ is defined as in Theorem \ref{The:OAMP_SE}, and $(\paraB{\alpha}_t,\bar{\paraB{\Delta}}_t)$ are defined as
\BE
\paraB{\alpha}_t =
\begin{bmatrix}
\mathbb{E}\left[\mathsf{X}_\star\bar{\mathsf{X}}_1\right]\\
%\mathbb{E}\left[\bar{f}_2(\mathsf{\Lambda}_\nu)\right]\\
\vdots\\
\mathbb{E}\left[\mathsf{X}_\star\bar{\mathsf{X}}_t\right]
\end{bmatrix}\quad \text{and}\quad
\bar{\paraB{\Delta}}_t =
\begin{bmatrix}
\mathbb{E}\left[\bar{\mathsf{X}}_1\bar{\mathsf{X}}_1\right]   &  \cdots & \mathbb{E}\left[\bar{\mathsf{X}}_1\bar{\mathsf{X}}_t\right]  \\
%\langle\partial_1 \bm{u}_3\rangle& \langle\partial_2 \bm{u}_3 \rangle &0 & &\\
\vdots& \ddots  & \vdots \\
\mathbb{E}\left[\bar{\mathsf{X}}_t\bar{\mathsf{X}}_1\right]  & \cdots  & \mathbb{E}\left[\bar{\mathsf{X}}_t\bar{\mathsf{X}}_t\right] 
\end{bmatrix}.
\EE
In the above equations, $\nu$ is a probability measure whose Stieltjes transform (defined as $m_\nu(z):=\int_{\mathbb{R}}\frac{\nu(\mathrm{d}\lambda)}{z-\lambda},\forall z\in\mathbb{C}\backslash\mathbb{R}$) is given by:
\BE
m_\nu(z)=\frac{m_\mu(z)}{1-\theta m_\mu(z)},\quad\forall z\in\mathbb{C}\backslash\mathbb{R}.
\EE
Further,
\BE
\bar{\mathsf{X}}_{t}:=g_{t}(\mathsf{X}_1,\ldots,\mathsf{X}_{t-1};\mathsf{A})-\sum_{i=1}^{t-1}\mathbb{E}\left[\partial_i g_{t}(\mathsf{X}_1,\ldots,\mathsf{X}_{t-1};\mathsf{A})\right]\cdot \mathsf{X}_i,\quad\forall t\ge2.
\EE
\ES
\end{lemma}

\begin{remark}[On the measure $\nu$] 
The state evolution equations in the above theorem involve a probability measure $\nu$. As shown in \cite{dudeja2024optimality}, the measure $\nu$ is the large-$N$ limit of the following empirical measure:
\BE
\nu_N:=\frac{1}{N}\sum_{i=1}^N \left(\bm{x}_\star^\UT\bm{u}_i(\bm{Y})\right)^2\cdot \delta_{\lambda_i(\bm{Y})},
\EE
where $(\lambda_i(\bm{Y}))_{1\le i\le N}$ and $(\bm{u}_i(\bm{Y}))_{1\le i\le N}$ denote the eigenvalues and eigenvectors of $\bm{Y}$ respectively. For more explicit characterization of $\nu$ by Stieltjes inversion see \cite[Appendix A.1]{dudeja2024optimality}.
\end{remark}

We can now apply our approach to derive AMP algorithms by reduction to OAMP algorithms. We consider a variant of {RI-AMP-MP} where the matrix processing function does not change across iterations. 

\begin{definition}[{RI-AMP-MP} algorithm for spiked models]
Starting from an initialization $\bm{u}_1\in\mathbb{R}^N$, {RI-AMP-MP} generates a sequence of iterates via
\BS\label{Eqn:RI-AMP-MP-spiked}
\begin{align}
\vr_t &= f(\bm{Y})\vu_t - \left(\para{e}_{t,1}\bm{u}_1+\para{e}_{t,2}\bm{u}_2+\cdots+\para{e}_{t,t}\bm{u}_t\right),\\
\vu_{t+1} &= \eta_t (\vr_1,\ldots,\bm{r}_t),
\end{align}
\ES
where $f:\mathbb{R}\mapsto\mathbb{R}$ is a continuous function that does not change across iterations, and the de-biasing matrix $\paraB{E}_t\in\mathbb{R}^{t\times t}$ is given by
\BE
\paraB{E}_t=\sum_{i=1}^t \tilde{\alpha}_i (\hat{\paraB{\Phi}}_t)^{i-1},
\EE
where $\hat{\paraB{\Phi}}_t$ is defined as in \eqref{Eqn:Phi_def} and $(\tilde{\alpha}_i)_{i\in\mathbb{N}}$ are the free cumulants of the random variable $f(\mathsf{\Lambda})$, $\mathsf{\Lambda}\sim\mu$.
\end{definition}

\begin{remark}[Connection with \cite{barbier2023fundamental}]
The above algorithm is essentially the ``Bayes-optimal AMP'' (BAMP) algorithm proposed in \cite{barbier2023fundamental}. It generalizes BAMP in the sense that the matrix processing function $f(\cdot)$ can be non-polynomial. 
\end{remark}

Again, we can use the same technique to reduce the above {RI-AMP-MP} algorithm to an OAMP algorithm, and then apply Lemma \ref{Lem:OAMP_spiked} to derive a state evolution for RI-AMP-MP. Our results are summarized in the following theorem. 

\begin{theorem}[State evolution of {RI-AMP-MP} for spiked models]\label{The:RI-AMP-MP-SE}
Let $(\bm{r}_t)_{t\ge1}$ be generated via \eqref{Eqn:RI-AMP-MP-spiked} and let $(\bar{\bm{u}}_t)_{t\ge1}$ be defined as in \eqref{Eqn:ut_bar_def2}. Then, the following statements hold. 
\begin{itemize}
\item[(1)] For all $t\ge1$:
\BS\label{Eqn:RI-AMP-MP-spiked-reduction}
\begin{align}
\begin{bmatrix}
\vr_1 \\
\vr_2 \\
\vdots \\
\vr_t
\end{bmatrix} 
&=
\begin{bmatrix}
\hat{J}_{1,1}(\bm{Y}) &  &  & \\
\hat{J}_{2,1}(\bm{Y}) &\hat{J}_{2,2}(\bm{Y})  &  & \\
\vdots & \vdots &\ddots  & \\
\hat{J}_{t,1}(\bm{Y}) & \hat{J}_{t,2}(\bm{Y}) &\cdots &\hat{J}_{t,t}(\bm{Y})
\end{bmatrix}
\begin{bmatrix}
\bar{\vu}_1 \\
\bar{\vu}_2 \\
\vdots \\
\bar{\vu}_t 
\end{bmatrix}.
\end{align}
Let $\hat{\paraB{J}}_t:\mathbb{R}\mapsto\mathbb{R}^{t\times t}$ be a matrix representation of the functions $(\hat{J}_{i,j})_{1\le i\le t,1\le j\le i}$. Then, $\hat{\paraB{J}}_t(\lambda)$ can be written into the following explicit form:
\BE\label{Eqn:Pt_Phi_t_lemma}
\hat{\paraB{J}}_t(\lambda)=\sum_{i=1}^t K_i(\lambda)\hat{\paraB{\Phi}}_t^{i-1},\quad\forall \lambda\in\mathbb{R},
\EE
where the sequence of functions $(K_n)_{n\ge1}$ is specified by the recurrence (with $ K_0(\lambda):=1,\forall\lambda\in\mathbb{R}$)
\BE\label{Eqn:poly_spiked_recursive}
K_n(\lambda)=f(\lambda) K_{n-1}(\lambda) - \sum_{i=1}^n \mathbb{E}_{\mathsf{\Lambda}\sim\mu}\left[f(\mathsf{\Lambda})  K_{i-1}(\mathsf{\Lambda})\right]\cdot K_{n-i}(\lambda),\quad\forall \lambda\in\mathbb{R},\ n\ge1.
\EE
\ES
\item[(2)] Suppose that Assumption \ref{Ass:RI-AMP} holds. Further, assume that $\bm{x}_\star\overset{W_2}{\longrightarrow} \mathsf{X}_\star$ where $\mathbb{E}[\mathsf{X}_\star^2]=1$. As $N\to\infty$:
\BS
\BE
(\bm{x}_\star,\bm{r}_1,\ldots,\bm{r}_t)\overset{W_2}{\longrightarrow}(\mathsf{X}_\star,\mathsf{R}_1,\ldots,\mathsf{R}_t),\quad\forall t\ge1,
\EE
where
\begin{align}
\left[\mathsf{R}_1,\mathsf{R}_2,\ldots,\mathsf{R}_t\right]^\UT
=
\paraB{\beta}_t
\mathsf{X}_\star+
\left[\mathsf{Z}_1,\mathsf{Z}_2,\ldots,\mathsf{Z}_t\right]^\UT,
\end{align}
where $[\mathsf{Z}_1,\ldots,\mathsf{Z}_t]\sim\mathcal{N}(\bm{0},\paraB{\Sigma}_{t,1}+\paraB{\Sigma}_{t,2})$ is independent of $\mathsf{X}_\star$, and $(\paraB{\beta}_t,\paraB{\Sigma}_{t,1},\paraB{\Sigma}_{t,2})$ are given by
\begin{align}
\paraB{\beta}_t&=
\mathbb{E}_{\mathsf{\Lambda}_\nu\sim\nu}\left[\paraB{J}_t(\mathsf{\Lambda}_\nu)\right]
\paraB{\alpha}_t,\\
\paraB{\Sigma}_{t,1}&=
\mathbb{E}_{\mathsf{\Lambda}_\nu\sim\nu}\left[\paraB{J}_t(\mathsf{\Lambda}_\nu)\paraB{\alpha}_t\paraB{\alpha}_t^\UT\paraB{J}_t(\mathsf{\Lambda}_\nu)^\UT\right]-\paraB{\beta}_t\paraB{\beta}_t^\UT,\\
\paraB{\Sigma}_{t,2}&=
\mathbb{E}_{\mathsf{\Lambda}\sim\mu}\left[\paraB{J}_t(\mathsf{\Lambda})(\bar{\paraB{\Delta}}_t-\paraB{\alpha}_t\paraB{\alpha}_t^\UT)\paraB{J}_t(\mathsf{\Lambda})^\UT\right],
\end{align}
where ${\paraB{J}}_t(\lambda):=\sum_{i=1}^t K_i(\lambda){\paraB{\Phi}}_t^{i-1},\,\forall \lambda\in\mathbb{R}$, the probability measure $\nu$ is defined in Lemma \ref{Lem:OAMP_spiked}, and $(\paraB{\alpha}_t,\bar{\paraB{\Delta}}_t)$ are defined as
\BE
\paraB{\alpha}_t =
\begin{bmatrix}
\mathbb{E}\left[\mathsf{X}_\star\bar{\mathsf{X}}_1\right]\\
%\mathbb{E}\left[\bar{f}_2(\mathsf{\Lambda}_\nu)\right]\\
\vdots\\
\mathbb{E}\left[\mathsf{X}_\star\bar{\mathsf{X}}_t\right]
\end{bmatrix}\quad \text{and}\quad
\bar{\paraB{\Delta}}_t =
\begin{bmatrix}
\mathbb{E}\left[\bar{\mathsf{X}}_1\bar{\mathsf{X}}_1\right]   &  \cdots & \mathbb{E}\left[\bar{\mathsf{X}}_1\bar{\mathsf{X}}_t\right]  \\
%\langle\partial_1 \bm{u}_3\rangle& \langle\partial_2 \bm{u}_3 \rangle &0 & &\\
\vdots& \ddots  & \vdots \\
\mathbb{E}\left[\bar{\mathsf{X}}_t\bar{\mathsf{X}}_1\right]  & \cdots  & \mathbb{E}\left[\bar{\mathsf{X}}_t\bar{\mathsf{X}}_t\right] 
\end{bmatrix}.
\EE
\ES
\end{itemize}
\end{theorem}

\begin{proof}
Claim (1) can be proved analogously to Lemma \ref{Lem:FOM_D}, with $\bm{W}$ replaced by $\bm{Y}$. Note that the reduction of RI-AMP to OAMP as presented in Lemma \ref{Lem:FOM_D} is a deterministic result, and still holds with $\bm{W}$ replaced by $\bm{Y}$. The proof of Claim (2) is similar to that of Theorem \ref{Th:AMP_SE}, the only difference being that we now appeal to Lemma \ref{Lem:OAMP_spiked} (state evolution of OAMP for spiked models). 
\end{proof}

\section{Numerical Results}\label{Sec:numerical}

We conduct a few numerical experiments to demonstrate the performance of the RI-AMP-MP algorithm (Definition \ref{Def:RI_AMP_MP}) and show the accuracy of its theoretical state evolution prediction. Note that RI-AMP-MP generalizes the BAMP algorithm proposed in \cite{barbier2023fundamental} in the sense that the matrix denoising function $f(\cdot)$ can be non-polynomial. For polynomial $f(\cdot)$, the de-biasing coefficients in BAMP algorithm and the RI-AMP-MP algorithm are slightly different (the former is based on reduction to certain RI-AMP \cite{fan2022approximate} with proper time re-indexing), but equivalent asymptotically.

Fig.~\ref{Fig:MSE_vs_Iteration} shows the empirical and theoretical mean square error (MSE) performances of RI-AMP-MP. We set $f(\cdot)$ using the function proposed in \cite{barbier2023fundamental,barbier2024information}. The theoretical state evolution prediction is based on the formula given in Theorem \ref{The:RI-AMP-MP-SE}. In sub-figure (a), the eigenvalue distribution $\mu$ corresponds to the \textit{trace ensemble} \cite{barbier2023fundamental} with pure quartic potential. Both BAMP and RI-AMP-MP use the single-variate MMSE denoiser. For this setup, the function $f(\cdot)$ is a polynomial and RI-AMP-MP is asymptotically equivalent to BAMP. Indeed, as shown in Fig.~\ref{Fig:MSE_vs_Iteration}-(a), the performances of BAMP and RI-AMP-MP are very close and are close to state evolution prediction.

Fig.~\ref{Fig:MSE_vs_Iteration}-(b) considers a setup where $f(\cdot)$ is non-polynomial and BAMP is not directly applicable. (In principle, it is possible to approximate $f(\cdot)$ using polynomials within BAMP, but the Onager term and the state evolution become somewhat cumbersome as the degree increases.) For experiment purposes, we set $\mu$ to the Marcenko–Pastur distribution to
    \begin{align}
        \mu(\lambda) &= \frac{1}{2\pi } \frac{ \sqrt{( a_+ -\lambda)( \lambda - a_- ) } }{\alpha \lambda}, \quad a_+ := (1+\sqrt{\alpha})^2, \;  a_- := (1-\sqrt{\alpha})^2,
    \end{align}
    where $\alpha>0$ is a parameter (which is set to $\alpha=0.2$ in our experiement). Following \cite{barbier2024information}, we set the matrix denoising function $f(\cdot)$ as
    \BE
    f(\lambda) = \frac{\theta}{\alpha}\left(1 +\frac{\alpha-1}{\lambda}\right) - \frac{\theta^2}{\alpha\lambda}.
    \EE
In the above equation, $\theta>0$ denotes the SNR parameter for the spiked model \eqref{Eqn:spiked_model}. The signal denoisers $\eta_t(\cdot)$ uses the function in \cite[Remark 3.3]{fan2022approximate}, namely, optimal linear combining followed by the univariate MMSE function. Fig.~\ref{Fig:MSE_vs_Iteration}-(b) confirms that our theoretical state evolution prediction is still quite accurate for this setup.

\begin{figure}[htbp]
\begin{center}
\subfloat[Pure quartic distribution.\label{Fig:quartic}]{
\includegraphics[width=0.45\linewidth]{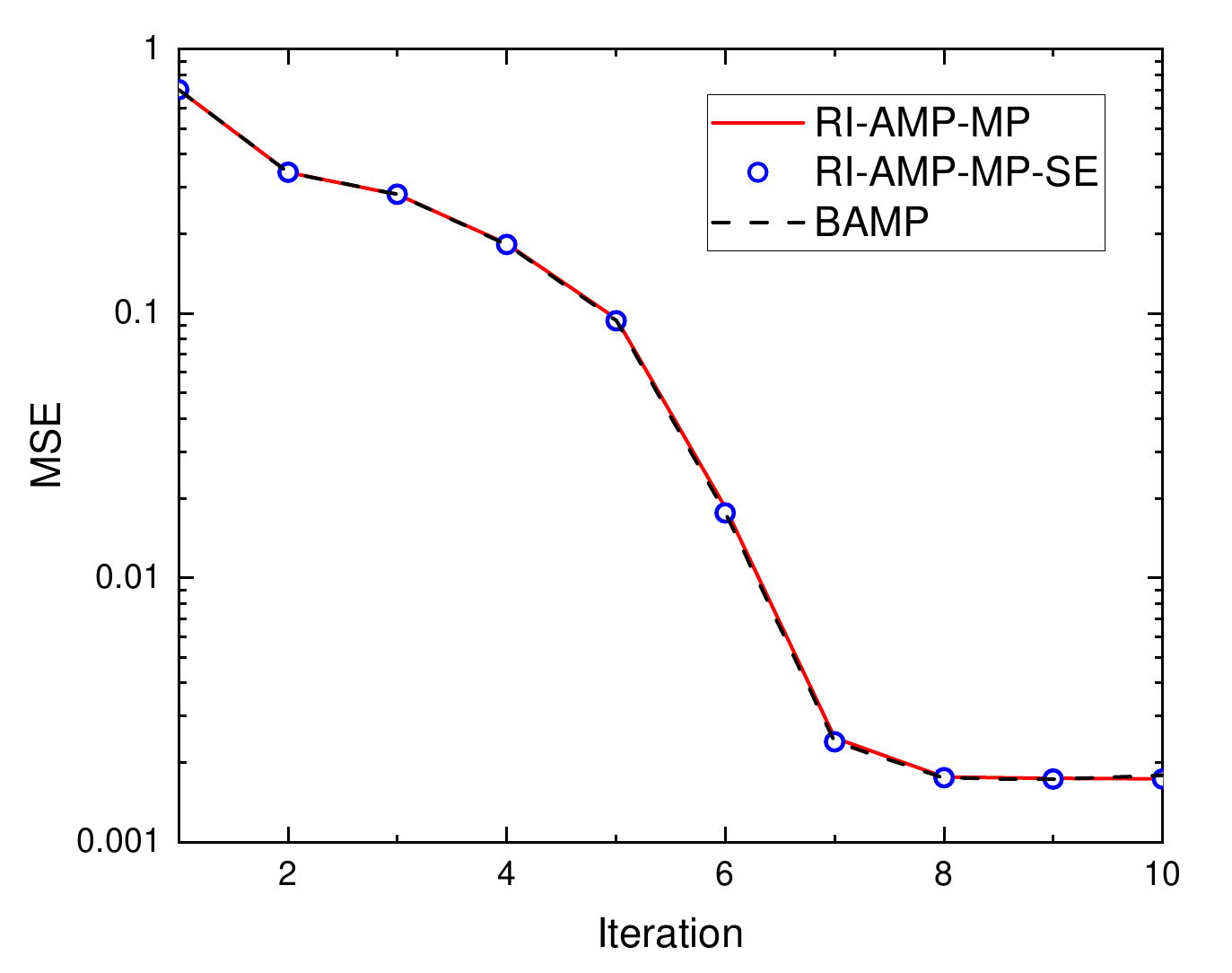}}
\subfloat[Marcenko–Pastur distribution. \label{Fig:MPlaw}]{
\includegraphics[width=0.45\linewidth]{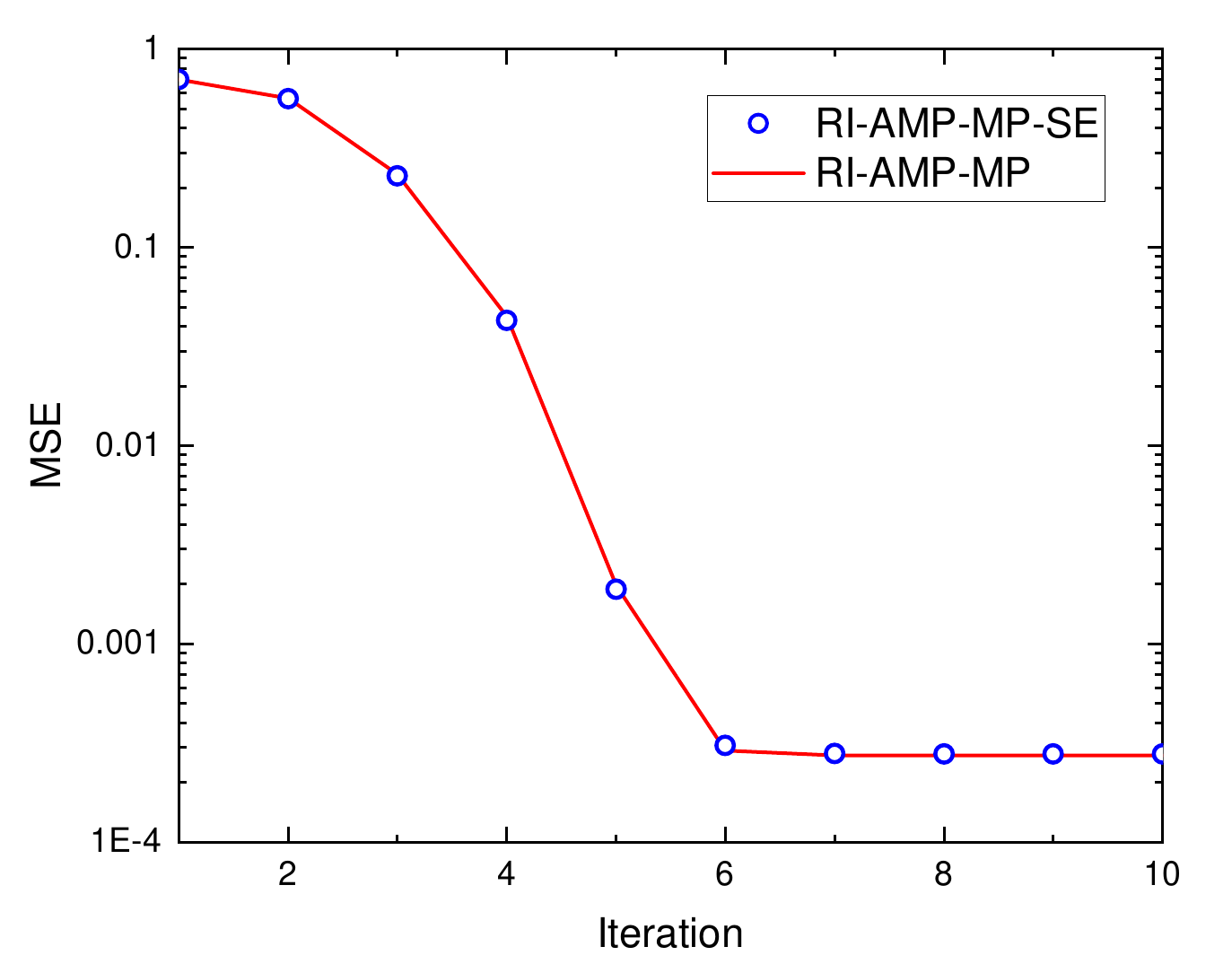} }
\end{center}
\caption{MSE performance of the RI-AMP-MP algorithm and the BAMP algorithm \citep{barbier2023fundamental}. In both experiments, we initialize the algorithms by $ \bm{u}_1 = \sqrt{\omega} \bm{x}_\star + \sqrt{(1-\omega)} \bm{n} $ where $\bm{n}$ is standard Gaussian independent of $\bm{x}_\star$ and $\omega = 0.3$. The empirical results are average over 50 independent runs. $N=5000$. The SNR parameter is $\theta=3$ for figure (a) and $\theta=1.5$ for figure (b).}
\label{Fig:MSE_vs_Iteration}
\end{figure}

\section*{Acknowledgement}
We are grateful to Rishabh Dudeja for many insightful conversations and comments that substantially impacted this work. We also thank Zhou Fan and Roland Speicher for helpful feedback about free cumulants. The second author would like to acknowledge Arian Maleki for stimulating discussions that motivate this work and Jizu Huang for discussions about Lemma \ref{Lem:RI_AMP_zero_trace}.

\bibliographystyle{plainnat}
\bibliography{refs,refs_PCA,Phase_retrieval3}

\begin{thebibliography}{60}
\providecommand{\natexlab}[1]{#1}
\providecommand{\url}[1]{\texttt{#1}}
\expandafter\ifx\csname urlstyle\endcsname\relax
  \providecommand{\doi}[1]{doi: #1}\else
  \providecommand{\doi}{doi: \begingroup \urlstyle{rm}\Url}\fi

\bibitem[Anderson et~al.(2010)Anderson, Guionnet, and
  Zeitouni]{anderson2010introduction}
Greg~W Anderson, Alice Guionnet, and Ofer Zeitouni.
\newblock \emph{An introduction to random matrices}.
\newblock Number 118. Cambridge university press, 2010.

\bibitem[Barbier et~al.(2019)Barbier, Krzakala, Macris, Miolane, and
  Zdeborov{\'a}]{Barbier2019}
Jean Barbier, Florent Krzakala, Nicolas Macris, L{\'e}o Miolane, and Lenka
  Zdeborov{\'a}.
\newblock Optimal errors and phase transitions in high-dimensional generalized
  linear models.
\newblock \emph{Proceedings of the National Academy of Sciences}, 116\penalty0
  (12):\penalty0 5451--5460, 2019.

\bibitem[Barbier et~al.(2023)Barbier, Camilli, Mondelli, and
  S{\'a}enz]{barbier2023fundamental}
Jean Barbier, Francesco Camilli, Marco Mondelli, and Manuel S{\'a}enz.
\newblock Fundamental limits in structured principal component analysis and how
  to reach them.
\newblock \emph{Proceedings of the National Academy of Sciences}, 120\penalty0
  (30):\penalty0 e2302028120, 2023.

\bibitem[Barbier et~al.(2024)Barbier, Camilli, Mondelli, and
  Xu]{barbier2024information}
Jean Barbier, Francesco Camilli, Marco Mondelli, and Yizhou Xu.
\newblock Information limits and {Thouless-Anderson-Palmer} equations for
  spiked matrix models with structured noise.
\newblock \emph{arXiv preprint arXiv:2405.20993}, 2024.

\bibitem[Bayati and Montanari(2011{\natexlab{a}})]{bayati2011dynamics}
Mohsen Bayati and Andrea Montanari.
\newblock The dynamics of message passing on dense graphs, with applications to
  compressed sensing.
\newblock \emph{IEEE Transactions on Information Theory}, 57\penalty0
  (2):\penalty0 764--785, 2011{\natexlab{a}}.

\bibitem[Bayati and Montanari(2011{\natexlab{b}})]{bayati2011lasso}
Mohsen Bayati and Andrea Montanari.
\newblock The {LASSO} risk for gaussian matrices.
\newblock \emph{IEEE Transactions on Information Theory}, 58\penalty0
  (4):\penalty0 1997--2017, 2011{\natexlab{b}}.

\bibitem[Bolthausen(2014)]{bolthausen2014iterative}
Erwin Bolthausen.
\newblock An iterative construction of solutions of the {TAP} equations for the
  sherrington--kirkpatrick model.
\newblock \emph{Communications in Mathematical Physics}, 325\penalty0
  (1):\penalty0 333--366, 2014.

\bibitem[Bu et~al.(2020)Bu, Klusowski, Rush, and Su]{bu2020algorithmic}
Zhiqi Bu, Jason~M Klusowski, Cynthia Rush, and Weijie~J Su.
\newblock Algorithmic analysis and statistical estimation of {SLOPE} via
  approximate message passing.
\newblock \emph{IEEE Transactions on Information Theory}, 67\penalty0
  (1):\penalty0 506--537, 2020.

\bibitem[Cademartori and Rush(2024)]{cademartori2024non}
Collin Cademartori and Cynthia Rush.
\newblock A non-asymptotic analysis of generalized vector approximate message
  passing algorithms with rotationally invariant designs.
\newblock \emph{IEEE Transactions on Information Theory}, 2024.

\bibitem[Cakmak et~al.(2014)Cakmak, Winther, and Fleury]{cakmak2014s}
Burak Cakmak, Ole Winther, and Bernard~H Fleury.
\newblock S-amp: Approximate message passing for general matrix ensembles.
\newblock In \emph{2014 IEEE Information Theory Workshop (ITW 2014)}, pages
  192--196. IEEE, 2014.

\bibitem[Celentano et~al.(2020)Celentano, Montanari, and
  Wu]{celentano2020estimation}
Michael Celentano, Andrea Montanari, and Yuchen Wu.
\newblock The estimation error of general first order methods.
\newblock In \emph{Conference on Learning Theory}, pages 1078--1141. PMLR,
  2020.

\bibitem[Cheng et~al.(2020)Cheng, Liu, and Ping]{cheng2020orthogonal}
Yiyao Cheng, Lei Liu, and Li~Ping.
\newblock Orthogonal {AMP} for massive access in channels with spatial and
  temporal correlations.
\newblock \emph{IEEE Journal on Selected Areas in Communications}, 39\penalty0
  (3):\penalty0 726--740, 2020.

\bibitem[Donoho et~al.(2009)Donoho, Maleki, and Montanari]{donoho2009message}
David~L Donoho, Arian Maleki, and Andrea Montanari.
\newblock Message-passing algorithms for compressed sensing.
\newblock \emph{Proceedings of the National Academy of Sciences}, 106\penalty0
  (45):\penalty0 18914--18919, 2009.

\bibitem[Donoho et~al.(2011)Donoho, Maleki, and Montanari]{donoho2011noise}
David~L Donoho, Arian Maleki, and Andrea Montanari.
\newblock The noise-sensitivity phase transition in compressed sensing.
\newblock \emph{IEEE Transactions on Information Theory}, 57\penalty0
  (10):\penalty0 6920--6941, 2011.

\bibitem[Donoho et~al.(2013)Donoho, Javanmard, and
  Montanari]{donoho2013information}
David~L Donoho, Adel Javanmard, and Andrea Montanari.
\newblock Information-theoretically optimal compressed sensing via spatial
  coupling and approximate message passing.
\newblock \emph{IEEE transactions on information theory}, 59\penalty0
  (11):\penalty0 7434--7464, 2013.

\bibitem[Dudeja and Bakhshizadeh(2022)]{dudeja2022universality}
Rishabh Dudeja and Milad Bakhshizadeh.
\newblock Universality of linearized message passing for phase retrieval with
  structured sensing matrices.
\newblock \emph{IEEE Transactions on Information Theory}, 68\penalty0
  (11):\penalty0 7545--7574, 2022.

\bibitem[Dudeja et~al.(2023)Dudeja, M.~Lu, and Sen]{dudeja2023universality}
Rishabh Dudeja, Yue M.~Lu, and Subhabrata Sen.
\newblock Universality of approximate message passing with semirandom matrices.
\newblock \emph{The Annals of Probability}, 51\penalty0 (5):\penalty0
  1616--1683, 2023.

\bibitem[Dudeja et~al.(2024{\natexlab{a}})Dudeja, Liu, and
  Ma]{dudeja2024optimality}
Rishabh Dudeja, Songbin Liu, and Junjie Ma.
\newblock Optimality of approximate message passing algorithms for spiked
  matrix models with rotationally invariant noise.
\newblock \emph{arXiv preprint arXiv:2405.18081}, 2024{\natexlab{a}}.

\bibitem[Dudeja et~al.(2024{\natexlab{b}})Dudeja, Sen, and
  Lu]{dudeja2022spectral}
Rishabh Dudeja, Subhabrata Sen, and Yue~M Lu.
\newblock Spectral universality in regularized linear regression with nearly
  deterministic sensing matrices.
\newblock \emph{IEEE Transactions on Information Theory}, 2024{\natexlab{b}}.

\bibitem[El~Alaoui et~al.(2021)El~Alaoui, Montanari, and
  Sellke]{el2021optimization}
Ahmed El~Alaoui, Andrea Montanari, and Mark Sellke.
\newblock Optimization of mean-field spin glasses.
\newblock \emph{The Annals of Probability}, 49\penalty0 (6):\penalty0
  2922--2960, 2021.

\bibitem[Fan(2022)]{fan2022approximate}
Zhou Fan.
\newblock Approximate message passing algorithms for rotationally invariant
  matrices.
\newblock \emph{The Annals of Statistics}, 50\penalty0 (1):\penalty0 197--224,
  2022.

\bibitem[Feng et~al.(2022)Feng, Venkataramanan, Rush, Samworth,
  et~al.]{feng2022unifying}
Oliver~Y Feng, Ramji Venkataramanan, Cynthia Rush, Richard~J Samworth, et~al.
\newblock A unifying tutorial on approximate message passing.
\newblock \emph{Foundations and Trends{\textregistered} in Machine Learning},
  15\penalty0 (4):\penalty0 335--536, 2022.

\bibitem[Fletcher et~al.(2018)Fletcher, Rangan, and
  Schniter]{fletcher2018inference}
Alyson~K Fletcher, Sundeep Rangan, and Philip Schniter.
\newblock Inference in deep networks in high dimensions.
\newblock In \emph{2018 IEEE International Symposium on Information Theory
  (ISIT)}, pages 1884--1888. IEEE, 2018.

\bibitem[Kabashima(2003)]{kabashima2003cdma}
Yoshiyuki Kabashima.
\newblock A {CDMA} multiuser detection algorithm on the basis of belief
  propagation.
\newblock \emph{Journal of Physics A: Mathematical and General}, 36\penalty0
  (43):\penalty0 11111, 2003.

\bibitem[Kabashima et~al.(2016)Kabashima, Krzakala, M{\'e}zard, Sakata, and
  Zdeborov{\'a}]{kabashima2016phase}
Yoshiyuki Kabashima, Florent Krzakala, Marc M{\'e}zard, Ayaka Sakata, and Lenka
  Zdeborov{\'a}.
\newblock Phase transitions and sample complexity in {B}ayes-optimal matrix
  factorization.
\newblock \emph{IEEE Transactions on information theory}, 62\penalty0
  (7):\penalty0 4228--4265, 2016.

\bibitem[Li and Wei(2022)]{li2022non}
Gen Li and Yuting Wei.
\newblock A non-asymptotic framework for approximate message passing in spiked
  models.
\newblock \emph{arXiv preprint arXiv:2208.03313}, 2022.

\bibitem[Li et~al.(2023)Li, Fan, and Wei]{li2023approximate}
Gen Li, Wei Fan, and Yuting Wei.
\newblock Approximate message passing from random initialization with
  applications to $\mathbb{Z}_2$ synchronization.
\newblock \emph{Proceedings of the National Academy of Sciences}, 120\penalty0
  (31):\penalty0 e2302930120, 2023.

\bibitem[Li and Sur(2023)]{li2023spectrum}
Yufan Li and Pragya Sur.
\newblock Spectrum-aware adjustment: A new debiasing framework with
  applications to principal components regression.
\newblock \emph{arXiv preprint arXiv:2309.07810}, 2023.

\bibitem[Liu et~al.(2022)Liu, Huang, and Kurkoski]{liu2022memory}
Lei Liu, Shunqi Huang, and Brian~M Kurkoski.
\newblock Memory {AMP}.
\newblock \emph{IEEE Transactions on Information Theory}, 68\penalty0
  (12):\penalty0 8015--8039, 2022.

\bibitem[Ma and Ping(2017)]{ma2017orthogonal}
Junjie Ma and Li~Ping.
\newblock Orthogonal {AMP}.
\newblock \emph{IEEE Access}, 5:\penalty0 2020--2033, 2017.

\bibitem[Ma et~al.(2014)Ma, Yuan, and Ping]{ma2014turbo}
Junjie Ma, Xiaojun Yuan, and Li~Ping.
\newblock Turbo compressed sensing with partial {DFT} sensing matrix.
\newblock \emph{IEEE Signal Processing Letters}, 22\penalty0 (2):\penalty0
  158--161, 2014.

\bibitem[Ma et~al.(2021)Ma, Dudeja, Xu, Maleki, and Wang]{ma2021spectral}
Junjie Ma, Rishabh Dudeja, Ji~Xu, Arian Maleki, and Xiaodong Wang.
\newblock Spectral method for phase retrieval: an expectation propagation
  perspective.
\newblock \emph{IEEE Transactions on Information Theory}, 67\penalty0
  (2):\penalty0 1332--1355, 2021.

\bibitem[Ma et~al.(2023)Ma, Xu, and Maleki]{ma2023towards}
Junjie Ma, Ji~Xu, and Arian Maleki.
\newblock Towards designing optimal sensing matrices for generalized linear
  inverse problems.
\newblock \emph{IEEE Transactions on Information Theory}, 2023.

\bibitem[Maillard et~al.(2019)Maillard, Foini, Castellanos, Krzakala,
  M{\'e}zard, and Zdeborov{\'a}]{maillard2019high}
Antoine Maillard, Laura Foini, Alejandro~Lage Castellanos, Florent Krzakala,
  Marc M{\'e}zard, and Lenka Zdeborov{\'a}.
\newblock High-temperature expansions and message passing algorithms.
\newblock \emph{Journal of Statistical Mechanics: Theory and Experiment},
  2019\penalty0 (11):\penalty0 113301, 2019.

\bibitem[Mimura and Okada(2014)]{mimura2014generating}
Kazushi Mimura and Masato Okada.
\newblock Generating functional analysis for iterative {CDMA} multiuser
  detectors.
\newblock \emph{IEEE transactions on information theory}, 60\penalty0
  (6):\penalty0 3645--3670, 2014.

\bibitem[Mingo and Speicher(2017)]{mingo2017free}
James~A Mingo and Roland Speicher.
\newblock \emph{Free probability and random matrices}, volume~35.
\newblock Springer, 2017.

\bibitem[Minka(2013)]{minka2013expectation}
Thomas~P Minka.
\newblock Expectation propagation for approximate bayesian inference.
\newblock \emph{arXiv preprint arXiv:1301.2294}, 2013.

\bibitem[Montanari and Venkataramanan(2021)]{montanari2021estimation}
Andrea Montanari and Ramji Venkataramanan.
\newblock Estimation of low-rank matrices via approximate message passing.
\newblock \emph{The Annals of Statistics}, 49\penalty0 (1), 2021.

\bibitem[Montanari and Wein(2024)]{montanari2024equivalence}
Andrea Montanari and Alexander~S Wein.
\newblock Equivalence of approximate message passing and low-degree polynomials
  in rank-one matrix estimation.
\newblock \emph{Probability Theory and Related Fields}, pages 1--53, 2024.

\bibitem[Montanari and Wu(2024)]{montanari2022statistically}
Andrea Montanari and Yuchen Wu.
\newblock Statistically optimal first order algorithms: a proof via
  orthogonalization.
\newblock \emph{Information and Inference: A Journal of the IMA}, 13\penalty0
  (4):\penalty0 iaae027, 2024.

\bibitem[Nica and Speicher(2006)]{nica2006lectures}
Alexandru Nica and Roland Speicher.
\newblock \emph{Lectures on the combinatorics of free probability}, volume~13.
\newblock Cambridge University Press, 2006.

\bibitem[Opper et~al.(2005)Opper, Winther, and Jordan]{opper2005expectation}
Manfred Opper, Ole Winther, and Michael~J Jordan.
\newblock Expectation consistent approximate inference.
\newblock \emph{Journal of Machine Learning Research}, 6\penalty0 (12), 2005.

\bibitem[Opper et~al.(2016)Opper, Cakmak, and Winther]{opper2016theory}
Manfred Opper, Burak Cakmak, and Ole Winther.
\newblock A theory of solving {TAP} equations for {Ising} models with general
  invariant random matrices.
\newblock \emph{Journal of Physics A: Mathematical and Theoretical},
  49\penalty0 (11):\penalty0 114002, 2016.

\bibitem[Pandit et~al.(2020)Pandit, Sahraee-Ardakan, Rangan, Schniter, and
  Fletcher]{pandit2020inference}
Parthe Pandit, Mojtaba Sahraee-Ardakan, Sundeep Rangan, Philip Schniter, and
  Alyson~K Fletcher.
\newblock Inference with deep generative priors in high dimensions.
\newblock \emph{IEEE Journal on Selected Areas in Information Theory},
  1\penalty0 (1):\penalty0 336--347, 2020.

\bibitem[Parker et~al.(2014)Parker, Schniter, and Cevher]{parker2014bilinear}
Jason~T Parker, Philip Schniter, and Volkan Cevher.
\newblock Bilinear generalized approximate message passing—part i:
  Derivation.
\newblock \emph{IEEE Transactions on Signal Processing}, 62\penalty0
  (22):\penalty0 5839--5853, 2014.

\bibitem[Rangan and Fletcher(2012)]{rangan2012iterative}
Sundeep Rangan and Alyson~K Fletcher.
\newblock Iterative estimation of constrained rank-one matrices in noise.
\newblock In \emph{2012 IEEE International Symposium on Information Theory
  Proceedings}, pages 1246--1250. IEEE, 2012.

\bibitem[Rangan et~al.(2019)Rangan, Schniter, and Fletcher]{rangan2019vector}
Sundeep Rangan, Philip Schniter, and Alyson~K Fletcher.
\newblock Vector approximate message passing.
\newblock \emph{IEEE Transactions on Information Theory}, 65\penalty0
  (10):\penalty0 6664--6684, 2019.

\bibitem[Reeves and Pfister(2019)]{reeves2019replica}
Galen Reeves and Henry~D Pfister.
\newblock The replica-symmetric prediction for random linear estimation with
  gaussian matrices is exact.
\newblock \emph{IEEE Transactions on Information Theory}, 65\penalty0
  (4):\penalty0 2252--2283, 2019.

\bibitem[Rossetti et~al.(2024)Rossetti, Nazer, and Reeves]{rossetti2024linear}
Riccardo Rossetti, Bobak Nazer, and Galen Reeves.
\newblock Linear operator approximate message passing (opamp).
\newblock \emph{arXiv preprint arXiv:2405.08225}, 2024.

\bibitem[Sur and Cand{\`e}s(2019)]{sur2019modern}
Pragya Sur and Emmanuel~J Cand{\`e}s.
\newblock A modern maximum-likelihood theory for high-dimensional logistic
  regression.
\newblock \emph{Proceedings of the National Academy of Sciences}, 116\penalty0
  (29):\penalty0 14516--14525, 2019.

\bibitem[Takeuchi(2019{\natexlab{a}})]{takeuchi2019rigorous}
Keigo Takeuchi.
\newblock Rigorous dynamics of expectation-propagation-based signal recovery
  from unitarily invariant measurements.
\newblock \emph{IEEE Transactions on Information Theory}, 66\penalty0
  (1):\penalty0 368--386, 2019{\natexlab{a}}.

\bibitem[Takeuchi(2019{\natexlab{b}})]{takeuchi2019unified}
Keigo Takeuchi.
\newblock A unified framework of state evolution for message-passing
  algorithms.
\newblock In \emph{2019 IEEE International Symposium on Information Theory
  (ISIT)}, pages 151--155. IEEE, 2019{\natexlab{b}}.

\bibitem[Takeuchi(2021)]{takeuchi2021bayes}
Keigo Takeuchi.
\newblock Bayes-optimal convolutional {AMP}.
\newblock \emph{IEEE Transactions on Information Theory}, 67\penalty0
  (7):\penalty0 4405--4428, 2021.

\bibitem[Takeuchi(2023)]{takeuchi2023orthogonal}
Keigo Takeuchi.
\newblock Orthogonal approximate message-passing for spatially coupled linear
  models.
\newblock \emph{IEEE Transactions on Information Theory}, 2023.

\bibitem[Venkataramanan et~al.(2022)Venkataramanan, K{\"o}gler, and
  Mondelli]{venkataramanan2022estimation}
Ramji Venkataramanan, Kevin K{\"o}gler, and Marco Mondelli.
\newblock Estimation in rotationally invariant generalized linear models via
  approximate message passing.
\newblock In \emph{International Conference on Machine Learning}, pages
  22120--22144. PMLR, 2022.

\bibitem[Wang et~al.(2020)Wang, Weng, and Maleki]{wang2020bridge}
Shuaiwen Wang, Haolei Weng, and Arian Maleki.
\newblock Which bridge estimator is the best for variable selection?
\newblock \emph{The Annals of Statistics}, 48\penalty0 (5):\penalty0
  2791--2823, 2020.

\bibitem[Wang et~al.(2024)Wang, Zhong, and Fan]{wang2024universality}
Tianhao Wang, Xinyi Zhong, and Zhou Fan.
\newblock Universality of approximate message passing algorithms and tensor
  networks.
\newblock \emph{The Annals of Applied Probability}, 34\penalty0 (4):\penalty0
  3943--3994, 2024.

\bibitem[Xu et~al.(2023)Xu, Hou, Liang, and Mondelli]{xu2023approximate}
Yizhou Xu, TianQi Hou, ShanSuo Liang, and Marco Mondelli.
\newblock Approximate message passing for multi-layer estimation in
  rotationally invariant models.
\newblock In \emph{2023 IEEE Information Theory Workshop (ITW)}, pages
  294--298. IEEE, 2023.

\bibitem[Yedla et~al.(2014)Yedla, Jian, Nguyen, and Pfister]{yedla2014simple}
Arvind Yedla, Yung-Yih Jian, Phong~S Nguyen, and Henry~D Pfister.
\newblock A simple proof of maxwell saturation for coupled scalar recursions.
\newblock \emph{IEEE Transactions on Information Theory}, 60\penalty0
  (11):\penalty0 6943--6965, 2014.

\bibitem[Zhong et~al.(2024)Zhong, Wang, and Fan]{zhong2024approximate}
Xinyi Zhong, Tianhao Wang, and Zhou Fan.
\newblock Approximate message passing for orthogonally invariant ensembles:
  Multivariate non-linearities and spectral initialization.
\newblock \emph{Information and Inference: A Journal of the IMA}, 13\penalty0
  (3):\penalty0 iaae024, 2024.

\end{thebibliography}

\appendices

\section{Proofs for Preliminaries}\label{App:preliminary_results}

\subsection{Derivations of the Polynomial Coefficients (Proof of Corollary \ref{Cor:poly_coefficients})}\label{App:poly_coefficients_proof}

Recall that $(Q_n)_{n\ge1}$ satisfies the recursion (see \eqref{Eqn:Q_def}):
\BS
\begin{align}
Q_n&= \mathsf{\Lambda}Q_{n-1}-\sum_{i=1}^n \mathbb{E}[\mathsf{\Lambda}Q_{i-1}]\cdot Q_{n-i},\quad\forall n\ge1\\
&\explain{\eqref{Eqn:kappa_Q}}{=} \mathsf{\Lambda}Q_{n-1}-\sum_{i=1}^n \kappa_i \cdot Q_{n-i}.
\end{align}
\ES
The polynomial representation of $Q_n$ is denoted as (see \eqref{Eqn:Q_poly_def})
\BE
Q_n =\sum_{i=0}^n \alpha_{n,i} \mathsf{\Lambda}^i,\quad\forall n\ge1.
\EE
Combining these two equations yields
\BS
\begin{align}
%\sum_{i=0}^n \alpha_{n,i} \mathsf{\Lambda}^i &=\mathsf{\Lambda}\cdot \sum_{i=0}^{n-1} \alpha_{n-1,i} \mathsf{\Lambda}^i - \sum_{j=1}^n \kappa_j\cdot \sum_{i=0}^{n-j}\alpha_{n-j,i} \mathsf{\Lambda}^i\\
\sum_{i=0}^n \alpha_{n,i} \mathsf{\Lambda}^i &= \sum_{i=0}^{n-1} \alpha_{n-1,i} \mathsf{\Lambda}^{i+1}- \sum_{j=1}^n \sum_{i=0}^{n-j}\kappa_j\alpha_{n-j,i}\mathsf{\Lambda}^i\\
&=  \sum_{i=1}^{n} \alpha_{n-1,i-1} \mathsf{\Lambda}^{i}-\sum_{j=1}^n \sum_{i=0}^{n-j}\kappa_j\alpha_{n-j,i}\mathsf{\Lambda}^i\\
&\explain{(a)}{=}\ \left(\alpha_{n-1,n-1}\mathsf{\Lambda}^{n}+  \sum_{i=1}^{n-1} \alpha_{n-1,i-1} \mathsf{\Lambda}^{i}\right)-\sum_{i=0}^{n-1}\sum_{j=1}^{n-i} \kappa_j\alpha_{n-j,i}\mathsf{\Lambda}^i\\
&=\alpha_{n-1,n-1}\mathsf{\Lambda}^{n}+  \sum_{i=1}^{n-1} \alpha_{n-1,i-1} \mathsf{\Lambda}^{i}-\sum_{j=1}^n \kappa_j\alpha_{n-j,0}- \sum_{i=1}^{n-1}\left(\sum_{j=1}^{n-i}\kappa_j\alpha_{n-j,i}\right)\mathsf{\Lambda}^i,\quad\forall n\ge2\\
&=\alpha_{n-1,n-1}\mathsf{\Lambda}^{n}+  \sum_{i=1}^{n-1} \left(\alpha_{n-1,i-1}-\sum_{j=1}^{n-i}\kappa_j\alpha_{n-j,i} \right)\mathsf{\Lambda}^{i}-\sum_{j=1}^n \kappa_j\alpha_{n-j,0},\quad\forall n\ge2,
\end{align}
\ES
where step (a) is due to a swap of the summation order.
Comparing the coefficients of the two sides shows the following ($\forall n\ge2,\forall 1\le i\le n-1$):
\BS
\begin{align}
\alpha_{n,n} &= \alpha_{n-1,n-1},\\
\alpha_{n,0} &= -\sum_{j=1}^n \kappa_j\alpha_{n-j,0},\\
\alpha_{n,i} &= \alpha_{n-1,i-1}-\sum_{j=1}^{n-i}\kappa_j\alpha_{n-j,i}.
\end{align}
\ES
For $n=1$, it is easy to verify that $\alpha_{1,0}=-\kappa_1$ and $\alpha_{1,1}=1$. Overall, the above recursion can be written into a unified formula \eqref{Eqn:alpha_recursion} together with the initialization $\alpha_{i,i}=1,\alpha_{i,-1}=0$, $\forall i\ge0$.

\subsection{A Reformulation of the Moment-Cumulant Formula}\label{App:moment_cumulant_new}

We will introduce a recursive characterization of free cumulants in Proposition \ref{Lem:cumulants} which is is useful for our derivation of RI-AMP. Before that, we present a useful reformulation of the right-hand side of the moment-cumulant formula \eqref{Eqn:free_cumulant_def}.

\begin{lemma}\label{Lem:moment_cumulant_new}
Let $q_0:=1$. The following holds for any sequence $(q_\ell)_{\ell\ge1}$:
\BE\label{Eqn:cumulant_moment_new}
\sum_{\pi\in\mathrm{NC}(\ell)}\prod_{B\in\pi}q_{|B|}=\sum_{(s_1,\ldots,s_\ell)\in\mathcal{S}(\ell)} \ \prod_{i=1}^\ell q_{s_{i}},\quad \forall \ell\ge1,
\EE
where $\mathcal{S}(\ell)$ denotes the collection of $\ell$-tuple $(s_1,\ldots,s_\ell)$ satisfying
\begin{itemize}
\item[(1)] $s_m\ge0,\ \forall m\in[\ell]$.
\item[(2)] $\sum_{j=1}^m s_j\le m,\  \forall m\in[\ell]$.
\item[(3)] $\sum_{j=1}^\ell s_j = \ell$.
\end{itemize}
\end{lemma}

\begin{proof}
First, following similar reasoning as \citep[pp.~364]{anderson2010introduction}, we can prove that there is a bijection between $\text{NC}(\ell)$ and $\mathcal{S}(\ell)$. For completeness, we include the proof here. We shall provide the explicit forms of the maps for both directions.

\paragraph{Map from $\text{NC}(\ell)$ to $\mathcal{S}(\ell)$:} Consider $\pi:=\{B_1,\ldots,B_N\}\in\text{NC}(\ell)$, where $B_i$ denotes the $i$-th block of $\pi$. Define a map $\phi_\ell:\pi\mapsto(s_1,\ldots,s_\ell)$:
\BE\label{Eqn:map_def}
s_i:=
\begin{cases}
%1, & \text{ if $\{i\}$ is a singleton of $\pi$},\\
|B_{j}|, & \text{ if $i$ is the largest element of some block $B_j$, where $j\in[\ell]$},\\
0, & \text{ otherwise}.
\end{cases}
\EE
We claim that $\phi_\ell(\pi):=(s_1,\ldots,s_\ell)\in\mathcal{S}(\ell)$ for any non-crossing partition $\pi\in\text{NC}(\ell)$. To see this, we verify that such $(s_1,\ldots,s_\ell)$ satisfies the following three conditions: (1) $s_m\ge0,\ \forall m\in[\ell]$; (2) $\sum_{j=1}^m s_j\le m,\  \forall m\in[\ell]$; (3) $\sum_{j=1}^\ell s_j = \ell$. Conditions (1) and (3) directly follow from the definition \eqref{Eqn:map_def}. Condition (2) is a consequence of the definition \eqref{Eqn:map_def} as well as the fact that $\pi$ is non-crossing. 

\paragraph{Map from $\mathcal{S}(\ell)$ to $\text{NC}(\ell)$:} Let $(s_1,\ldots,s_\ell)\in\mathcal{S}(\ell)$. We start from $s_\ell$ and count backwards towards $s_1$ until we meet some nonzero $s_k$. We claim that one of the following scenarios must happen: $s_k=1$ or $(s_{k-j},\ldots,s_k)=(\underbrace{0,\ldots,0}_{j-1 \text{ zeros}},j)$ for some $1\le j\le k-1$. Otherwise, the defining properties of $\mathcal{S}(\ell)$ would be violated. In the latter case, we form a block consisting of $\{s_{k-j},\ldots,s_k\}$. Next, we remove the elements in this block and repeat the whole procedure (starting from the largest elements) until no elements in $[\ell]$ are left. The blocks produced by this procedure constitute a partition of $[\ell]$. Finally, this construction guarantees that the partition is non-crossing.

\paragraph{Proof of Lemma \ref{Lem:moment_cumulant_new}.} Let $\pi\in\text{NC}(\ell)$ be an arbitrary non-crossing partition, and denote $\phi_\ell(\pi)=(s_1(\pi),\ldots,s_n(\pi))\in\mathcal{S}(\ell)$. By the definition of the map $\phi_\ell$ and $q_0:=1$, we have  
\[
q_{\pi}:=\prod_{B\in\pi}q_{|B|}=\prod_{j=1}^\ell q_{s_j(\pi)}.
\]
Because of the bijection between $\text{NC}(\ell)$ and $\mathcal{S}(\ell)$, we have 
\BE
\sum_{\pi\in\text{NC}(\ell)}\prod_{B\in\pi}q_{|B|}=\sum_{(s_1,\ldots,s_\ell)\in\mathcal{S}(\ell)}\ \prod_{j=1}^\ell q_{s_j},
\EE
which completes the proof. 
\end{proof}

\begin{figure}[htbp]
\centering
\includegraphics[width=0.7\textwidth]{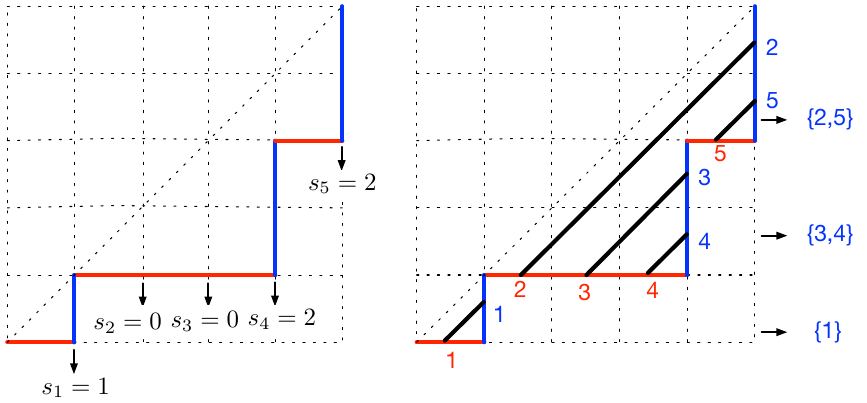}
\caption{\textbf{Left:} Map from $(s_1,\ldots,s_5)=(1,0,0,2,2)\in\mathcal{S}(5)$ to a Dyck path of length 10 marked in solid lines. \textbf{Right:} Map from the Dyck path to a non-crossing partition of $[5]$: $(\{1\}, \{2,5\},\{3,4\})$. (The non-crossing partition is depicted in Fig.~\ref{Fig:Dyck}.) The 5 horizontal/vertical steps of the Dyck path are marked in red/blue colors.}\label{Fig:Dyck}
\end{figure}

\begin{remark}[Connections of $\mathcal{S}(n)$, Dyck paths and non-crossing partitions]
The object $\mathcal{S}(\ell)$ is also closely related to (and bijective to) \textit{Dyck path} \citep{nica2006lectures}. A {Dyck path} of length $2\ell$ is a lattice path in $\mathbb{Z}^2$ from $(0,0)$ to $(k,k)$, consisting of $k$ east steps and $k$ north steps, which lives below the diagonal line $y=x$ (but may touch it).\footnote{Notice that the definition of Dyck path here is a flipped version of the usual definition, but they are equivalent.} Let $\text{Dyck}(\ell)$ be the collection of Dyck paths of length $2\ell$. {An example of a Dyck path of length 10 is shown in Fig.~\ref{Fig:Dyck}.} 

The three objects $\mathcal{S}(n)$, Dyck paths and non-crossing partitions are bijective to each other. For the connection between $\mathcal{S}(n)$ and Dyck paths, we note that each $s_i$ ($\forall i\in[\ell]$) can be interpreted as the vertical increment of a Dyck path at the vertical axis $x=i$; see illustration in the first sub-figure of Fig.~\ref{Fig:Dyck}. Therefore, if $s_i>0$, the Dyck path has a vertical step (marked in blue color) of length $s_i$ at the axis $x=i$; while if $s_i=0$, the Dyck path has no vertical step at the axis $x=i$. The connection between Dyck paths and non-crossing partitions are well-known \cite{nica2006lectures}, and is illustrated in the second sub-figure of Fig.~\ref{Fig:Dyck}. To map a Dyck path to a non-crossing partition, we label the $\ell$ horizontal steps of the Dyck paths in increasing order from left to right. The $\ell$ vertical steps of the Dyck paths are labeled in a way as shown in Fig.~\ref{Fig:Dyck}. We group the labels in the same vertical axis into one block. These blocks form a non-crossing partition of $[\ell]$. From this map, it is clear that the cardinalities of the blocks of the non-crossing partition $(|B_1|,\ldots,|B_k|)$, which appear in the cumulant-moment formula \eqref{Eqn:free_cumulant_def}, correspond to the lengths of the vertical segments of a Dyck path.
\end{remark}

%===========================
\subsection{A Recursive Characterization of Free Cumulants (Proof of Proposition \ref{Lem:cumulants})}\label{App:cumulants}

%\begin{proof}
Let $(Q_n)_{n\ge1}$ be defined as in \eqref{Eqn:Q_def}. Denote
\BE\label{Eqn:qk_def}
q_n:=\mathbb{E}[\mathsf{\Lambda}Q_{n-1}], \quad \forall  n\ge1.
\EE
We first prove that the following holds for all $\ell\ge1$, $n\ge0$:
%\BE\label{Eqn:H_mn}
%\mathcal{H}_{\ell,n}:\quad  \mathbb{E}[X^{\ell}Q_{n}]={\sum_{i_1=0}^{n+1}\sum_{i_2=0}^{n+2-i_1}\cdots\sum_{i_{\ell-1}=0}^{n+(\ell-1)-(i_1+i_2+\cdots +i_{\ell-1})} }q_{i_1}\cdot q_{i_2}\cdots q_{i_{\ell-1}}\cdot q_{\ell+n-(i_1+i_2+\cdots +i_{\ell-1})},
%\EE
\BE\label{Eqn:H_mn}
\mathcal{H}_{\ell,n}:\quad  \mathbb{E}[\mathsf{\Lambda}^{\ell}Q_{n}]=\sum_{(s_1,\ldots,s_\ell)\in\mathcal{S}(\ell,n)}\ \prod_{j=1}^\ell q_{s_j},
\EE
where $q_0:=1$ and $\mathcal{S}(\ell,n)$ denotes the collection of $\ell$-tuple $(s_1,\ldots,s_\ell)$ satisfying
\begin{itemize}
\item[(1)] $s_m\ge0$, $\forall m\in[\ell]$.
\item[(2)] $\sum_{j=1}^m s_j\le m+n$,  $\forall m\in[\ell]$.
\item[(3)] $\sum_{j=1}^\ell s_j=\ell+n$.
\end{itemize}
Next, we shall prove \eqref{Eqn:H_mn} by induction on $\ell$:
\begin{enumerate}
\item Base case: $\{\mathcal{H}_{1,n},\forall n\ge0\}$.
\item Induction step: for all $\ell\ge1$, $\{\mathcal{H}_{\ell,n}, \forall n\ge0\} \  \Longrightarrow \  \{\mathcal{H}_{\ell+1,n}, \forall n\ge0\}$.
\end{enumerate}
\paragraph{Base case:} When $\ell=1$, \eqref{Eqn:H_mn} reduces to
\BE
\mathcal{H}_{1,n}:\quad  \mathbb{E}[\mathsf{\Lambda} Q_{n}] = q_{1+n},
\EE
which follows from the definition in \eqref{Eqn:qk_def}.

\paragraph{Induction step:} We shall prove the following holds for all $\ell\ge1$ and  $n\ge1$:
\BE\label{Eqn:induction}
\mathcal{H}_{\ell,n}\quad \Longrightarrow \quad \mathcal{H}_{\ell+1,n-1}.
\EE
This would imply $\{\mathcal{H}_{\ell,n},\forall n\ge1\} \Longrightarrow\{  \mathcal{H}_{\ell+1,n'},\forall n'\ge0\}$ and hence the induction step we aim to prove $\{\mathcal{H}_{\ell,n},\forall n\ge0\} \Longrightarrow\{  \mathcal{H}_{\ell+1,n'},\forall n'\ge0\}$. To prove \eqref{Eqn:induction}, we unfold $Q_{n+1}$:
\begin{align*}
\mathbb{E}[\mathsf{\Lambda}^\ell Q_{n+1}] &=\mathbb{E}\left[\mathsf{\Lambda}^\ell \left(\mathsf{\Lambda}Q_{n}-\sum_{i=1}^{n+1} \mathbb{E}[\mathsf{\Lambda} Q_{i-1}]\cdot Q_{n+1-i}\right)\right]\\
&=\mathbb{E}[\mathsf{\Lambda}^{\ell+1}Q_{n}]-\sum_{i=1}^{n+1} q_i\cdot \mathbb{E}\left[\mathsf{\Lambda}^\ell Q_{n+1-i}\right].
\end{align*}
Re-arranging terms and noting $q_0=1$:
\BS\label{Eqn:Q_temporary}
\begin{align}
\mathbb{E}[\mathsf{\Lambda}^{\ell+1}Q_{n}] &=\sum_{i=0}^{n+1}q_i\cdot \mathbb{E}\left[\mathsf{\Lambda}^\ell Q_{n+1-i}\right]\\
&\explain{(a)}{=}\sum_{i=0}^{n+1}q_i\cdot\left(\sum_{(s_1,\ldots,s_\ell)\in\mathcal{S}(\ell,n+1-i)}\ q_{s_1}\cdots q_{s_\ell}\right)\\
&\explain{(b)}{=}\sum_{(i,s_1,\ldots,s_\ell)\in\mathcal{S}(\ell+1,n)}\ q_i\cdot q_{s_1}\cdots p_{s_\ell}
\end{align}
\ES
where step (a) is a consequence of the induction hypothesis $\{\mathcal{H}_{\ell,n}, \forall n\ge0\}$ and step (b) is from the following recursion (which can be verified from the definition of $\mathcal{S}(\ell+1,n)$)
\[
\mathcal{S}(\ell+1,n)=\big\{(i,s_1,\ldots,s_\ell): 0\le i\le n+1,\, (s_1,\ldots,s_\ell)\in \mathcal{S}(\ell,n+1-i)\big\}.
\]
This proves \eqref{Eqn:induction} and completes the proof of \eqref{Eqn:H_mn}.

Finally, setting $n=0$ in \eqref{Eqn:H_mn} yields
\BS\label{Eqn:moment_cumulant_recall}
\begin{align}
m_\ell\overset{(a)}{=}\mathbb{E}[\mathsf{\Lambda}^\ell Q_0] &=\sum_{(s_1,\ldots,s_\ell)\in\mathcal{S}(\ell,0)}\ \prod_{j=1}^\ell q_{s_j}\\
&\explain{(b)}{=}\sum_{(s_1,\ldots,s_\ell)\in\mathcal{S}(\ell)}\ \prod_{j=1}^\ell q_{s_j}\\
&\explain{(c)}{=} \sum_{\pi\in\text{NC}(\ell)}\prod_{B\in\pi}q_{|B|},
\end{align}
\ES
where step (a) follows from $Q_0:=1$; step (b) is from the fact that $\mathcal{S}(\ell,0)$ in \eqref{Eqn:moment_cumulant_recall} is equal to $\mathcal{S}(\ell)$ defined in Lemma \ref{Eqn:cumulant_moment_new}; and step (c) is due to Lemma \ref{Lem:moment_cumulant_new}. Note that \eqref{Eqn:moment_cumulant_recall} is precisely the moment-cumulant formula that defines free cumulants. We can invert the above formula \citep[Lecture 10]{nica2006lectures} and conclude that the sequence $(q_{\ell})_{\ell\ge1}$ is equal to the sequence of free cumulants $(\kappa_{\ell})_{\ell\ge1}$.
%\end{proof}

%===================================
\subsection{Connection with the Partial Moments in \cite{fan2022approximate}}\label{App:partial_moments}

\paragraph{Proof of \eqref{Eqn:partial_moment_identity}:} We consider the following cases:
\begin{itemize}
\item $t=0,j=0$: we have $\E [ Q_{0}] = 1 = c_{0,0}$. 
\item $t=0, j\geq 1$: we have $\E [ Q_{j}] = 0 = c_{0,j}$.
\item $t=1$: From the definition \eqref{Eqn:partial_moments_def}, we have
\[
c_{1,j} = \sum_{m=0}^{j+1} c_{0,m} \; \kappa_{j+1-m} \explain{(a)}{=} \kappa_{j+1},
\]
where step (a) is due to the definition of $c_{0,m}$ for the above two corner cases. On the other hand, Proposition \ref{Lem:cumulants} shows that $\E [ \mathsf{\Lambda} Q_{j} ]  = \kappa_{j+1},\forall j\ge0$. Hence,
\[ 
\E [ \mathsf{\Lambda} Q_{j} ] = c_{1,j} ,\quad\forall j\ge0.
\]
\item $t > 1$: we shall prove via induction. Suppose that $c_{k,j} = \E [\mathsf{\Lambda}^k Q_{j}] $ for all $0\le k\le t-1$ and $j\ge0$. Then, the following holds $\forall j\ge0$:
\BS
\begin{align}
\mathbb{E}[\mathsf{\Lambda}^{t}Q_{j}] &\explain{(a)}{=}\sum_{i=0}^{j+1} \kappa_i\cdot \mathbb{E}\left[\mathsf{\Lambda}^{t-1} Q_{j+1-i}\right] \\
&\explain{(b)}{=} \sum_{i=0}^{j+1} \kappa_i\cdot c_{t-1,j+1-i} \\
&\explain{(c)}{=}  \sum_{m=0}^{j+1} \kappa_{j+1-m}\cdot c_{t-1,m} \\
&\explain{(d)}{=} c_{t,j},
\end{align}
\ES
where 
\begin{itemize}
\item[{(a)}] This step is due to \eqref{Eqn:Q_temporary}. Note that from the definition of $q_i$ in \eqref{Eqn:qk_def}, $q_i:=\mathbb{E}[\mathsf{\Lambda}Q_{i-1}]$. On the other hand, Proposition \ref{Lem:cumulants} shows $\mathbb{E}[\mathsf{\Lambda}Q_{i-1}]=\kappa_i$. Substituting $q_i=\kappa_i$ into \eqref{Eqn:Q_temporary} proves this step.
\item[{(b)}] This step follows from the induction hypothesis.
\item[{(c)}] This is a change of variable.
\item[{(d)}] This is from the definition of $c_{t,j}$; see \eqref{Eqn:partial_moments_def}.
\end{itemize}
\end{itemize} 
The proof is now complete.
\section{State Evolution of OAMP (Proof of Theorem \ref{The:OAMP_SE})}\label{App:proof_OAMP}

Consider the following minor variant of the OAMP algorithm in \eqref{Eqn:LM_OAMP_recall}:
\BS\label{Eqn:LM_OAMP_recall_app}
\begin{align}
\bm{w}_t &= \left(f_t(\bm{W})-\mathbb{E}\left[f_t(\mathsf{\Lambda})\right]\cdot\bm{I}_N\right)\bar{\bm{u}}_{t},\\
\bar{\bm{u}}_{t+1} &= g_{t+1} (\bm{w}_{\le t};\bm{a})-\sum_{i=1}^t \mathbb{E}\left[\partial_i  g_{t+1} (\mathsf{X}_{\le t};\mathsf{A})\right]\cdot \bm{w}_i,\label{Eqn:LM_OAMP_recall_app_b}
\end{align}
\ES
where the initialization $\bar{\bm{u}}_1=\bar{\bm{x}}_1$, and the expectations in \eqref{Eqn:LM_OAMP_recall_app_b} are taken w.r.t. the random variables $(\mathsf{X}_1,\ldots,\mathsf{X}_t;\mathsf{A})$ defined recursively in \eqref{Eqn:SE_OAMP_first}-\eqref{Eqn:SE_OAMP_b}. Note that \eqref{Eqn:LM_OAMP_recall_app} is an instance of the vector AMP algorithm as defined in \cite[Section 4.1.2]{dudeja2022spectral}. In the above equation, we have introduced the shorthand
\[
\bm{w}_{\le t}:=(\bm{w}_1,\ldots,\bm{w}_t).
\]
(We will keep this notation throughout this section.) By \cite[Theorem 2]{dudeja2022spectral}, the following convergence holds
\BE\label{Eqn:OAMP_population_variables}
(\bm{w}_1,\bm{w}_2,\ldots,\bm{w}_t;\bm{a})\overset{W_2}{\longrightarrow}(\mathsf{X}_1,\mathsf{X}_2,\ldots,\mathsf{X}_t;\mathsf{A}),
\EE
where the state evolution random variables $(\mathsf{X}_1,\ldots,\mathsf{X}_t;\mathsf{A})$ are defined in \eqref{Eqn:SE_OAMP_first}-\eqref{Eqn:SE_OAMP_b}. Compared with the original OAMP algorithm \eqref{Eqn:LM_OAMP_recall}, the iteration \eqref{Eqn:LM_OAMP_recall_app} replaces $\frac{1}{N}\text{tr}\left(f_t(\bm{W})\right)$ and $\left( \langle\partial_i  g_{t+1} (\bm{x}_1,\ldots,\bm{x}_t;\bm{a})\rangle\right)_{i\in[t]}$ by $\mathbb{E}\left[f_t(\mathsf{\Lambda})\right]$ and $\left(\mathbb{E}\left[\partial_i  g_{t+1} (\mathsf{X}_1,\ldots,\mathsf{X}_t;\mathsf{A})\right]\right)_{i\in[t]}$, respectively. Similar to the arguments used in the proof of \cite[Theorem 1]{dudeja2024optimality}, we can show that the difference between these two algorithms is asymptotically negligible.

Following \cite{dudeja2024optimality}, we say $\bm{u}\explain{$\dim\rightarrow \infty$}{\simeq}\bm{v}$ for two random vectors $\bm{u},\bm{v}\in\mathbb{R}^N$, if
\[
\frac{\|\bm{u}-\bm{v}\|^2}{N}\overset{\mathbb{P}}{\longrightarrow}0\quad \text{as}\quad N\to\infty.
\]
Given the convergence result in \eqref{Eqn:OAMP_population_variables}, it suffices to prove that the the iterates generated by the original OAMP algorithm and that generated by the auxiliary iteration \eqref{Eqn:LM_OAMP_recall_app} satisfies 
\BE\label{Eqn:OAMP_SE_app_error_goal}
\bm{x}_t \explain{$\dim\rightarrow \infty$}{\simeq} \bm{w}_t,\quad\forall t\ge1.
\EE
We show this by induction. Suppose that this is true up to iteration $t$. Recall the definitions in \eqref{Eqn:LM_OAMP} and \eqref{Eqn:LM_OAMP_recall_app}:
\BS
\begin{align}
\bm{x}_{t+1} &=\left(f_{t+1}(\bm{W})-\frac{\text{tr}\left(f_{t+1}(\bm{W})\right)}{N}\cdot\bm{I}_N\right)\bar{\bm{x}}_{t+1},\\
\bm{w}_{t+1} &=\left(f_{t+1}(\bm{W})-\mathbb{E}[f_{t+1}(\mathsf{\Lambda})]\cdot\bm{I}_N\right)\bar{\bm{u}}_{t+1}.
\end{align}
\ES
Hence,
\BS\label{Eqn:OAMP_SE_app_error}
\begin{align}
\frac{\|\bm{x}_{t+1} - \bm{w}_{t+1}\|^2}{N} &\le\frac{1}{N}\left\|\left(f_{t+1}(\bm{W})-\frac{\text{tr}\left(f_{t+1}(\bm{W})\right)}{N}\cdot\bm{I}_N\right)\left(\bar{\bm{x}}_{t+1}-\bar{\bm{u}}_{t+1}\right)\right\|^2\nonumber\\
&\quad + \left(\mathbb{E}[f_{t+1}(\mathsf{\Lambda})]-\frac{\text{tr}\left(f_{t+1}(\bm{W})\right)}{N}\right)^2\cdot\frac{1}{N}\left\|\bar{\bm{u}}_{t+1}\right\|^2\\
&\le \left\|f_{t+1}(\bm{W})-\frac{\text{tr}\left(f_{t+1}(\bm{W})\right)}{N}\cdot\bm{I}_N\right\|_{\mathrm{op}}\cdot\frac{1}{N}\left\|\bar{\bm{x}}_{t+1}-\bar{\bm{u}}_{t+1}\right\|^2\nonumber\\
&\quad + \left(\mathbb{E}[f_{t+1}(\mathsf{\Lambda})]-\frac{\text{tr}\left(f_{t+1}(\bm{W})\right)}{N}\right)^2\cdot\frac{1}{N}\left\|\bar{\bm{u}}_{t+1}\right\|^2.
\end{align}
\ES
To control the above term, we bound the operator norm of the matrix $f_{t+1}(\bm{W})-\frac{\text{tr}\left(f_{t+1}(\bm{W})\right)}{N}\cdot\bm{I}_N$:
\BS\label{Eqn:op_bound}
\begin{align}
\left\|f_{t+1}(\bm{W})-\frac{\text{tr}\left(f_{t+1}(\bm{W})\right)}{N}\cdot\bm{I}_N\right\|_{\mathrm{op}}&=\max_{1\le i\le N}\ \left|f_{t+1}(\lambda_i)-\frac{1}{N}\sum_{j=1}^N f_{t+1}(\lambda_j)\right|\\
&\le\max_{1\le i,i'\le N}\ \left|f_{t+1}(\lambda_i)-f_{t+1}(\lambda_{i'})\right|\\
&\le 2 \max_{1\le i\le N}\ \left|f_{t+1}(\lambda_i)\right|\le C',
\end{align}
\ES
for some $N$-independent constant $C'$. The last step is due to the fact that $\max_{1\le i\le N}\ \left|\lambda_i\right|$ is bounded by an $N$-independent constant and $f_{t+1}$ is continuous and independent of $N$; see Assumption \ref{Ass:LM_OAMP}. Further, the empirical eigenvalue distribution of $\bm{W}$ converges to a compactly support measure, hence
\BE\label{Eqn:trace_convergence}
\lim_{N\to\infty}\frac{\text{tr}\left(f_{t+1}(\bm{W})\right)}{N}=\mathbb{E}[f_{t+1}(\mathsf{\Lambda})].
\EE
In light of \eqref{Eqn:OAMP_SE_app_error}-\eqref{Eqn:trace_convergence}, to prove \eqref{Eqn:OAMP_SE_app_error_goal}, it remains to show
\BS
\begin{align}
&\underset{N\to\infty}{\mathrm{plim}}\ \frac{1}{N}\left\|\bar{\bm{x}}_{t+1}-\bar{\bm{u}}_{t+1}\right\|^2 =0,\label{Eqn:OAMP_SE_bar_diff}\\
& \underset{N\to\infty}{\mathrm{plim}\,\mathrm{sup}}\ \frac{1}{N}\left\|\bar{\bm{u}}_{t+1}\right\|^2 <\infty.\label{Eqn:OAMP_SE_bar_bound}
\end{align}
\ES
Note that $\bar{\bm{u}}_{t+1}$ is generated via the auxiliary iteration \eqref{Eqn:LM_OAMP_recall_app}, and \eqref{Eqn:OAMP_SE_bar_bound} follows from the convergence result \eqref{Eqn:OAMP_population_variables}. Finally, to prove \eqref{Eqn:OAMP_SE_bar_diff}, recall the definitions of $\bar{\bm{x}}_{t+1}$ and $\bar{\bm{u}}_{t+1}$ in \eqref{Eqn:LM_OAMP} and \eqref{Eqn:LM_OAMP_recall_app}:
\BS
\begin{align*}
 &\frac{1}{N}\left\|\bar{\bm{x}}_{t+1}-\bar{\bm{u}}_{t+1}\right\|^2\\
 &=\frac{1}{N}\left\|g_{t+1} (\bm{x}_{\le t};\bm{a})-\sum_{i=1}^{t} \langle\partial_i  g_{t} (\bm{x}_{\le t};\bm{a})\rangle\cdot \bm{x}_i-\left(g_{t+1} (\bm{w}_{\le t};\bm{a})-\sum_{i=1}^t \mathbb{E}\left[\partial_i  g_{t+1} (\mathsf{X}_{\le t};\mathsf{A})\right]\cdot \bm{w}_i\right)\right\|^2\\
 &\le \frac{1}{N}\left\|g_{t+1} (\bm{x}_{\le t};\bm{a})-g_{t+1} (\bm{w}_{\le t};\bm{a})\right\|^2+\sum_{i=1}^t \frac{1}{N}\big\| \langle\partial_i  g_{t} (\bm{x}_{\le t})\rangle\cdot \bm{x}_i - \mathbb{E}\left[\partial_i  g_{t+1} (\mathsf{X}_{\le t};\mathsf{A})\right]\cdot \bm{w}_i \big\|^2.
\end{align*}
\ES
The above term converges to zero in probability due to the following facts:
\begin{itemize}
\item Induction hypothesis: $\bm{x}_s \explain{$\dim\rightarrow \infty$}{\simeq} \bm{w}_s,\quad\forall s=1,\ldots,t$.
\item The assumption that $g_{t+1}$ is Lipschitz continuous.
\item The convergence $\langle\partial_i  g_{t} (\bm{x}_{\le t};\bm{a})\rangle\overset{\mathbb{P}}{\longrightarrow}\mathbb{E}\left[\partial_i  g_{t+1} (\mathsf{X}_{\le t};\mathsf{A})\right]$, $\forall i=1,\ldots,t$. This is a consequence of: (1) the induction hypothesis: $\bm{x}_s \explain{$\dim\rightarrow \infty$}{\simeq} \bm{w}_s,\quad\forall s=1,\ldots,t$; (2) the assumption that $g_{t+1}$ is Lipschitz and continuously differentiable (and hence $\partial_i  g_{t+1}$ is continuously bounded).
\end{itemize}
The proof is now complete.

\section{Proofs for Main Results}\label{App:main_results}

\subsection{Reformulation of First-order Method (Proof of Lemma \ref{Lem:RI_AMP_reformulation})}\label{App:reformulation}

%\begin{proof}

We stack the iterates of the RI-AMP algorithm as
\BS\label{Eqn:OAMP_derive_r_u}
\begin{align}
\begin{bmatrix}
\vr_1 \\
\vr_2 \\
\vdots \\
\vr_t
\end{bmatrix}
&=
\begin{bmatrix}
\bm{W}\vu_1 \\
\bm{W}\vu_2 \\
\vdots \\
\bm{W}\vu_t 
\end{bmatrix}
-
\begin{bmatrix}
\para{b}_{1,1}\bm{I}_N & & & \\
\para{b}_{2,1}\bm{I}_N & \para{b}_{2,2}\bm{I}_N & & \\
\vdots & \vdots & \ddots & \\
\para{b}_{t,1} \bm{I}_N& \para{b}_{t,2}\bm{I}_N & \cdots & \para{b}_{t,t}\bm{I}_N
\end{bmatrix} 
\begin{bmatrix}
\vu_1 \\
\vu_2 \\
\vdots \\
\vu_t 
\end{bmatrix}\\
%&=(\bm{I}_t\otimes\bm{W})
%\begin{bmatrix}
%\vu_1 \\
%\vu_2 \\
%\vdots \\
%\vu_t 
%\end{bmatrix}
%-(\mathsf{B}_t\otimes\bm{I}_N)
%\begin{bmatrix}
%\vu_1 \\
%\vu_2 \\
%\vdots \\
%\vu_t 
%\end{bmatrix}\\
&=\left(\paraB{I}_t\otimes\bm{W}-\paraB{B}_t\otimes\bm{I}_N\right)
\begin{bmatrix}
\vu_1 \\
\vu_2 \\
\vdots \\
\vu_t 
\end{bmatrix}.
\end{align}
\ES
Using the definition of $\bar{\bm{u}}_t$ in \eqref{Eqn:ut_bar_def}, we have
\BS\label{Eqn:OAMP_derive_u_ubar}
\begin{align}
\begin{bmatrix}
\bm{u}_1\\
\bm{u}_2\\
\vdots\\
\bm{u}_t
\end{bmatrix}
&=
\begin{bmatrix}
\bar{\bm{u}}_1\\
\bar{\bm{u}}_2\\
\vdots\\
\bar{\bm{u}}_t
\end{bmatrix}+
\begin{bmatrix}
\bm{0}_N & & & &\\
\para{d}_{1,1}\bm{I}_N & \bm{0}_N & & & \\
\para{d}_{2,1}\bm{I}_N &\para{d}_{2,2}\bm{I}_N &\bm{0}_N  & & \\
\vdots & \vdots &\vdots & \ddots &\\
\para{d}_{t-1,1}\bm{I}_N &\para{d}_{t-1,2}\bm{I}_N &\cdots &\para{d}_{t-1,t-1}\bm{I}_N &\bm{0}_N   
\end{bmatrix}
\begin{bmatrix}
\bm{r}_1\\
\bm{r}_2\\
\vdots\\
\bm{r}_t
\end{bmatrix}\\
&=
\begin{bmatrix}
\bar{\bm{u}}_1\\
\bar{\bm{u}}_2\\
\vdots\\
\bar{\bm{u}}_t
\end{bmatrix}
+
(\paraB{D}_t \otimes \bm{I}_N)
\begin{bmatrix}
\bm{r}_1\\
\bm{r}_2\\
\vdots\\
\bm{r}_t
\end{bmatrix}.
\end{align}
\ES
Substituting \eqref{Eqn:OAMP_derive_u_ubar} into \eqref{Eqn:OAMP_derive_r_u} yields
\BS
\begin{align}
\begin{bmatrix}
\vr_1 \\
\vr_2 \\
\vdots \\
\vr_t
\end{bmatrix}
%&=\left(\bm{I}_t\otimes\bm{W}-\mathsf{B}_t\otimes\bm{I}_N\right)\begin{bmatrix}
%\bar{\bm{u}}_1\\
%\bar{\bm{u}}_2\\
%\vdots\\
%\bar{\bm{u}}_t
%\end{bmatrix}+\left(\bm{I}_t\otimes\bm{W}-\mathsf{B}_t\otimes\bm{I}_N\right)(\mathsf{D}_t \otimes \bm{I}_N)
%\begin{bmatrix}
%\bm{r}_1\\
%\bm{r}_2\\
%\vdots\\
%\bm{r}_t
%\end{bmatrix}\\
&=\left(\paraB{I}_t\otimes\bm{W}-\paraB{B}_t\otimes\bm{I}_N\right)\begin{bmatrix}
\bar{\bm{u}}_1\\
\bar{\bm{u}}_2\\
\vdots\\
\bar{\bm{u}}_t
\end{bmatrix}+\left(\paraB{D}_t \otimes\bm{W}- \paraB{B}_t\paraB{D}_t\otimes\bm{I}_N\right)
\begin{bmatrix}
\bm{r}_1\\
\bm{r}_2\\
\vdots\\
\bm{r}_t
\end{bmatrix}.
\end{align}
\ES
Re-arranging the equation, we have
\BE
\left[\bm{I}_{tN}-\left(\paraB{D}_t \otimes\bm{W}- \paraB{B}_t\paraB{D}_t\otimes\bm{I}_N\right)\right]
\begin{bmatrix}
\bm{r}_1\\
\bm{r}_2\\
\vdots\\
\bm{r}_t
\end{bmatrix}
=\left(\paraB{I}_t\otimes\bm{W}-\paraB{B}_t\otimes\bm{I}_N\right)
\begin{bmatrix}
\bar{\bm{u}}_1\\
\bar{\bm{u}}_2\\
\vdots\\
\bar{\bm{u}}_t
\end{bmatrix}.
\EE
Note that both $ \paraB{B}_t$ and $ \paraB{D}_t$ are lower triangular matrices. Moreover, $ \paraB{D}_t$ is strictly lower triangular. We can then verify that $\bm{I}_{tN}-\left(\paraB{D}_t \otimes\bm{W}- \paraB{B}_t\paraB{D}_t\otimes\bm{I}_N\right)$ is lower triangular with diagonal elements all equal to one, and thus invertible. Solving the above equation yields
\BE\label{Eqn:lemma_reparametrization2}
\begin{bmatrix}
\bm{r}_1\\
\bm{r}_2\\
\vdots\\
\bm{r}_t
\end{bmatrix}
=\left[\bm{I}_{tN}-\left(\paraB{D}_t \otimes\bm{W}- \paraB{B}_t\paraB{D}_t\otimes\bm{I}_N\right)\right]^{-1}\left(\paraB{I}_t\otimes\bm{W}-\paraB{B}_t\otimes\bm{I}_N\right)
\begin{bmatrix}
\bar{\bm{u}}_1\\
\bar{\bm{u}}_2\\
\vdots\\
\bar{\bm{u}}_t
\end{bmatrix}.
\EE
At this point, we have written $(\bm{r}_1,\ldots,\bm{r}_t)$ as a linear combination of $(\bar{\bm{u}}_1,\ldots,\bar{\bm{u}}_t)$. It remains to verify the equivalence between \eqref{Eqn:lemma_reparametrization2} and \eqref{Eqn:lemma_reparametrization}.  

Let $\bm{W}=\bm{O\Lambda O}^\UT$ be the eigenvalue decomposition of $\bm{W}$ with $\bm{\Lambda}:=\{\lambda_1,\ldots,\lambda_N\}$. Using this decomposition, we have
\BS
\begin{eqnarray*}
& &\left[\bm{I}_{tN}-\left(\paraB{D}_t \otimes\bm{W}- \paraB{B}_t\paraB{D}_t\otimes\bm{I}_N\right)\right]^{-1}\left(\paraB{I}_t\otimes\bm{W}-\paraB{B}_t\otimes\bm{I}_N\right)\\
&=&(\paraB{I}_t\otimes\bm{O})\left[\bm{I}_{tN}-\paraB{D}_t \otimes\bm{\Lambda}+\paraB{B}_t\paraB{D}_t\otimes\bm{I}_N\right]^{-1}\left(\paraB{I}_t\otimes\bm{\Lambda}-\paraB{B}_t\otimes\bm{I}_N\right)(\paraB{I}_t\otimes\bm{O})^\UT\\
&\explain{(a)}{=}&(\paraB{I}_t\otimes\bm{O})\left[\bm{\Pi}\left(\bm{I}_{tN}-\bm{\Lambda}\otimes \paraB{D}_t +\bm{I}_N\otimes \paraB{B}_t\paraB{D}_t\right)\bm{\Pi}^\UT\right]^{-1}\bm{\Pi}\left(\bm{\Lambda}\otimes \paraB{I}_t-\bm{I}_N\otimes \paraB{B}_t\right)\bm{\Pi}^\UT(\paraB{I}_t\otimes\bm{O})^\UT\\
&=&(\paraB{I}_t\otimes\bm{O})\bm{\Pi}\left[\bm{I}_{tN}-\bm{\Lambda}\otimes \paraB{D}_t +\bm{I}_N\otimes \paraB{B}_t\paraB{D}_t\right]^{-1}\left(\bm{\Lambda}\otimes \paraB{I}_t-\bm{I}_N\otimes \paraB{B}_t\right)\bm{\Pi}^\UT(\paraB{I}_t\otimes\bm{O})^\UT\\
%&=& (\bm{I}_t\otimes\bm{Q})\bm{\Pi}
%\begin{bmatrix}
%(\bm{I}_t-\lambda_1\para{D}_t+\para{B}_t\para{D}_t)^{-1}(\lambda_1\bm{I}_t-\para{B}_t) & &\\
%& \ddots &\\
%& & (\bm{I}_t-\lambda_N\para{D}_t+\para{B}_t\para{D}_t)^{-1}(\lambda_N\bm{I}_t-\para{B}_t) 
%\end{bmatrix}
%\bm{\Pi}^\UT(\bm{I}_t\otimes\bm{Q})^\UT\\
&=& (\paraB{I}_t\otimes\bm{O})\bm{\Pi}\,
\text{diag}\left\{(\paraB{I}_t-\lambda_1\paraB{D}_t+\paraB{B}_t\paraB{D}_t)^{-1}(\lambda_1\paraB{I}_t-\paraB{B}_t) ,\ldots,
 (\paraB{I}_t-\lambda_N\paraB{D}_t+\paraB{B}_t\paraB{D}_t)^{-1}(\lambda_N\paraB{I}_t-\paraB{B}_t) 
\right\}\,
\bm{\Pi}^\UT(\paraB{I}_t\otimes\bm{O})^\UT\\
&\explain{(b)}{=}&(\paraB{I}_t\otimes\bm{O})\,\bm{\Pi}\,
\text{diag}\left\{\paraB{P}(\lambda_1),\ldots,\paraB{P}(\lambda_N)\right\}\,
\bm{\Pi}^\UT\,(\paraB{I}_t\otimes\bm{O})^\UT\\
& \explain{(c)}{=} & (\paraB{I}_t\otimes\bm{O}) 
\begin{bmatrix}
{P}_{1,1}(\bm{\Lambda}) &  &  & \\
{P}_{2,1}(\bm{\Lambda}) &{P}_{2,2}(\bm{\Lambda})  &  & \\
\vdots & \vdots &\ddots  & \\
{P}_{t,1}(\bm{\Lambda}) & {P}_{t,2}(\bm{\Lambda}) &\cdots &{P}_{t,t}(\bm{\Lambda})
\end{bmatrix}
 (\paraB{I}_t\otimes\bm{O})^\UT\\
 & = &\begin{bmatrix}
{P}_{1,1}(\bm{W}) &  &  & \\
{P}_{2,1}(\bm{W}) &{P}_{2,2}(\bm{W})  &  & \\
\vdots & \vdots &\ddots  & \\
{P}_{t,1}(\bm{W}) & {P}_{t,2}(\bm{W}) &\cdots &{P}_{t,t}(\bm{W})
\end{bmatrix},
\end{eqnarray*}
\ES
where
\begin{itemize}
\item In step (a), $\bm{\Pi}\in\mathbb{R}^{t N\times t N}$ denotes the unique permutation matrix that reverses the order of the Kronecker product:
\BS\label{Eqn:Pi_def}
\BE
\bm{M}_t\otimes \bm{M}_N=\bm{\Pi}(\bm{M}_N\otimes \bm{M}_t)\bm{\Pi}^\UT,\quad \forall \bm{M}_t\in\mathbb{R}^{t\times t},\  \bm{M}_N\in\mathbb{R}^{N\times N}.
\EE
Specifically, let $\sigma:[t N]\mapsto [t N]$ be a permutation with the following mapping rule:
\BE
\sigma\left((i_1-1)N+i_2\right) =(i_2-1)t+i_1,\quad \forall i_1\in[t],\ i_2\in[N].
\EE
Then, the operation $\bm{A}\mapsto \bm{\Pi}\bm{A}\bm{\Pi}^\UT$ ($\forall \bm{A}\in\mathbb{R}^{t N\times t N}$) is a re-ordering of the entries of $\bm{A}$ according to
\BE
\bm{A}_{m,n}\mapsto \bm{A}_{\sigma(m),\sigma(n)},\quad \forall m,n\in[tN].
\EE
\ES
\item In step (b), we introduced the notation:
\[
\paraB{P}(\lambda_i):=(\paraB{I}_t-\lambda_i\paraB{D}_t+\paraB{B}_t\paraB{D}_t)^{-1}(\lambda_i\paraB{I}_t-\paraB{B}_t)\in\mathbb{R}^{t\times t},\quad\forall i\in[N].
\]
Note that since $\paraB{B}_t$ and $\paraB{D}_t$ are lower triangular matrices, $(\paraB{P}(\lambda_i))_{i\in[N]}$ are also lower triangular.
%\item In step (c), we denoted
%\[
%P_{m,n}(\bm{\Lambda}):=\text{diag}\left\{\bm{P}(\lambda_1)[m,n],\bm{P}(\lambda_2)[m,n],\ldots,\bm{P}(\lambda_N)[m,n]\right\},\quad \forall m\in[t], n\in[m].
%\]
%where $\bm{P}(\lambda_i)[m,n]$ denotes the $(m,n)$th element of $\bm{P}(\lambda_i)$. Step (c) is consequence of the permutation matrix $\bm{\Pi}$.
\item Step (c) is consequence of the definition of the permutation matrix $\bm{\Pi}$, cf.~\eqref{Eqn:Pi_def}. Here, $(P_{i,j})_{1\le i\le t,1\le j\le i}$ are understood as a sequence of scalar functions, which are defined as follows:
\BS
\begin{align}
\begin{bmatrix}
P_{1,1}(\lambda)& 0 & \cdots  & 0 \\
P_{2,1}(\lambda)& P_{2,2}(\lambda)& \cdots  &  0 \\
\vdots & \vdots & \ddots & \\
P_{t,1}(\lambda)& P_{t,1}(\lambda)& \cdots  &  P_{t,t}(\lambda)
\end{bmatrix}
&=(\paraB{I}_t-\lambda\paraB{D}_t+\paraB{B}_t\paraB{D}_t)^{-1}(\lambda\paraB{I}_t-\paraB{B}_t),\quad\forall \lambda\in\mathbb{R}.\label{Eqn:step_c_proof}
%&=\sum_{i=0}^{t-1}\left(\lambda\bm{D}_t-\bm{B}_t\bm{D}_t\right)^{i}(\lambda\bm{I}_t-\bm{B}_t),
\end{align}
\ES
Note that the lower triangular structure of the matrix on the LHS of \eqref{Eqn:step_c_proof} is a consequence of the lower triangular structure of the matrix on the RHS, which is further due to the lower triangular structures of $\paraB{B}_t$ and $\paraB{D}_t$.
\end{itemize}
This completes the proof.
%\end{proof}

%===============================
\subsection{Choice of De-biasing Matrix (Proof of Lemma \ref{Lem:RI_AMP_zero_trace})}\label{App:zero_trace}

%\begin{proof}

\textbf{Proof of (1).} For $t=1$, we have $\paraB{B}_1=\para{b}_{1,1}$ and $\paraB{D}_1=0$. Then, $\paraB{P}_1(\lambda)=\lambda -\para{b}_{1,1}$. Clearly, $\para{b}_{1,1}=\mathbb{E}[\mathsf{\Lambda}]$ is the unique solution to $\mathbb{E}[\paraB{P}_1(\mathsf{\Lambda})]=\bm{0}_{t\times t}$. In what follows, we assume $t\ge2$. We partition $\paraB{B}_t$ (cf.~\eqref{Eqn:B_t_def}) and $\paraB{D}_t$ (cf.~\eqref{Eqn:D_def}) as follows
\BE\label{Eqn:Bt_partition}
\paraB{B}_t
=
\begin{bmatrix}
\paraB{B}_{t-1} & \bm{0}_{t-1\times 1}\\
{\paraB{b}}_{t} & \para{b}_{t,t} 
\end{bmatrix},
\quad
\paraB{D}_t
=
\begin{bmatrix}
\paraB{D}_{t-1} & \bm{0}_{t-1\times 1}\\
{\paraB{d}}_{t} & 0
\end{bmatrix},
\EE
where $\paraB{b}_t,\paraB{d}_t\in\mathbb{R}^{1\times (t-1)}$. Then, 
\BS
\begin{align*}
\paraB{P}_t(\lambda)&=(\paraB{I}_t-\lambda\paraB{D}_t+\paraB{B}_t\paraB{D}_t)^{-1}(\lambda\paraB{I}_t-\paraB{B}_t)\\
&=
\begin{bmatrix}
\paraB{I}_{t-1}-(\lambda\paraB{I}_{t-1}-\paraB{B}_{t-1})\paraB{D}_{t-1} & \bm{0}_{t-1\times 1}\\
{\paraB{b}}_{t}\paraB{D}_{t-1}-(\lambda-\para{b}_{t,t}\paraB{d}_t) & 1
\end{bmatrix}^{-1}
\begin{bmatrix}
\lambda\paraB{I}_{t-1}-\paraB{B}_{t-1} & \bm{0}_{t-1\times 1}\\
-{\paraB{b}}_{t}& \lambda-\para{b}_{t,t}
\end{bmatrix}\\
&\explain{(a)}{=}
\begin{bmatrix}
\left[\paraB{I}_{t-1}-(\lambda\paraB{I}_{t-1}-\paraB{B}_{t-1}\paraB{D}_{t-1})\right]^{-1} & \bm{0}_{t-1\times 1}\\
\left(-\paraB{b}_{t}\paraB{D}_{t-1}+(\lambda-\para{b}_{t,t})\paraB{d}_t\right)\left[\paraB{I}_{t-1}-(\lambda\paraB{I}_{t-1}-\paraB{B}_{t-1}\paraB{D}_{t-1})\right]^{-1}& 1
\end{bmatrix}
\begin{bmatrix}
\lambda\paraB{I}_{t-1}-\paraB{B}_{t-1} & \bm{0}_{t-1\times 1}\\
-{\paraB{b}}_{t}& \lambda-\para{b}_{t,t}
\end{bmatrix}\\
&\explain{(b)}{=}
\begin{bmatrix}
\paraB{P}_{t-1}(\lambda)& \bm{0}_{t-1\times 1}\\
\left(-\paraB{b}_{t}\paraB{D}_{t-1}+(\lambda-\para{b}_{t,t})\paraB{d}_t\right)\paraB{P}_{t-1}(\lambda)-\paraB{b}_t& \lambda-\para{b}_{t,t},
\end{bmatrix}
\end{align*}
\ES
where step (a) is due to the matrix inverse lemma and step (b) is is due to the definition of $\paraB{P}_{t-1}(\lambda)$,
Now, enforcing the condition $\mathbb{E}\left[ \paraB{P}_t(\mathsf{\Lambda}) \right]=\bm{0}_{t\times t}$ gives us
\BS
\begin{align}
&\mathbb{E}\left[ \paraB{P}_{t-1}(\mathsf{\Lambda}) \right]=\bm{0}_{t-1\times t-1},\\
&\mathbb{E}\left[\left(-\paraB{b}_{t}\paraB{D}_{t-1}+(\mathsf{\Lambda}-\para{b}_{t,t})\paraB{d}_t\right)\paraB{P}_{t-1}(\mathsf{\Lambda})-\paraB{b}_t\right]=\bm{0}_{1\times t-1},\\
& \mathbb{E}\left[\mathsf{\Lambda}-\para{b}_{t,t}\right]=0.
\end{align}
\ES
Notice that $\paraB{P}_{t-1}(\lambda)$ only depends on the sub-matrices $\paraB{B}_{t-1}$ and $\paraB{D}_{t-1}$ (cf.~\eqref{Eqn:Bt_partition}) but not on the last rows of $\paraB{B}_{t}$ and $\paraB{D}_t$. Suppose that the equation $\mathbb{E}\left[ \paraB{P}_{t-1}(\mathsf{\Lambda}) \right]=\bm{0}_{t-1\times t-1}$ uniquely determines the sub-matrix $\paraB{B}_{t-1}$ (as a function of $\paraB{D}_{t-1}$). The last row of $\paraB{B}_{t-1}$, namely $(\paraB{b}_t,\para{b}_{t,t})$ are uniquely determined as 
\BS
\begin{align*}
\para{b}_{t,t} &= m_1,\\
\paraB{b}_t &=-\paraB{b}_{t}\paraB{D}_{t-1}\mathbb{E}\left[\paraB{P}_{t-1}(\mathsf{\Lambda})\right] +\paraB{d}_t\mathbb{E}\left[\left(\mathsf{\Lambda}-m_1\right)\paraB{P}_{t-1}(\mathsf{\Lambda})\right]\\
&=\paraB{d}_t\mathbb{E}\left[\left(\mathsf{\Lambda}-m_1\right)\paraB{P}_{t-1}(\mathsf{\Lambda})\right]\\
&=\paraB{d}_t\mathbb{E}\left[\mathsf{\Lambda}\paraB{P}_{t-1}(\mathsf{\Lambda})\right],
\end{align*}
\ES
where we have used $\mathbb{E}\left[ \paraB{P}_{t-1}(\mathsf{\Lambda}) \right]=\bm{0}_{t-1\times t-1}$. Hence, we have shown that, if $\mathbb{E}\left[ \paraB{P}_{t-1}(\mathsf{\Lambda}) \right]=\bm{0}_{t-1\times t-1}$ uniquely determines the sub-matrix $\paraB{B}_{t-1}$, then $\mathbb{E}\left[ \paraB{P}_{t}(\mathsf{\Lambda}) \right]=\bm{0}_{t\times t}$ uniquely determines the matrix $\paraB{B}_{t}$. Hence, the claim holds by induction.

\textbf{Proof of (2) and (4).} Claim (1) shows that the solution \eqref{Eqn:trace_free_eqn} is unique. It remains to verify that this a solution $\paraB{B}_t$ can be represented as a polynomial of $\paraB{D}_t$. Moreover, the coefficients of the polynomial $(\alpha_i)_{i\ge1}$ only depend on the law $\mu$, but not on $\paraB{D}_t$. In other words, there exists a single polynomial that solves \eqref{Eqn:trace_free_eqn} for all $\paraB{D}_t$. Assume that
\BE\label{Eqn:Bt_poly}
\paraB{B}_t = \sum_{i=1}^t \alpha_i\paraB{D}_t^{i-1},
\EE
where the sequence $(\alpha_i)_{i\ge1}$ are yet to be determined. (Note that $\paraB{D}_t$ is strictly lower triangular and hence $\paraB{D}_t^{i}=\bm{0}_{t\times t},\forall i\ge t$.) The matrix $\paraB{P}_t$ can be written as (cf.~\eqref{Eqn:Pt_inverse_def}):
\BS
\begin{align*}
\paraB{P}_t(\lambda)&=(\paraB{I}_t-\lambda\paraB{D}_t+\paraB{B}_t\paraB{D}_t)^{-1}(\lambda\paraB{I}_t-\paraB{B}_t)\\
&=\sum_{i=1}^t\left(\lambda\paraB{D}_t-\paraB{B}_t\paraB{D}_t\right)^{i-1}(\lambda\paraB{I}_t-\paraB{B}_t),
\end{align*}
\ES
where the second step is due to the following identity for a strictly lower triangular matrix $\paraB{A}\in\mathbb{R}^{t\times t}$: $(\paraB{I}_t-\paraB{A})^{-1}=\paraB{I}+\paraB{A}+\paraB{A}^2+\cdots+\paraB{A}^{t-1}$. Clearly, assuming $\paraB{B}_t$ is a polynomial of $\paraB{D}_t$, $\paraB{P}_t$ is also a polynomial of $\paraB{D}_t$, which we denote as:
\BE\label{Eqn:Pt_poly}
\paraB{P}_t(\lambda)=\sum_{i=1}^t Q_i(\lambda)\paraB{D}_t^{i-1},
\EE
where the coefficients $(Q_i(\lambda))_{i\ge1}$ depend on $\lambda$. We next provide a recursive characterization of $(Q_i(\lambda))_{i\ge1}$ in terms of $(\alpha_i)_{i\ge1}$. From the definition of $\paraB{P}_t(\lambda)$ (cf.~\eqref{Eqn:Pt_inverse_def}), the following identity holds
\BE\label{Eqn:identically_zero_0}
(\paraB{I}_t-\lambda\paraB{D}_t+\paraB{B}_t\paraB{D}_t)\paraB{P}_t(\lambda)-\lambda\paraB{I}_t+\paraB{B}_t=\bm{0}_{t\times t}.
\EE
Using the polynomial representations of $\paraB{B}_t$ and $\paraB{P}_t(\lambda)$ in \eqref{Eqn:Bt_poly} and \eqref{Eqn:Pt_poly}, we write the LHS of the above equation as
\BS\label{Eqn:identically_zero}
\begin{align}
&(\paraB{I}_t-\lambda\paraB{D}_t+\paraB{B}_t\paraB{D}_t)\paraB{P}_t(\lambda)-\lambda\paraB{I}_t+\paraB{B}_t\\
&=\left(\paraB{I}_t-\lambda\paraB{D}_t+\sum_{i=1}^\infty \alpha_i\paraB{D}_t^{i}\right)\sum_{i=1}^\infty Q_i(\lambda)\paraB{D}_t^{i-1}-\lambda\paraB{I}_t+\sum_{i=1}^\infty\alpha_i\paraB{D}_t^{i-1}\\
&=\sum_{i=1}^\infty Q_i(\lambda)\paraB{D}_t^{i-1} -\sum_{i=1}^\infty \lambda Q_i(\lambda)\paraB{D}_t^{i}+\sum_{i=1}^\infty\sum_{j=1}^\infty \alpha_i Q_j(\lambda)\paraB{D}_t^{i+j-1}-\lambda\paraB{I}_t+\sum_{i=1}^\infty\alpha_i\paraB{D}_t^{i-1}\\
%&\explain{(a)}{=}\left(\bm{I}_t-\lambda\bm{D}_t+\sum_{i=1}^\infty \alpha_i\bm{D}_t^{i}\right)\sum_{i=1}^\infty q_i(\lambda)\bm{D}_t^{i-1}-\lambda\bm{I}_t+\sum_{i=1}^\infty\alpha_i\bm{D}_t^{i-1}\\
&=\sum_{i=1}^\infty Q_i(\lambda)\paraB{D}_t^{i-1} -\sum_{i=1}^\infty \lambda Q_i(\lambda)\paraB{D}_t^{i}+\sum_{i=1}^\infty\sum_{k=1+i}^\infty\alpha_iQ_{k-i}(\lambda)\paraB{D}_t^{k-1}-\lambda\paraB{I}_t+\sum_{i=1}^\infty\alpha_i\paraB{D}_t^{i-1}\\
&\explain{(a)}{=}\sum_{i=1}^\infty Q_i(\lambda)\paraB{D}_t^{i-1} -\sum_{i=2}^\infty \lambda Q_{i-1}(\lambda)\paraB{D}_t^{i-1}+\sum_{k=2}^\infty\left(\sum_{i=1}^{k-1}\alpha_iQ_{k-i}(\lambda)\right)\paraB{D}_t^{k-1}-\lambda\paraB{I}_t+\sum_{i=1}^\infty\alpha_i\paraB{D}_t^{i-1}\\
&=(Q_1(\lambda)-\lambda+\alpha_1)\paraB{I}_t+\sum_{n=2}^\infty\left(Q_n(\lambda)-\lambda Q_{n-1}(\lambda)+\sum_{i=1}^{n-1}\alpha_iQ_{n-i}(\lambda)+\alpha_n\right)\paraB{D}_t^{n-1},
\end{align}
\ES
where step (a) is due to a swap of the summation order. Note that we have represented the finite order polynomials of $\paraB{D}_t$ as infinite power series (since $\paraB{D}_t^i=\bm{0}_{t\times t},\forall i\ge t$) in the above equation. This representation is somewhat more convenient for calculations. Since \eqref{Eqn:identically_zero} is identically zero (from \eqref{Eqn:identically_zero_0}), we must have that
\BS\label{Eqn:qn_recursion}
\begin{align}
Q_1(\lambda)&=\lambda-\alpha_1,\\
Q_n(\lambda)&=\lambda Q_{n-1}(\lambda)-\sum_{i=1}^{n-1}\alpha_iQ_{n-i}(\lambda)-\alpha_n,\quad\forall n\ge2,
\end{align}
\ES
which provides a recursive definition of $(Q_i(\lambda))_{i\ge1}$ in terms of $(a_i)_{i\ge1}$. Recall that $\paraB{P}_t(\lambda)=\sum_{i=1}^t Q_i(\lambda)\paraB{D}_t^{i-1}$. Therefore, if $\mathbb{E}[Q_i(\mathsf{\Lambda})]=0,\forall i\in[t]$, then $\mathbb{E}\left[\paraB{P}_t(\mathsf{\Lambda})\right]=\bm{0}_{t\times t}$, $\forall \paraB{D}_t$. Taking expectations over $\mathsf{\Lambda}\sim\mu$ in \eqref{Eqn:qn_recursion} and setting $\mathbb{E}[Q_i(\mathsf{\Lambda})]=0,\forall i\in[t]$ leads to the following choice of $(\alpha_i)_{i\ge1}$:
\BS
\begin{align}
\alpha_1&=\mathbb{E}[\mathsf{\Lambda}],\\
\alpha_n&=\mathbb{E}[\mathsf{\Lambda} Q_{n-1}(\mathsf{\Lambda})],\quad \forall n\ge2.
\end{align}
\ES
Substituting the above equation into \eqref{Eqn:qn_recursion} leads to the following recursive characterization of $(Q_i(\lambda))_{i\ge1}$:
\BS\label{Eqn:qn_recursion2}
\begin{align}
Q_n(\lambda)&=\lambda Q_{n-1}(\lambda)-\sum_{i=1}^{n-1}\alpha_iQ_{n-i}(\lambda)-\alpha_n\\
&=\lambda Q_{n-1}(\lambda)-\sum_{i=1}^{n-1}\mathbb{E}[\mathsf{\Lambda} Q_{i-1}(\mathsf{\Lambda})]\cdot Q_{n-i}(\lambda)-\mathbb{E}[\mathsf{\Lambda} Q_{n-1}(\mathsf{\Lambda})]\\
&=\lambda Q_{n-1}(\lambda)-\sum_{i=1}^{n}\mathbb{E}[\mathsf{\Lambda} Q_{i-1}(\mathsf{\Lambda})]\cdot Q_{n-i}(\lambda),
\end{align}
\ES
where we defined $Q_0(\lambda):=1$ in the last step. This proves the claim.

\textbf{Proof of (3).} Item (2) shows that the matrix $\paraB{B}_t$ that solves \eqref{Eqn:trace_free_eqn} is a polynomial in $\paraB{D}_t$ with coefficients characterized by \eqref{Eqn:poly_recursive_def}. The claimed result is then a consequence of this recursion together with Proposition \ref{Eqn:Q_def}.
%\end{proof}

%=================================================
\subsection{State Evolution of RI-AMP (Proof of Theorem \ref{Th:AMP_SE})}\label{App:AMP_SE}

We first address some subtle points in reducing RI-AMP to OAMP. We then apply the general state evolution result of OAMP to calculate the claimed update equation of the covariance matrix $\paraB{\bm{\Sigma}}_t$.

\paragraph{Reduction to OAMP:} From Lemma \ref{Lem:FOM_D}, the iterates $\bm{r}_t$ in RI-AMP can be written as (cf.~\eqref{Eqn:lemma_reparametrization_recall})
\BS\label{Eqn:OAMP_RI_AMP_app}
\begin{align}
\bm{r}_t &= \hat{P}_{t,1}(\bm{W})\bar{\bm{u}}_1+\cdots + \hat{P}_{t,t}(\bm{W})\bar{\bm{u}}_t,\\
{\bm{u}}_{t+1} &= \eta_{t+1} (\vr_1,\ldots,\vr_t),\\
\bar{\bm{u}}_{t+1} &= \eta_{t+1} (\vr_1,\ldots,\vr_t)-\left(\langle\partial_1 \bm{u}_{t+1}\rangle\cdot\bm{r}_1+\cdots+\langle\partial_{t} \bm{u}_{t+1}\rangle \cdot\bm{r}_{t}\right),
\end{align}
\ES
where the sequence of functions $(\hat{P}_{t,i}(\lambda))_{1\le t,1\le i\le t}$ are given by the last row of $\hat{\paraB{P}}_t(\lambda)\in\mathbb{R}^{t\times t}$:
\BE
\hat{\paraB{P}}_t(\lambda)=\sum_{i=1}^t Q_i(\lambda)\hat{\paraB{\Phi}}_t^{i-1}.
\EE
The above iterations can be written into an OAMP algorithm as defined in Definition \ref{Def:LM_OAMP} by introducing intermediate variables $\bm{z}_{t,i} = \hat{P}_{t,i}(\bm{W})\bar{\bm{u}}_i$ and properly re-indexing the iterates. The only subtle point here is that the functions $(\hat{P}_{t,i}(\lambda))_{1\le t,1\le i\le t}$, which depend on the empirical divergences, are random and satisfy the required trace-free condition in OAMP only in certain asymptotical sense. Nevertheless, similar to the proof in Appendix \ref{App:proof_OAMP}, we can use a simple approximation argument to show that this difference is asymptotically negligible. More specifically, we note that if we replace the functions $(\hat{P}_{t,i}(\lambda))_{1\le t,1\le i\le t}$ by $({P}_{t,i}(\lambda))_{1\le t,1\le i\le t}$, which are given by the last row of the following matrix
\BE
{\paraB{P}}_t(\lambda)=\sum_{i=1}^t Q_i(\lambda){\paraB{\Phi}}_t^{i-1},
\EE
then the state evolution in Appendix \ref{App:proof_OAMP} applies. Here, the empirical divergences $\hat{\paraB{\Phi}}_t$ are replaced by their limits ${\paraB{\Phi}}_t:=\underset{N\to\infty}{\text{plim}}\ \hat{\paraB{\Phi}}_t$. 
(Note that the above limit holds by an inductive argument: we assume $\hat{\paraB{\Phi}}_t\overset{\mathbb{P}}{\to}{\paraB{\Phi}}_t$, then prove the state evolution using our following arguments, and then show $\hat{\paraB{\Phi}}_{t+1}\overset{\mathbb{P}}{\to}{\paraB{\Phi}}_{t+1}$.) This matrix satisfies the trace-free condition: $\mathbb{E}_{\mathsf{\Lambda}\sim\mu}\left[\paraB{P}_t(\mathsf{\Lambda})\right]=\bm{0}_{t\times t}$.

Then, we bound the approximation error between these two versions of OAMP, based on the same arguments used in Appendix \ref{App:proof_OAMP}. We do not repeat the full argument, but it suffices to show $\|P_{t,i}(\bm{W})-\hat{P}_{t,i}(\bm{W})\|_{\text{op}}\overset{\mathbb{P}}{\to}0$ for all $i\in[t]$. Note that $\hat{P}_{t,i}(\lambda)$ and $P_{t,i}(\lambda)$ are linear combinations of $(Q_i(\lambda))_{i\ge1}$, which we denote as
\BS
\begin{align}
\hat{P}_{t,i}(\lambda) &:=\hat{\alpha}_{t,1}Q_1(\lambda)+\hat{\alpha}_{t,2}Q_2(\lambda)+\cdots+\hat{\alpha}_{t,t}Q_t(\lambda),\\
{P}_{t,i}(\lambda) &:={\alpha}_{t,1}Q_1(\lambda)+{\alpha}_{t,2}Q_2(\lambda)+\cdots+{\alpha}_{t,t}Q_t(\lambda).
\end{align}
\ES
Then, the following holds for all $t\ge1$ and $1\le i\le t$:
\BS
\begin{align}
\|P_{t,i}(\bm{W})-\hat{P}_{t,i}(\bm{W})\|_{\text{op}} &=\left\|\sum_{i=1}^t (\alpha_{t,i}-\hat{\alpha}_{t,i})Q_i(\bm{W})\right\|_{\text{op}}\le  \sum_{i=1}^t|\alpha_{t,i}-\hat{\alpha}_{t,i}|\cdot \|Q_i(\bm{W})\|_{\text{op}}\overset{\mathbb{P}}{\longrightarrow}0,
\end{align}
\ES
where the last step is due to the assumption that $\|\bm{W}\|_{\text{op}}$ is bounded by an $N$-independent constant (see Assumption \ref{Ass:RI-AMP}) and $(Q_i)_{i\ge1}$ are continuous functions (in the current case, polynomials), and $\hat{\alpha}_{t,i}\overset{\mathbb{P}}{\to}{\alpha}_{t,i},\forall i=1,\ldots,t$ (from the fact that $\hat{\paraB{\Phi}}_t\overset{\mathbb{P}}{\to}{\paraB{\Phi}}_t$).

\paragraph{Covariance matrix in state evolution:} From the above arguments, we can apply Theorem \ref{The:OAMP_SE} to show that the empirical distributions of $(\bm{r}_1,\ldots,\bm{r}_t)$ converges to a joint Gaussian distribution with zero mean and covariance
\BS
\begin{align}
\paraB{\Sigma}_t &\explain{(a)}{=}\mathbb{E}\left[\paraB{P}_t(\mathsf{\Lambda})
\begin{bmatrix}
\bar{\mathsf{U}}_1\\
\vdots\\
\bar{\mathsf{U}}_t
\end{bmatrix}
\begin{bmatrix}
\bar{\mathsf{U}}_1,\ldots,\bar{\mathsf{U}}_t
\end{bmatrix}
\paraB{P}_t(\mathsf{\Lambda})^\UT
\right],\quad \mathsf{\Lambda}\sim\mu,\\
&\explain{(b)}{=}
\mathbb{E}\left[\paraB{P}_t(\mathsf{\Lambda})
\bar{\paraB{\Delta}}_t
\paraB{P}_t(\mathsf{\Lambda})^\UT,
\right]
% &=\mathbb{E}\left[ \left(\sum_{i=1}^t Q_i(\mathsf{\Lambda}){\paraB{\Phi}}_t^{i-1}\right)
% \bar{\paraB{\Delta}}_t
% \left(\sum_{i=1}^t Q_i(\mathsf{\Lambda}){\paraB{\Phi}}_t^{i-1}\right)^\UT
% \right]\\
% &=\sum_{i=1}^t\sum_{j=1}^t \mathbb{E}_{\mathsf{\Lambda}\sim\mu}\left[Q_i(\mathsf{\Lambda})Q_j(\mathsf{\Lambda})\right]\cdot \paraB{\Phi}_t^{i-1}\, \bar{\paraB{\Delta}}_t\, (\paraB{\Phi}_t^{j-1})^\UT ,
\end{align}
\ES
where step (a) follows Theorem \ref{The:OAMP_SE} with the empirical divergences replaced by the population-level divergences, namely (cf.~\eqref{Eqn:Pt_Phi_t_lemma1}), ${\paraB{P}}_t(\lambda)=\sum_{i=1}^t Q_i(\lambda){\paraB{\Phi}}_t^{i-1}$; step (b) is from the fact that the state evolution random variables $(\bar{\mathsf{U}}_1,\ldots,\bar{\mathsf{U}}_t)$ are independent of $\mathsf{\Lambda}$.

%=================================================
\subsection{Equivalence of State Evolution (Proof of Proposition \ref{Pro:equivalence})}\label{App:SE_equivalence}

The following lemma will be used in our proof of Proposition \ref{Pro:equivalence}.

\begin{lemma}\label{Lem:QQ_kappa}
Let $(Q_n)_{n\ge1}$ and be defined as in \eqref{Eqn:Q_def}. Then, the following holds
\BE
\mathbb{E}\left[Q_IQ_J\right] = \sum_{m=1}^I\sum_{n=1}^J \mathbb{E}\left[Q_{I-m}Q_{J-n}\right]\cdot \kappa_{n+m},\quad\forall I,J\ge1.
\EE
\end{lemma}

\begin{proof}
From Proposition \ref{Lem:cumulants}, \eqref{Eqn:Q_def} can be rewritten as
\BE\label{Eqn:Q_recursion_app}
Q_n= \mathsf{\Lambda}Q_{n-1}-\sum_{i=1}^n \kappa_i\cdot Q_{n-i},\quad\forall n\ge1,
\EE
with $Q_0:=1$. Using the above identity, we can relate $Q_IQ_J$ and $Q_{I-1}Q_{J+1}$ as follows:
\BS\label{Eqn:Q_IJ_1}
\begin{align}
Q_IQ_J &\explain{(a)}{=}\left(\mathsf{\Lambda}Q_{I-1}-\sum_{i=1}^I \kappa_i Q_{I-i}\right)Q_J\\
&=\mathsf{\Lambda}Q_{I-1}Q_J-\sum_{i=1}^I\kappa_i Q_{I-i}Q_J\\
&=\mathsf{\Lambda}Q_{I-1}Q_J- \sum_{j=1}^{J+1}\kappa_j Q_{I-1}Q_{J+1-j}  +\sum_{j=1}^{J+1}\kappa_j Q_{I-1}Q_{J+1-j}  -\sum_{i=1}^I\kappa_iQ_{I-i}Q_J\\
&=Q_{I-1}\Big(\mathsf{\Lambda}Q_J-\sum_{j=1}^{J+1}\kappa_j Q_{J+1-j}\Big)+ Q_{I-1}\Big(\sum_{j=1}^{J+1}\kappa_j Q_{J+1-j} \Big) -\sum_{i=1}^I\kappa_i Q_{I-i}Q_J\\
&\explain{(b)}{=}Q_{I-1}Q_{J+1} + Q_{I-1}\Big(\sum_{j=1}^{J+1}\kappa_j Q_{J+1-j} \Big) -\sum_{i=1}^I\kappa_iQ_{I-i}Q_J,
\end{align}
\ES
where both step (a) and step (b) used the identity \eqref{Eqn:Q_recursion_app}. We can apply the same manipulations to relate $Q_{I-1}Q_{J+1} $ and $Q_{I-2}Q_{J+2} $. Continuing for $\ell$ steps eventually us the following identity between $Q_{I-1}Q_{J+1} $ and $Q_{I-\ell}Q_{J+\ell} $:
\BE
Q_IQ_J   = Q_{I-\ell}Q_{J+\ell}+\sum_{k=1}^\ell\sum_{j=1}^{J+k}\kappa_jQ_{I-k}Q_{J+k-j}-\sum_{k=1}^\ell\sum_{i=1}^{I-k}\kappa_iQ_{I-k+1-i}Q_{J+k-1}.
\EE
Setting $\ell=I$ in the above identity yields
\BS\label{Eqn:app_QQ_1}
\begin{align}
Q_IQ_J  & = Q_{0}Q_{J+I}+\sum_{k=1}^{I}\sum_{j=1}^{J+k}\kappa_jQ_{I-k}Q_{J+k-j}-\sum_{k=1}^{I}\sum_{i=1}^{I-k}\kappa_iQ_{I-k+1-i}Q_{J+k-1}\\
&=Q_{J+I}+\sum_{k=1}^{I}\sum_{j=1}^{J+k}\kappa_jQ_{I-k}Q_{J+k-j}-\sum_{k=1}^{I}\sum_{i=1}^{I-k}\kappa_iQ_{I-k+1-i}Q_{J+k-1}\quad (Q_0=1)\\
&=Q_{J+I}+\underbrace{\sum_{k=1}^{I}\sum_{j=k+1}^{J+k}\kappa_jQ_{I-k}Q_{J+k-j}}_{\textsf{Term I}}+\underbrace{\sum_{k=1}^{I}\sum_{j=1}^{k}\kappa_jQ_{I-k}Q_{J+k-j}-\sum_{k=1}^{I}\sum_{i=1}^{I-k}\kappa_iQ_{I-k+1-i}Q_{J+k-1}}_{\textsf{Term II}},
\end{align}
\ES
where in the last step we split the sum over $j$ into two terms. By a change of variable, we rewrite $\textsf{Term I}$ as
\BE\label{Eqn:app_QQ_2}
\textsf{Term I}=\sum_{k=1}^{I}\sum_{j'=1}^{J}\kappa_{k+j'}Q_{I-k}Q_{J-j'}.
\EE
Note that $\textsf{Term II}$ involves $(Q_i)_{J+1\le i\le I+J-1 }$, which are higher order terms that do not appear in the desired result. It turns out that $\textsf{Term II}$ vanishes:
\BS\label{Eqn:app_QQ_3}
\begin{align}
\textsf{Term II} &= \sum_{k=1}^{I}\sum_{j=1}^{k}\kappa_jQ_{I-k}Q_{J+k-j}-\sum_{k=1}^{I}\sum_{i=1}^{I-k}\kappa_iQ_{I-k+1-i}Q_{J+k-1}\\
%&=\sum_{j=1}^{I-1}\sum_{k=j}^{I-1}\kappa_jQ_{I-k}Q_{J+k-j}-\sum_{k=1}^{I-1}\sum_{i=1}^{I-k}\kappa_iQ_{I-k+1-i}Q_{J+k-1}\quad (\text{change orders of summation})\\
%&=\sum_{j=1}^{I-1}\sum_{k'=1}^{I-j}\kappa_jQ_{I-(k'+j-1)}Q_{J+k'-1}-\sum_{k=1}^{I-1}\sum_{i=1}^{I-k}\kappa_iQ_{I-k+1-i}Q_{J+k-1}\quad (k:=k'+j-1)\\
&\explain{(a)}{=}\sum_{k'=1}^{I}\sum_{i=1}^{I-k'}\kappa_iQ_{I-(k'+i-1)}Q_{J+k'-1}-\sum_{k=1}^{I}\sum_{i=1}^{I-k}\kappa_iQ_{I-k+1-i}Q_{J+k-1},\\
&=0,
\end{align}
\ES
where step (a) is due to a change of variable $(k,j)\mapsto(k',i)$ via the map
\[
\begin{bmatrix}
k'\\
i
\end{bmatrix}
=
\begin{bmatrix}
1 & -1\\
0 & 1
\end{bmatrix}
\begin{bmatrix}
k\\
j
\end{bmatrix}
+
\begin{bmatrix}
1\\
0
\end{bmatrix}.
\]
Note that the map is one-to-one from $\{(k,j):1\le k\le I,1\le j\le k\}$ to $\{(k',i):1\le k'\le I,1\le i\le I-k'\}$, and hence the reformulation of the nested sums holds.

Summarizing \eqref{Eqn:app_QQ_1}-\eqref{Eqn:app_QQ_3} yields
\BE
Q_IQ_J  = Q_{J+I} + \sum_{k=1}^{I}\sum_{j=1}^{J}\kappa_{k+j}Q_{I-k}Q_{J-j}.
\EE
The claimed identity then follows from taking expectation in the above identity and recalling that the random variables $(Q_i)_{i\ge1}$ have zero mean.
\end{proof}

\paragraph{Proof of Proposition \ref{Pro:equivalence}.} Recall the following definitions that will be used in our proof:
\BS\label{Eqn:matrix_def_app}
\begin{align}
\paraB{\Sigma}_t&:=
\begin{bmatrix}
\mathbb{E}\left[{\mathsf{R}}_1^2\right]&\mathbb{E}\left[{\mathsf{R}}_1{\mathsf{R}}_2\right]   &  \cdots & \mathbb{E}\left[{\mathsf{R}}_1{\mathsf{R}}_t\right]  \\
\mathbb{E}\left[{\mathsf{R}}_2{\mathsf{R}}_1\right] & \mathbb{E}\left[{\mathsf{R}}_2^2\right]   &  \cdots &\mathbb{E}\left[{\mathsf{R}}_2{\mathsf{R}}_t\right]  \\
%\langle\partial_1 \bm{u}_3\rangle& \langle\partial_2 \bm{u}_3 \rangle &0 & &\\
\vdots& \vdots &\ddots  & \vdots \\
\mathbb{E}\left[{\mathsf{R}}_t{\mathsf{R}}_1\right] & \mathbb{E}\left[{\mathsf{R}}_t{\mathsf{R}}_2\right]   & \cdots  & \mathbb{E}\left[{\mathsf{R}}_t^2\right] 
\end{bmatrix},\\
\paraB{\Delta}_t&:=
\begin{bmatrix}
\mathbb{E}\left[{\mathsf{U}}_1^2\right]&\mathbb{E}\left[{\mathsf{U}}_1{\mathsf{U}}_2\right]   &  \cdots & \mathbb{E}\left[{\mathsf{U}}_1{\mathsf{U}}_t\right]  \\
\mathbb{E}\left[{\mathsf{U}}_2{\mathsf{U}}_1\right] & \mathbb{E}\left[{\mathsf{U}}_2^2\right]   &  \cdots &\mathbb{E}\left[{\mathsf{U}}_2{\mathsf{U}}_t\right]  \\
%\langle\partial_1 \bm{u}_3\rangle& \langle\partial_2 \bm{u}_3 \rangle &0 & &\\
\vdots& \vdots &\ddots  & \vdots \\
\mathbb{E}\left[{\mathsf{U}}_t{\mathsf{U}}_1\right] & \mathbb{E}\left[{\mathsf{U}}_t{\mathsf{U}}_2\right]   & \cdots  & \mathbb{E}\left[{\mathsf{U}}_t^2\right] 
\end{bmatrix},\\
\bar{\paraB{\Delta}}_t& :=
\begin{bmatrix}
\mathbb{E}\left[\bar{\mathsf{U}}_1^2\right]&\mathbb{E}\left[\bar{\mathsf{U}}_1\bar{\mathsf{U}}_2\right]   &  \cdots & \mathbb{E}\left[\bar{\mathsf{U}}_1\bar{\mathsf{U}}_t\right]  \\
\mathbb{E}\left[\bar{\mathsf{U}}_2\bar{\mathsf{U}}_1\right] & \mathbb{E}\left[\bar{\mathsf{U}}_2^2\right]   &  \cdots &\mathbb{E}\left[\bar{\mathsf{U}}_2\bar{\mathsf{U}}_t\right]  \\
%\langle\partial_1 \bm{u}_3\rangle& \langle\partial_2 \bm{u}_3 \rangle &0 & &\\
\vdots& \vdots &\ddots  & \vdots \\
\mathbb{E}\left[\bar{\mathsf{U}}_t\bar{\mathsf{U}}_1\right] & \mathbb{E}\left[\bar{\mathsf{U}}_t\bar{\mathsf{U}}_2\right]   & \cdots  & \mathbb{E}\left[\bar{\mathsf{U}}_t^2\right] 
\end{bmatrix},\\
\paraB{\Phi}_t&:=
\begin{bmatrix}
0&  &  &  &  \\
\mathbb{E}\left[\partial_1 \para{U}_2\right]& 0 &  &  & \\
\mathbb{E}\left[\partial_1 \para{U}_3\right]& \mathbb{E}\left[\partial_2 \para{U}_3\right] &0 & &\\
\vdots & \vdots & \ddots &  &  \\
\mathbb{E}\left[\partial_1 \para{U}_t\right] &\mathbb{E}\left[\partial_2 \para{U}_t\right] & \cdots & \mathbb{E}\left[\partial_{t-1} \para{U}_t\right]  & 0
\end{bmatrix},
\end{align}
where  $(\mathsf{R}_1,\ldots,\mathsf{R}_{j-1})$ are jointly Gaussian with zero mean, and we denoted
\BE
\mathbb{E}\left[\partial_i \para{U}_j\right] :=\mathbb{E}\left[\partial_i\eta_j(\mathsf{R}_1,\ldots,\mathsf{R}_{j-1})\right],\quad\forall j>1, i\in[j-1].
\EE
\ES
Recall that the covariance $\paraB{\Sigma}_t$ in \eqref{Eqn:AMP_SE_convergence} is expressed using $\bar{\paraB{\Delta}}_t$, while the covariance in \eqref{Eqn:cov_Fan} is expressed using ${\paraB{\Delta}}_t$, where $\bar{\paraB{\Delta}}_t$ and ${\paraB{\Delta}}_t$ denote the covariance of the state evolution random variables $(\bar{\mathsf{U}}_j)_{j\in[t]}$ and $({\mathsf{U}}_j)_{j\in[t]}$ respectively (see \eqref{Eqn:Delta_bar_def} and \eqref{Eqn:Delta_def_Fan}):

 To related these two expressions, recall the definition
\BE
{\mathsf{U}}_j=\bar{\mathsf{U}}_j+\sum_{i=1}^{j-1}\mathbb{E}\left[\partial_i{\mathsf{U}}_j\right]\cdot \mathsf{R}_i,\quad\forall j >1,
\EE
which can be written into a matrix form
\BS\label{Eqn:Delta_deltabar_relation_app}
\begin{align}
\begin{bmatrix}
{\mathsf{U}}_1\\
{\mathsf{U}}_2\\
\vdots\\
{\mathsf{U}}_t
\end{bmatrix}
&=
\begin{bmatrix}
\bar{\mathsf{U}}_1\\
\bar{\mathsf{U}}_2\\
\vdots\\
\bar{\mathsf{U}}_t
\end{bmatrix}
+
\begin{bmatrix}
0&  &  &  &  \\
\mathbb{E}\left[\partial_1 \para{U}_2\right]& 0 &  &  & \\
\mathbb{E}\left[\partial_1 \para{U}_3\right]& \mathbb{E}\left[\partial_2 \para{U}_3\right] &0 & &\\
\vdots & \vdots & \ddots &  &  \\
\mathbb{E}\left[\partial_1 \para{U}_t\right] &\mathbb{E}\left[\partial_2 \para{U}_t\right] & \cdots & \mathbb{E}\left[\partial_{t-1} \para{U}_t\right]  & 0
\end{bmatrix}
\begin{bmatrix}
{\mathsf{R}}_1\\
{\mathsf{R}}_2\\
\vdots\\
{\mathsf{R}}_{t}
\end{bmatrix}\\
&:=
\begin{bmatrix}
\bar{\mathsf{U}}_1\\
\bar{\mathsf{U}}_2\\
\vdots\\
\bar{\mathsf{U}}_t
\end{bmatrix}
+\paraB{\Phi}_t
\begin{bmatrix}
{\mathsf{R}}_1\\
{\mathsf{R}}_2\\
\vdots\\
{\mathsf{R}}_{t}
\end{bmatrix}.
\end{align}
\ES
From Proposition \ref{Pro:orthogonality}, we have the following orthogonality property:
\BE\label{Eqn:orthogonality_app}
\mathbb{E}\left[\bar{\mathsf{U}}_j\mathsf{R}_i\right]=0,\quad\forall i,j\ge1.
\EE
Following \eqref{Eqn:matrix_def_app}, \eqref{Eqn:Delta_deltabar_relation_app} and \eqref{Eqn:orthogonality_app}, the covariance matrices $\paraB{\Delta}_t$ and $\bar{\paraB{\Delta}}_t$ satisfy the following relation:
\begin{equation}\label{Eqn:Delta_deltabar_relation}
\paraB{\Delta}_t=  \bar{\paraB{\Delta}}_t+\paraB{\Phi}_t\paraB{\Sigma}_{t}\paraB{\Phi}_t^\UT.
\end{equation}
Using \eqref{Eqn:Delta_deltabar_relation}, we can rewrite \eqref{Eqn:Cov_SE_new} as
\BS\label{Eqn:Cov_SE_new_call}
\begin{align}
\paraB{\Sigma}_t &=
\mathbb{E}\left[\paraB{P}_t(\mathsf{\Lambda})
\bar{\paraB{\Delta}}_t
\paraB{P}_t(\mathsf{\Lambda})^\UT
\right]\\
&=\mathbb{E}\left[ \left(\sum_{i=1}^t Q_i(\mathsf{\Lambda}){\paraB{\Phi}}_t^{i-1}\right)
\bar{\paraB{\Delta}}_t
\left(\sum_{i=1}^t Q_i(\mathsf{\Lambda}){\paraB{\Phi}}_t^{i-1}\right)^\UT
\right]\\
&=\sum_{i=1}^t\sum_{j=1}^t \mathbb{E}_{\mathsf{\Lambda}\sim\mu}\left[Q_i(\mathsf{\Lambda})Q_j(\mathsf{\Lambda})\right]\cdot \paraB{\Phi}_t^{i-1}\, \bar{\paraB{\Delta}}_t\, (\paraB{\Phi}_t^{j-1})^\UT \\
&:=\sum_{i=1}^t\sum_{j=1}^t \Omega_{i,j} \paraB{\Phi}_t^{i-1}\, \bar{\paraB{\Delta}}_t\, (\paraB{\Phi}_t^{j-1})^\UT \\
 &=\sum_{i=1}^t\sum_{j=1}^t \Omega_{i,j} \paraB{\Phi}_t^{i-1}\, \left(\paraB{\Delta}_t-\paraB{\Phi}_t\paraB{\Sigma}_{t}\paraB{\Phi}_t^\UT\right)\, (\paraB{\Phi}_t^{j-1})^\UT \\
&=\sum_{i=1}^t\sum_{j=1}^t \Omega_{i,j} \paraB{\Phi}_t^{i-1}\, \paraB{\Delta}_t\, (\paraB{\Phi}_t^{j-1})^\UT - \sum_{i=1}^t\sum_{j=1}^t \Omega_{i,j} \paraB{\Phi}_t^{i}\, \paraB{\Sigma}_t\, (\paraB{\Phi}_t^{j})^\UT,
\end{align}
where for convenience we denoted
\BE
\Omega_{i,j}:=\mathbb{E}_{\mathsf{\Lambda}\sim\mu}\left[Q_i(\mathsf{\Lambda})Q_j(\mathsf{\Lambda})\right],\quad \forall i\ge0,j\ge0.
\EE
\ES
Since $Q_0=1$, we can write $\paraB{\Sigma}_t =\mathbb{E}_{\mathsf{\Lambda}\sim\mu}\left[Q_0(\mathsf{\Lambda})Q_0(\mathsf{\Lambda})\right]\cdot \paraB{\Phi}_t^{0}\, \paraB{\Sigma}_t\, (\paraB{\Phi}_t^{0})^\UT$. Moreover, $\mathbb{E}[Q_i(\mathsf{\Lambda})]=0$, $\forall i\ge1$. Hence, 
\BE
\paraB{\Sigma}_t +\sum_{i=1}^t\sum_{j=1}^t\Omega_{i,j} \paraB{\Phi}_t^{i}\, \paraB{\Sigma}_t\, (\paraB{\Phi}_t^{j})^\UT=\sum_{i=0}^t\sum_{j=0}^t \Omega_{i,j} \paraB{\Phi}_t^{i}\, \paraB{\Sigma}_t\, (\paraB{\Phi}_t^{j})^\UT.
\EE
Therefore, we can rewrite \eqref{Eqn:Cov_SE_new_call} as follows
\BE\label{Eqn:Sigma_Delta_app}
\sum_{i=0}^t\sum_{j=0}^t \Omega_{i,j} \paraB{\Phi}_t^{i}\, \paraB{\Sigma}_t\, (\paraB{\Phi}_t^{j})^\UT = \sum_{i=1}^t\sum_{j=1}^t \Omega_{i,j} \paraB{\Phi}_t^{i-1}\, \paraB{\Delta}_t\, (\paraB{\Phi}_t^{j-1})^\UT.
\EE
We treat \eqref{Eqn:Sigma_Delta_app} as an equation of $\paraB{\Sigma}_t$, with $(\paraB{\Phi}_t,\paraB{\Delta}_t)$ being fixed matrices. We remark that this equation has a unique solution. To see this, we use the identity $\text{vec}(\bm{AXB}^\UT)=(\bm{B}\otimes\bm{A})\text{vec}(\bm{X})$. Then, a vectorization of the LHS of \eqref{Eqn:Sigma_Delta_app} reads
\BE
\left(\bm{I}_{t^2}+\Omega_{1,0}\bm{I}_t\otimes\paraB{\Phi}_t+\Omega_{0,1}\paraB{\Phi}_t\otimes\bm{I}_t+\Omega_{2,0}\bm{I}_t\otimes\paraB{\Phi}_t^2+\cdots\right)\text{vec}(\paraB{\Sigma}_t):=\bm{M}\text{vec}(\paraB{\Sigma}_t)
\EE
Note that $\paraB{\Phi}_t$ is strictly lower triangular. Hence, the matrix $\bm{M}$ in the above display is lower triangular with diagonal elements all equal to one. Therefore, $\bm{M}$ is invertible and hence \eqref{Eqn:Sigma_Delta_app} has a unique solution. 

Next, we verify that one solution (and hence the only one) to \eqref{Eqn:Sigma_Delta_app} is 
\BE\label{Eqn:Fan_Sigma_app}
\paraB{\Sigma}_t = \sum_{j=0}^{\infty}\sum_{i=0}^j \kappa_{j+2}\paraB{\Phi}_t^i\paraB{\Delta}_t\left((\paraB{\Phi}_t)^{j-i}\right)^\UT,
\EE
which is precisely \eqref{Eqn:cov_Fan}. This would conclude the equivalence between \eqref{Eqn:Cov_SE_new} and \eqref{Eqn:cov_Fan} which we aim to prove. We first make a change-of-variable $(i,j)\mapsto (m,n)$ in \eqref{Eqn:Fan_Sigma_app} via
\BE
\begin{bmatrix}
i\\
j
\end{bmatrix}
=
\begin{bmatrix}
1 & 0\\
1 & 1
\end{bmatrix}
\begin{bmatrix}
m\\
n
\end{bmatrix}
-\begin{bmatrix}
1\\
2
\end{bmatrix}.
\EE
This map is bijective from $\{(i,j):0\le j< \infty, 0\le i\le j\}$ to $\{(m,n):1\le m <\infty, 1\le n <\infty \}$. We then write \eqref{Eqn:Fan_Sigma_app} as
\BE\label{Eqn:Fan_Sigma_app2}
\paraB{\Sigma}_t = \sum_{m=1}^{\infty}\sum_{n=1}^\infty \kappa_{m+n}\paraB{\Phi}_t^{m-1}\paraB{\Delta}_t\left((\paraB{\Phi}_t)^{n-1}\right)^\UT=\sum_{m=1}^{t}\sum_{n=1}^t \kappa_{m+n}\paraB{\Phi}_t^{m-1}\paraB{\Delta}_t\left((\paraB{\Phi}_t)^{n-1}\right)^\UT,
\EE
where the second step is due to the fact that $\paraB{\Phi}_t$ strictly lower triangular. Towards proving that \eqref{Eqn:Fan_Sigma_app2} is a solution to \eqref{Eqn:Sigma_Delta_app}, we substitute \eqref{Eqn:Fan_Sigma_app} into the LHS of \eqref{Eqn:Sigma_Delta_app}:
\BS\label{Eqn:Delta_RHS_final}
\begin{align}
\sum_{i=0}^t\sum_{j=0}^t \Omega_{i,j} \paraB{\Phi}_t^{i}\, \paraB{\Sigma}_t\, (\paraB{\Phi}_t^{j})^\UT  &= \sum_{i=0}^t\sum_{j=0}^t \Omega_{i,j} \paraB{\Phi}_t^{i}\, \left(\sum_{m=1}^{t}\sum_{n=1}^t \kappa_{m+n}\paraB{\Phi}_t^{m-1}\paraB{\Delta}_t\left((\paraB{\Phi}_t)^{n-1}\right)^\UT\right)\, (\paraB{\Phi}_t^{j})^\UT\\
&=\sum_{i=0}^t\sum_{j=0}^t \sum_{m=1}^{t}\sum_{n=1}^t \Omega_{i,j} \kappa_{m+n} \paraB{\Phi}_t^{i+m-1} \paraB{\Delta}_t\left((\paraB{\Phi}_t)^{n+j-1}\right)^\UT\\
&\explain{(a)}{=}\sum_{(I,J)\in[t]\times [t]} \Bigg(\sum_{(m',n')\in [I]\times [J]}\Omega_{I-m',J-n'}\cdot\kappa_{m'+n'}\Bigg)\cdot \paraB{\Phi}_t^{I-1} \paraB{\Delta}_t\left(\paraB{\Phi}_t^{J-1}\right)^\UT\\
&\explain{(b)}{=}\sum_{(I,J)\in[t]\times [t]}\Omega_{I,J}\cdot \paraB{\Phi}_t^{I-1} \paraB{\Delta}_t\left(\paraB{\Phi}_t^{J-1}\right)^\UT,
\end{align}
\ES
where step (a) is due to a change of variable $(m,i)\mapsto (m',I)$ and $(n,j)\mapsto (n',J)$ via the map:
\begin{align*}
&m'=m,\quad n'=n,\\
&I = m+i,\quad J=n+j,
\end{align*}
and step (b) follows from Lemma \ref{Lem:QQ_kappa} (and the definition $\Omega_{i,j}=\mathbb{E}\left[Q_iQ_j\right] $). We now recognize that \eqref{Eqn:Delta_RHS_final} is identical to the RHS of \eqref{Eqn:Sigma_Delta_app}. This concludes our proof of Proposition \ref{Pro:equivalence}.

\section{{RI-AMP-DF} Algorithm}

\subsection{Reduction of {RI-AMP-DF} to OAMP (Proof of Theorem \ref{The:RI-AMP-DF})}\label{App:RI-AMP-DF-reduction}

\paragraph{Proof of item (1):} The de-biasing matrix \eqref{Eqn:C_t_def} in RI-AMP-DF can be derived following the calculations in Appendix \ref{App:reformulation}. We collect the iterates of {RI-AMP-DF} into the following form (cf.~\eqref{Eqn:RI-AMP-DF-def}):
\BS\label{Eqn:AMP_DF_r_u}
\begin{align}
\begin{bmatrix}
\vr_1 \\
\vr_2 \\
\vdots \\
\vr_t
\end{bmatrix}
&=
\begin{bmatrix}
\bm{W}\vu_1 \\
\bm{W}\vu_2 \\
\vdots \\
\bm{W}\vu_t 
\end{bmatrix}
-
\begin{bmatrix}
\para{c}_{1,1}\bm{I}_N & & & \\
\para{c}_{2,1}\bm{I}_N & \para{b}_{2,2}\bm{I}_N & & \\
\vdots & \vdots & \ddots & \\
\para{c}_{t,1} \bm{I}_N& \para{b}_{t,2}\bm{I}_N & \cdots & \para{b}_{t,t}\bm{I}_N
\end{bmatrix} 
\begin{bmatrix}
\bar{\vu}_1 \\
\bar{\vu}_2 \\
\vdots \\
\bar{\vu}_t 
\end{bmatrix}\\
&=\left(\paraB{I}_t\otimes\bm{W}\right)
\begin{bmatrix}
\vu_1 \\
\vu_2 \\
\vdots \\
\vu_t 
\end{bmatrix}
-(\paraB{C}_t\otimes\bm{I}_N)
\begin{bmatrix}
\bar{\vu}_1 \\
\bar{\vu}_2 \\
\vdots \\
\bar{\vu}_t 
\end{bmatrix}\\
&\explain{(a)}{=}
\left(\paraB{I}_t\otimes\bm{W}\right)\left(
\begin{bmatrix}
\bar{\vu}_1 \\
\bar{\vu}_2 \\
\vdots \\
\bar{\vu}_t 
\end{bmatrix}
+(\hat{\paraB{\Phi}}_t\otimes \bm{I}_N)
\begin{bmatrix}
\vr_1 \\
\vr_2 \\
\vdots \\
\vr_t
\end{bmatrix}
\right)
-(\paraB{C}_t\otimes\bm{I}_N)
\begin{bmatrix}
\bar{\vu}_1 \\
\bar{\vu}_2 \\
\vdots \\
\bar{\vu}_t 
\end{bmatrix}\\
&=(\hat{\paraB{\Phi}}_t\otimes\bm{W})
\begin{bmatrix}
\vr_1 \\
\vr_2 \\
\vdots \\
\vr_t
\end{bmatrix}
+\left(\paraB{I}_t\otimes\bm{W}-\paraB{C}_t\otimes\bm{I}_N\right)
\begin{bmatrix}
\bar{\vu}_1 \\
\bar{\vu}_2 \\
\vdots \\
\bar{\vu}_t 
\end{bmatrix},
\end{align}
\ES
where step (a) is due to the relationship between $\bm{u}_t$ and $\bar{\bm{u}}_t$ (which follows the same calculation as \eqref{Eqn:OAMP_derive_u_ubar} but with $\paraB{D}_t$ replaced by $\hat{\paraB{\Phi}}_t$). Then, following the procedure in Appendix \ref{App:reformulation}, it can be shown that
\BS
\begin{align}
\begin{bmatrix}
\vr_1 \\
\vr_2 \\
\vdots \\
\vr_t
\end{bmatrix} 
&=\left(\bm{I}_{tN}-\hat{\paraB{\Phi}}_t\otimes\bm{W}\right)^{-1}\left(\paraB{I}_t\otimes\bm{W}-\paraB{C}_t\otimes\bm{I}_N\right)
\begin{bmatrix}
\bar{\vu}_1 \\
\bar{\vu}_2 \\
\vdots \\
\bar{\vu}_t 
\end{bmatrix}\\
&:=
\begin{bmatrix}
\hat{G}_{1,1}(\bm{W}) &  &  & \\
\hat{G}_{2,1}(\bm{W}) &\hat{G}_{2,2}(\bm{W})  &  & \\
\vdots & \vdots &\ddots  & \\
\hat{G}_{t,1}(\bm{W}) & \hat{G}_{t,2}(\bm{W}) &\cdots &\hat{G}_{t,t}(\bm{W})
\end{bmatrix}
\begin{bmatrix}
\bar{\vu}_1 \\
\bar{\vu}_2 \\
\vdots \\
\bar{\vu}_t 
\end{bmatrix},
\end{align}
where $(\hat{G}_{i,j})_{1\le i\le t,1\le j\le i}$ is a sequence of polynomials defined as
\BE\label{Eqn:H_matrix_def}
\hat{\paraB{G}}_t(\lambda):=
\begin{bmatrix}
\hat{G}_{1,1}(\lambda)& 0 & \cdots  & 0 \\
\hat{G}_{2,1}(\lambda)& \hat{G}_{2,2}(\lambda)& \cdots  &  0 \\
\vdots & \vdots & \ddots & \\
\hat{G}_{t,1}(\lambda)& \hat{G}_{t,1}(\lambda)& \cdots  &  \hat{G}_{t,t}(\lambda)
\end{bmatrix}
=(\paraB{I}_t-\lambda\hat{\paraB{\Phi}}_t)^{-1}(\lambda\paraB{I}_t-\paraB{C}_t),\quad\forall \lambda\in\mathbb{R}.
\EE
\ES

Following Lemma \ref{Lem:RI_AMP_zero_trace}, we set the de-biasing matrix $\paraB{C}_t$ as the solution to the equation 
\BE\label{Eqn:H_matrix_zero}
\mathbb{E}\left[\hat{\paraB{G}}_t(\mathsf{\Lambda})\right]=\bm{0}_{t\times t},
\EE
where $\hat{\paraB{G}}_t(\mathsf{\Lambda})$ is defined in \eqref{Eqn:H_matrix_def} and the expectation in \eqref{Eqn:H_matrix_zero} is taken w.r.t. $\mathsf{\Lambda}\sim\mu$ independent of $\hat{\paraB{\Phi}}_t$. Enforcing the above condition yields
\BS\label{Eqn:C_equation}
\begin{align}
\paraB{C}_t &= \left(\mathbb{E}\left[(\paraB{I}_t-\mathsf{\Lambda}\hat{\paraB{\Phi}}_t)^{-1}\right]\right)^{-1}\mathbb{E}\left[\mathsf{\Lambda}(\paraB{I}_t-\mathsf{\Lambda}\hat{\paraB{\Phi}}_t)^{-1}\right]\\
&\explain{(a)}{=}\left(\sum_{i=1}^t m_{i-1}\hat{\paraB{\Phi}}_t^{i-1}\right)^{-1}\left(\sum_{i=1}^t m_{i}\hat{\paraB{\Phi}}_t^{i-1}\right),
\end{align}
\ES
where $m_i:=\mathbb{E}[\mathsf{\Lambda}^i]$, $\forall i=0,1,\ldots$ (with the convention $m_0:=1$), and step (a) is due to the expansion
\[
(\paraB{I}_t-\mathsf{\Lambda}\hat{\paraB{\Phi}}_t)^{-1} = \sum_{i=1}^t \mathsf{\Lambda}^{i-1}\hat{\paraB{\Phi}}_t^{i-1}.
\]
(Recall that $\hat{\paraB{\Phi}}_t$ is strictly lower triangular.) From \eqref{Eqn:C_equation}, $\paraB{C}_t $ is clearly a polynomial of $\hat{\paraB{\Phi}}_t$. Further, from \eqref{Eqn:H_matrix_def}, $\hat{\paraB{G}}_t(\lambda)$ is also a polynomial of $\hat{\paraB{\Phi}}_t$. In the following, we apply the trick in the proof of Lemma \ref{Lem:RI_AMP_zero_trace}-(2) to derive a recursive formula for the coefficients in this polynomial representations of $\paraB{C}_t$ and $\hat{\paraB{G}}_t(\lambda)$. 

Let us denote
\BS\label{Eqn:C_G_poly}
\begin{align}
\paraB{C}_t&:=\sum_{i=1}^t \gamma_i\hat{\paraB{\Phi}}_t^{i-1},\\
\hat{\paraB{G}}_t(\lambda) &:=\sum_{i=1}^t H_i(\lambda)\hat{\paraB{\Phi}}_t^{i-1}.
\end{align}
\ES
From \eqref{Eqn:H_matrix_def}, the following equation holds
\BE
(\paraB{I}_t-\lambda\hat{\paraB{\Phi}}_t)\hat{\paraB{G}}_t(\lambda)-\lambda \paraB{I}_t +\paraB{C}_t=\paraB{0}_{t\times t}.
\EE
Using the polynomial representations of $\paraB{C}_t$ and $\hat{\paraB{G}}_t(\lambda)$ in \eqref{Eqn:C_G_poly}, we write the LHS of the above equation as
\BS
\begin{eqnarray}
& &(\paraB{I}_t-\lambda\hat{\paraB{\Phi}}_t)\hat{\paraB{G}}_t(\lambda)-\lambda \paraB{I}_t +\paraB{C}_t\\
&=& \left(\paraB{I}_t-\lambda\hat{\paraB{\Phi}}_t\right)\sum_{i=1}^\infty H_i(\lambda)\hat{\paraB{\Phi}}_t^{i-1}-\lambda \paraB{I}_t +\sum_{i=1}^\infty \gamma_i\hat{\paraB{\Phi}}_t^{i-1}\\
&=&\sum_{i=1}^\infty H_i(\lambda)\hat{\paraB{\Phi}}_t^{i-1}-\sum_{i=1}^\infty\lambda H_i(\lambda)\hat{\paraB{\Phi}}_t^{i}-\lambda \paraB{I}_t +\sum_{i=1}^\infty \gamma_i\hat{\paraB{\Phi}}_t^{i-1}\\
&=&\sum_{i=1}^\infty H_i(\lambda)\hat{\paraB{\Phi}}_t^{i-1}-\sum_{i=2}^\infty\lambda H_{i-1}(\lambda)\hat{\paraB{\Phi}}_t^{i-1}-\lambda \paraB{I}_t +\sum_{i=1}^\infty \gamma_i\hat{\paraB{\Phi}}_t^{i-1}\\
&=& (H_1(\lambda)-\lambda +\gamma_1)\paraB{I}_t+\sum_{i=2}^\infty \left(H_i(\lambda)-\lambda H_{i-1}(\lambda)+\gamma_i\right)\hat{\paraB{\Phi}}_t^{i-1}.
\end{eqnarray}
\ES
Since the above equation is identically zero for any $\hat{\paraB{\Phi}}_t$, we must have
\BS
\begin{align}
H_1(\lambda) &= \lambda-\gamma_1,\\
H_i(\lambda) &= \lambda H_{i-1}(\lambda)-\gamma_i,\quad \forall i\ge2.
\end{align}
\ES
This provides a recursive definition of $(H_i(\lambda))_{i\ge1}$ in terms of the sequence $(\gamma_i)_{i\ge1}$. On the other hand, $(\gamma_i)_{i\ge1}$ are set such that $(H_i(\mathsf{\Lambda}))_{i\ge1}$ have zero means w.r.t. $\mathsf{\Lambda}\in\mu$. Hence,
\BS
\begin{align}
\gamma_1&=\mathbb{E}[\mathsf{\Lambda}],\\
\gamma_i &= \mathbb{E}[\lambda H_{i-1}(\mathsf{\Lambda})],\quad \forall i\ge2.
\end{align}
\ES

To summarize, we have the following recursive representations of $(H_i(\lambda))_{i\ge1}$:
\BE
H_i(\lambda)= \lambda H_{i-1}(\lambda)-\mathbb{E}\left[\lambda H_{i-1}(\lambda)\right],\quad\forall i\ge1,
\EE
where $H_0(\lambda):=1$. The sequences $(\gamma_i)_{i\ge1}$ are given by
\BE
\gamma_i=\mathbb{E}[\lambda H_{i-1}(\mathsf{\Lambda})],\quad \forall i\ge1.
\EE

\paragraph{Proof of item (2):} The state evolution of {RI-AMP-DF} readily follows from the master state evolution result of OAMP. Its proof is similar to that of Theorem \ref{Th:AMP_SE} and omitted.

\subsection{Reduction of GFOM to {RI-AMP-DF}}\label{App:RI-AMP-DF-GFOM}

We first recall the definition of generalized first order method (GFOM) introduced in \cite{montanari2022statistically,celentano2020estimation}.

\begin{definition}[GFOM \cite{montanari2022statistically,celentano2020estimation}]

A generalized first order method (GFOM) generates the iterates $(\bm{x}_t)_{t\ge1}$ via
\BE\label{Eqn:GFOM}
\bm{x}_t = \bm{W} \phi_t(\bm{x}_1,\ldots,\bm{x}_{t-1};\bm{a}) + \psi_t(\bm{x}_1,\ldots,\bm{x}_{t-1};\bm{a}),\quad \forall t\ge1.
\EE
At each iteration, the output $\hat{\bm{x}}_t $ is generated by further applying a post-processing:
\BE
\hat{\bm{x}}_t = h_t\left(\bm{x}_1,\ldots,\bm{x}_t;\bm{a}\right).
\EE
In the above equations, the functions $\phi_t:\mathbb{R}^{t-1}\times\mathbb{R}^k\mapsto\mathbb{R}$, $\psi_t:\mathbb{R}^{t-1}\times\mathbb{R}^k\mapsto\mathbb{R}$ and $h_t:\mathbb{R}^{t}\times\mathbb{R}^k\mapsto\mathbb{R}$, are all continuously-differentiable and Lipschitz. Further, they all act on the $N$ components of their input vectors separately. Moreover, these functions do not depend on the dimension $N$.
\end{definition}

For the purpose of establishing a precise reduction result, we make a minor change to the {RI-AMP-DF} algorithm: we replace the divergence terms and the de-biasing terms by their limiting deterministic equivalents. We call it deterministic {RI-AMP-DF}. 

\begin{definition}[Deterministic {RI-AMP-DF}]
Let $\bm{u}_1=\bar{\bm{u}}_1\in\mathbb{R}^N$ and generate $(\bm{r}_t)_{t\ge1}$ through
\BS\label{Eqn:RI-AMP-DF-deterministic}
\begin{align}
\vr_t &= \mW\vu_t - \left(\para{c}_{t,1}\bar{\bm{u}}_1+\para{c}_{t,2}\bar{\bm{u}}_2+\cdots+\para{c}_{t,t}\bar{\bm{u}}_t\right), \quad\forall t\ge1,\\
\vu_{t+1} &= \eta_{t+1} (\vr_1,\ldots,\vr_t),\\
\bar{\bm{u}}_{t+1} &=\eta_{t+1} (\vr_1,\ldots,\vr_t)-\left(\para{d}_{t,1}\bm{r}_1+\cdots+\para{d}_{t,t} \bm{r}_{t}\right),
\end{align}
\ES
where, for all $t\ge1$, $(\para{c}_{t,i})_{1\le t,1\le i\le t}$ and $(\para{d}_{t,1})_{1\le t,1\le i\le t}$ are deterministic constants that only depends on the function $(\eta_i)_{i\ge1}$ and the limiting spectrum $\mu$, but not on the dimension $N$.
\end{definition}

Following the idea in \cite{montanari2022statistically,celentano2020estimation}, we show in the following proposition that (deterministic) {RI-AMP-DF} can implement any GFOM through a proper change of variables. Since {RI-AMP-DF} is itself a GFOM, this implies that {RI-AMP-DF} and GFOM belong to the same class of algorithms. It is straightforward to show that the same claim applies to the original RI-AMP algorithm.

\begin{proposition}[Deterministic {RI-AMP-DF} can implement any GFOM]\label{Prop:RI-AMP-DF-reduction}
Let $(\bm{x}_t)_{t\ge1}$ be generated by any GFOM. There exists a deterministic {RI-AMP-DF} algorithm, whose iterates are denoted as $(\bm{r}_t)_{t\ge1}$, and a post-processing function $\varphi_t:\mathbb{R}^t\times\mathbb{R}^k\mapsto\mathbb{R}$, such that the following holds:
\[
\bm{x}_t = \varphi_t(\bm{r}_1,\ldots,\bm{r}_t;\bm{a}),\quad\forall t\ge1.
\]
\end{proposition}

\begin{proof}
The proof is essentially identical to \cite[Lemma 4.1]{montanari2022statistically}, and we include it here for completeness. We prove by induction. The claim clearly holds for $t=1$. Suppose it holds up to iteration $t\ge1$, we prove that it also holds for iteration $t+1$. From \eqref{Eqn:GFOM}, the new iterate $\bm{x}_{t+1}$ of GFOM reads:
\BS
\begin{align}
\bm{x}_{t+1}& = \bm{W} \phi_{t+1}(\bm{x}_1,\ldots,\bm{x}_{t};\bm{a}) + \psi_{t+1}(\bm{x}_1,\ldots,\bm{x}_{t};\bm{a})\\
&\explain{(a)}{=} \bm{W} \phi_{t+1}\left(\varphi_1(\bm{r}_1;\bm{a}),\ldots,\varphi_t(\bm{r}_{\le t};\bm{a});\bm{a}\right) + \psi_{t+1}\left(\varphi_1(\bm{r}_1;\bm{a}),\ldots,\varphi_t(\bm{r}_{\le t};\bm{a});\bm{a}\right)\\
&\explain{(b)}{=}\bm{W}  \eta_{t+1}(\bm{r}_{\le t};\bm{a}) + \psi_{t+1}\left(\varphi_1(\bm{r}_1;\bm{a}),\ldots,\varphi_t(\bm{r}_{\le t};\bm{a});\bm{a}\right)\\
&\explain{(c)}{=}\underbrace{\bm{W} \bm{u}_{t+1}-\sum_{i=1}^t  \para{c}_{t,i} \bar{\bm{u}}_i}_{\bm{r}_{t+1}}+\sum_{i=1}^t \para{c}_{t,i}\bar{\bm{u}}_i+ \psi_{t+1}\left(\varphi_1(\bm{r}_1;\bm{a}),\ldots,\varphi_t(\bm{r}_{\le t};\bm{a});\bm{a}\right)\\
&\explain{(d)}{=}\varphi_{t+1}(\bm{r}_1,\ldots,\bm{r}_{t+1};\bm{a}),
\end{align}
\ES
where 
\begin{itemize}
\item Step (a) is from the induction hypothesis;
\item Step (b) is from the definition of the new denoising function $\eta_{t+1}$ for deterministic {RI-AMP-DF};
\item Step (c) is from the definition of the new iterate $\bm{r}_{t+1}$ in {RI-AMP-DF}. Note that $(\para{c}_{t,i})_{1\le t,1\le i\le t}$ are deterministic constants that depend on the denoising functions in previous iterations and the spectrum $\mu$;
\item Step (d) is a definition of the new post-processing function $\varphi_{t+1}$.
\end{itemize}
The proof is now complete.
\end{proof}

\end{document}